\documentclass{article}
\usepackage[utf8]{inputenc}
\usepackage{authblk}
\usepackage{enumerate}
\usepackage[shortlabels]{enumitem}
\usepackage{color}
\usepackage{textcomp}
\usepackage{amsmath,amsfonts,amssymb}
\usepackage{multicol}
\usepackage{amsmath}
\usepackage{amsthm}
\usepackage[english]{babel}
\usepackage{graphicx,float,subcaption,epstopdf}
\usepackage{fancyhdr} 
\usepackage{mathtools}
\usepackage{pdflscape}
\usepackage{float}

\newtheorem{theorem}{Theorem}
\newtheorem{corollary}{Corollary}[theorem]
\newtheorem{lemma}[theorem]{Lemma}
\newtheorem{proposition}[theorem]{Proposition}
\newtheorem{remark}[theorem]{Remark}

\newcommand{\mizh}{{\underline{\pmb{z}}_{h}} }
\newcommand{\misigmah}{{\underline{\pmb{\sigma}}_{h}} }
\newcommand{\misigmac}{{\underline{\pmb{\sigma}}}}
\newcommand{\misigmag}{{\underline{\pmb{\hat{\sigma}}}_{h}}}
\newcommand{\esigma}{\underline{\pmb{e_{\sigma}}}}
\newcommand{\dsigma}{\underline{\pmb{\delta_{\sigma}}}}
\newcommand{\esigmag}{\underline{\pmb{e_{\hat{\sigma}}}}}
\newcommand{\fAsigma}{\pmb{\Lambda}^{\mathcal{A}(\underline{\pmb{\sigma}}_{h})}(\boldsymbol{x})}
\newcommand{\fAsigmawx}{\pmb{\Lambda}^{\mathcal{A}(\underline{\pmb{\sigma}}_{h})}}
\newcommand{\fAdsigma}{\pmb{\Lambda}^{\mathcal{A}\underline{\pmb{\delta_{\sigma}}}}(\boldsymbol{x})}

\newcommand{\fAdsigman}{\pmb{\Lambda}^{\mathcal{A}\underline{\pmb{\delta_{\sigma}}}}}

\newcommand{\fAdsigmawx}{\pmb{\Lambda}^{\mathcal{A}
\underline{\pmb{\delta_{\sigma}}}}}

\newcommand{\fAesigma}{\pmb{\Lambda}^{\mathcal{A}\underline{\pmb{e_{\sigma}}}}(\boldsymbol{x})}

\newcommand{\fAesigman}{\pmb{\Lambda}^{\mathcal{A}\underline{\pmb{e_{\sigma}}}}}

\newcommand{\fAesigmawx}{\pmb{\Lambda}^{\mathcal{A}\underline{\pmb{e_{\sigma}}}}}

\newcommand{\mirh}{{\pmb{r}_{h}}}
\newcommand{\miuh}{{\pmb{u}_{h}}}
\newcommand{\miuc}{\pmb{u}}
\newcommand{\mirg}{\pmb{\hat{r}}_{h}}
\newcommand{\miug}{\pmb{\hat{u}}_{h}}
\newcommand{\eu}{\pmb{e_{u}}}
\newcommand{\du}{\pmb{\delta_{u}}}
\newcommand{\eug}{\pmb{e_{\hat{u}}}}

\newcommand{\piwD}{\pmb{\Pi_{W}}}

\newcommand{\proMD}{\pmb{\mathcal{P}_{M}}}
\newcommand{\proI}{\pmb{\mathcal{I}}}
\newcommand{\piA}{\underline{\pmb{\Pi_{A}}}}
\newcommand{\pivD}{\underline{\pmb{\Pi}}^{\pmb{D}}}
\newcommand{\pcero}{\underline{\pmb{P}}_{0}}

\newcommand{\miroh}{\underline{\pmb{\rho}}_{h}}
\newcommand{\miroc}{\underline{\pmb{\rho}}}
\newcommand{\erho}{\underline{\pmb{e_{\rho}}}}
\newcommand{\drho}{\underline{\pmb{\delta_{\rho}}}}
\newcommand{\frho}{\pmb{\Lambda}^{\underline{\pmb{\rho}}_{h}}(\boldsymbol{x})}

\newcommand{\frhowx}{\pmb{\Lambda}^{\underline{\pmb{\rho}}_{h}}}
\newcommand{\fdrho}{\pmb{\Lambda}^{\underline{\pmb{\delta_{\rho}}}}(\boldsymbol{x})}
\newcommand{\fdrhon}{\pmb{\Lambda}^{\underline{\pmb{\delta_{\rho}}}}}
\newcommand{\fdrhowx}{\pmb{\Lambda}^{\underline{\pmb{\delta_{\rho}}}}}
\newcommand{\ferho}{\pmb{\Lambda}^{\underline{\pmb{e_{\rho}}}}(\boldsymbol{x})}
\newcommand{\ferhon}{\pmb{\Lambda}^{\underline{\pmb{e_{\rho}}}}}
\newcommand{\ferhowx}{\pmb{\Lambda}^{\underline{\pmb{e_{\rho}}}}}

\newcommand{\mipsi}{\underline{\pmb{\psi}}}
\newcommand{\dpsi}{\underline{\pmb{\delta_{\psi}}}}

\newcommand{\dxi}{\underline{\pmb{\delta_{\xi}}}}
\newcommand{\mixi}{\underline{\pmb{\xi}}}

\newcommand{\miphi}{\pmb{\phi}}
\newcommand{\dphi}{\pmb{\delta_{\phi}}}

\newcommand{\miA}{\mathcal{A}}

\newcommand{\miv}{\underline{\pmb{v}}}
\newcommand{\mieta}{\underline{\pmb{\eta}}}
\newcommand{\mivar}{\underline{\pmb{\varsigma}}}
\newcommand{\miw}{\pmb{w}}

\newcommand{\mimu}{\pmb{\mu}}

\newcommand{\mib}{b_{k}}
\newcommand{\migama}{\underline{\pmb{\gamma}}}
\newcommand{\Cero}{\pmb{0}}
\newcommand{\Cerot}{\underline{\pmb{0}}}
\newcommand{\exb}{\boldsymbol{\bar{x}}}
\newcommand{\ex}{\boldsymbol{x}}
\newcommand{\lx}{l(\boldsymbol{x})}
\newcommand{\tx}{\boldsymbol{t}(\boldsymbol{x})}
\newcommand{\Teta}{\Theta(\misigmac,\miroc)}

\newcommand{\trian}{_{\mathcal{T}_{h}}}
\newcommand{\Trian}{{\mathcal{T}_{h}}}
\newcommand{\dtrian}{_{\partial\mathcal{T}_{h}}}
\newcommand{\Dtrian}{{\partial\mathcal{T}_{h}}}

\newcommand{\midiv}{\nabla\hspace{-0.05cm}\cdot }
\newcommand{\normal}{\pmb{n}}
\newcommand{\nT}{\pmb{n}^{t}}
\newcommand{\nore}{\pmb{n}_{e}}
\newcommand{\gamah}{_{\Gamma_{h}}}

\newcommand{\sumE}{ \sum_{e \in \mathcal{E}_{h}^{\partial}}}

\newcommand{\gtildeh}{\pmb{\widetilde{g}}_{h}}
\newcommand{\gtilde}{\pmb{\widetilde{g}}}
\newcommand{\mif}{\pmb{f}}
\newcommand{\mig}{\pmb{g}}
\newcommand{\ipd}[2]{\langle #1,#2 \rangle}
\newcommand{\ipe}[2]{\langle #1,#2 \rangle_{e}}
\newcommand{\ipg}[2]{\langle #1,#2 \rangle_{\Gamma_{h}}}

\newcommand{\ipt}[2]{( #1,#2 )_{\mathcal{T}_{h}}}
\newcommand{\ipb}[2]{\langle #1,#2 \rangle_{\partial \mathcal{T}_{h}}}
\newcommand{\ipf}[2]{\langle #1,#2 \rangle_{\partial \mathcal{T}_{h}\setminus\Gamma_{h}}}
\newcommand{\ipk}[2]{( #1,#2 )_{K}}

\newcommand{\ipdk}[2]{\langle #1,#2 \rangle_{\partial K}}

\newcommand{\miteta}{\pmb{\theta}}
\newcommand{\dn}[1]{\partial_{\pmb{n}}(#1 \pmb{n})}
\newcommand{\psn}{\partial_{\pmb{n}}}
\newcommand{\miT}[2]{\mathbb{T}^{#1}_{#2}}

\newcommand{\maxE}{\max_{e \in \mathcal{E}_{h}^{\partial}}}

\newcommand{\norm}[3]{\| #1 \|_{#2}^{#3}}

\newcommand{\norme}[2]{\| #1 \|_{e}^{#2}}
\newcommand{\normel}[2]{\| #1 \|_{e,l^{-1}}^{#2}}
\newcommand{\normell}[2]{\| #1 \|_{e,l}^{#2}}

\newcommand{\normea}[2]{\| #1 \|_{e,\tau}^{#2}}
\newcommand{\normKe}[2]{\| #1 \|_{K^{e}}^{#2}}

\newcommand{\normk}[2]{\| #1 \|_{K}^{#2}}
\newcommand{\normKex}[2]{\| #1 \|_{K_{ext}^{e}}^{#2}}
\newcommand{\normKext}[2]{\| #1 \|_{K_{ext}^{e},(h^{\perp})^{2}}^{#2}}
\newcommand{\normDhc}[2]{\| #1 \|_{\Omega_{h}^{c},(h^{\perp})^{2}}^{#2}}
\newcommand{\normDh}[2]{\| #1 \|_{\Omega_{h}}^{#2}}

\newcommand{\nDhA}[2]{\| #1 \|_{\Omega_{h},\miA}^{#2}}
\newcommand{\normO}[2]{\| #1 \|_{\Omega}^{#2}}

\newcommand{\normgh}[2]{\| #1 \|_{\Gamma_{h}}^{#2}}
\newcommand{\normgl}[2]{\| #1 \|_{\Gamma_{h},l^{-1}}^{#2}}
\newcommand{\normgpl}[2]{\| #1 \|_{\Gamma_{h},l}^{#2}}

\newcommand{\normgll}[2]{\| #1 \|_{\Gamma_{h},l^{2}}^{#2}}
\newcommand{\normgml}[2]{\| #1 \|_{\Gamma_{h},l^{-2}}^{#2}}
\newcommand{\normga}[2]{\| #1 \|_{\Gamma_{h},\tau}^{#2}}
\newcommand{\normgal}[2]{\| #1 \|_{\Gamma_{h},\tau^{2},l^{2}}^{#2}}
\newcommand{\normGl}[3]{\| #1 \|_{\Gamma_{h},l^{#2}}^{#3}}
\newcommand{\normGh}[3]{\| #1 \|_{\Gamma_{h},(h^{\bot})^{#2}}^{#3}}

\newcommand{\Cext}{C_{ext}^{e}}

\newcommand{\Cinv}{C_{inv}^{e}}
\newcommand{\Ctr}{C_{tr}^{e}}

\newcommand{\he}{h_{e}^{\bot}}
\newcommand{\He}{H_{e}^{\bot}}
\newcommand{\alfa}{\tau}
\newcommand{\CA}{\dfrac{1}{2\mu}}
\newcommand{\miR}{R}
\newcommand{\mir}{r_{e}}

\newcommand{\miVh}[1]{\underline{\pmb{V}}_{h}#1}

\newcommand{\miV}[1]{\underline{\pmb{V}}(#1)}

\newcommand{\miWh}[1]{\pmb{W}_{h}#1}

\newcommand{\miAh}[1]{\underline{\pmb{A}}_{h}#1}
\newcommand{\miAc}[1]{\underline{\pmb{A}}(#1)}
\newcommand{\miAt}{{\underline{\pmb{\widetilde{A}}}}(K)}
\newcommand{\miAs}[2]{\pmb{A}_{#1}(#2)}
\newcommand{\miAS}[1]{\underline{\pmb{AS}}(#1)}
\newcommand{\miB}[1]{\underline{\pmb{B}}(#1)}
\newcommand{\miBh}{\underline{\mathcal{B}}_{h}}

\newcommand{\miMh}[1]{\pmb{M}_{h}#1}
\newcommand{\miM}[1]{\pmb{M}(#1)}
\newcommand{\miMD}[1]{\pmb{M}(#1)}

\newcommand{\Eh}[1]{\mathcal{E}_{h}^{#1}}
\newcommand{\Ltv}[1]{\pmb{L}^{2}(#1)}
\newcommand{\Ltt}[1]{\underline{\pmb{L}}^{2}(#1)}
\newcommand{\Hv}[2]{\pmb{H}^{#1}(#2)}
\newcommand{\Ht}[2]{\underline{\pmb{H}}^{#1}(#2)}
\newcommand{\Hdiv}[1]{\underline{\pmb{H}}({\rm div};#1)}
\newcommand{\Hg}[1]{\pmb{H}^{1/2}(#1)}
\newcommand{\Pkt}[1]{\underline{\mathbf{P}}_{k}(#1)}

\newcommand{\Pt}[2]{\underline{\mathbf{P}}_{#1}(#2)}
\newcommand{\Pv}[2]{\mathbf{P}_{#1}(#2)}

\newcommand{\Pcerok}[1]{\underline{\mathbf{P}}_{0}(#1)}
\newcommand{\Punok}[1]{\underline{\mathbf{P}}_{1}(#1)}
\newcommand{\Pkv}[1]{\mathbf{P}_{k}(#1)}
\newcommand{\Pke}[1]{P_{k}(#1)}

\newcommand{\vertiii}[1]{{\left\vert\kern-0.25ex\left\vert\kern-0.25ex\left\vert #1 
    \right\vert\kern-0.25ex\right\vert\kern-0.25ex\right\vert}}
\newcommand{\bignorm}{\vertiii{(\esigma,\eu-\eug,\gtilde-\gtildeh)}}

\title{A high order unfitted  hybridizable discontinuous Galerkin method for linear elasticity}

\author[1]{Juan Manuel C\'ardenas}
\author[2,3]{Manuel Solano}
\affil[1]{Department of Mathematics, Simon Fraser University, Vancouver, Canada.{\it email: jcardena@sfu.ca}}
\affil[2]{Departamento de Ingenier\'ia Matem\'atica, Facultad de Ciencias F\'isicas y Matem\'aticas, Universidad de Concepci\'on, Concepci\'on, Chile.}
\affil[3]{Centro de Investigaci\'on 
en Ingenier\'ia Matem\'atica (CI$^2$MA), Universidad de Concepci\'on, Concepci\'on, Chile.{\it email: msolano@ing-mat.udec.cl}}

\date{\today}

\begin{document}

\maketitle

\begin{abstract}
This work analyzes a high order hybridizable discontinuous Galerkin (HDG) method for the linear elasticity problem in a domain not necessarily polyhedral. The domain is approximated by a polyhedral computational domain where the HDG solution can be computed. The introduction of the rotation as one of the unknowns allows us to use the gradient of the displacements to obtain an explicit representation of the boundary data in the computational domain. The boundary data is transferred from the true boundary to the computational boundary by line integrals, where the integrand depends on the Cauchy stress tensor and the rotation. Under closeness assumptions between the computational and true boundaries, the scheme is shown to be well-posed and optimal error estimates are provided even in the nearly incompressible. Numerical experiments in two-dimensions are presented.
\end{abstract}



\section{Introduction}\label{sec;introduction}


This work introduces and analyses a hybridizable discontinuous Galerkin (HDG) method for the isotropic linear elasticity problem 
\begin{subequations}\label{Equation1a}
\begin{align}
\miA{}\misigmac{}-\underline{\pmb{\epsilon}}(\miuc{})& =0\hspace{0.5cm}  \text{in} \hspace{0.2cm}  \Omega,\\
\midiv{}\misigmac{} & =\mif{} \hspace{0.5cm}  \text{in} \hspace{0.2cm}  \Omega,\\
\miuc{} & =\mig \hspace{0.5cm} \text{on} \hspace{0.2cm}  \Gamma,
\end{align}
\end{subequations}
where $\Omega \in \mathbb{R}^{n}, n \in \lbrace 2,3\rbrace$ is a bounded domain, not necessarily polyhedral, with boundary $\Gamma$  compact, Lipschitz and piecewise $\mathcal{C}^2$. Here, $\miuc$ is the unknown displacement, $\underline{\pmb{\epsilon}}(\miuc):=\frac{1}{2}(\nabla\miuc+\nabla^{t}\miuc)$ is the strain tensor, $\misigmac$ is the Cauchy stress tensor, $\mif \in \Ltv{\Omega}$ is a source term, $\mig \in \Hg{\Gamma}$ is a given boundary data,  $\miA^{-1}$ is the elasticity tensor determined by the Hooke's Law, that is, for a tensor $\mixi$,
\begin{align}
\miA^{-1}(\mixi)= 2\mu \mixi + \lambda {\rm tr}(\mixi)\underline{\pmb{I}} 
\quad \textrm{and} \quad
\miA(\mixi)=\frac{1}{2\mu}\mixi - \frac{\lambda}{2\mu(n\lambda + 2\mu)}{\rm tr}(\mixi)\underline{\pmb{I}}, \label{Atensor}
\end{align}
where, $\underline{\pmb{I}}$ denotes the identity tensor, $\displaystyle {\rm tr}(\mixi):=\sum_{i=1}^{n}\mixi_{ii}$, $\lambda$ and $\mu$ are the Lam\'e constant such that 
$ \mu := \dfrac{E}{2(1+\nu)}$  and $ \lambda := \dfrac{E \nu}{(1+\nu)(1-2\nu)}$,
with $E$ the Young's modulus and $\nu$ the Poisson ratio.

One of the first HDG schemes for the linear elasticity problem has been proposed in \cite{Soon} for the formulation \eqref{Equation1a} in polyhedral domains, where the symmetry of the stress tensor is imposed exactly. There, numerical experiments showed the performance of the method. Later, the authors in \cite{FuCoSt2015} theoretically proved optimal order of convergence for the displacement and suboptimal for the other variables. They also provided numerical experiments showing that their error estimates are sharp. In addition, also for the formulation in \eqref{Equation1a}, \cite{QiShSh2018} devised a new HDG scheme by considering polynomials of degree $k$, $k+1$ and $k$ for the approximation of the stress, displacement and trace of the displacements, respectively. The symmetry of the stress tensor is also imposed on the discrete spaces and the numerical trace of the stress is suitable defined in order to be able to use the standard $L^2$-projection in the error analysis, instead of the HDG-projection \cite{diffusion}. Recently, two new theoretical tools have been developed to devise and analyze HDG method for elasticity problems. One of them is the M-decomposition for devising  superconvergent HDG methods  \cite{CoFu2017} and the other is related to the construction of a tailored projection that provides way of analyzing a family of HDG methods \cite{DuSA2020}. Recently, the authors in \cite{Fu} proposed a novel gradient-robust and locking-free HDG method.

On the other hand, the work in \cite{Ke} analyzed an HDG method where the symmetry is imposed weakly by introducing the rotation $\miroc{}(\miuc{})= (\nabla \miuc - \nabla^{t} \miuc)/2$ as an additional unknown. In this setting, \eqref{Equation1a} can be written as 
\begin{subequations}\label{Equation2a}
\begin{align}
\miA{}\misigmac{}-\nabla \miuc{}+ \miroc{}& =0\hspace{0.5cm}  \text{in $\Omega$,}\label{Equation2_1}\\
\midiv{}\misigmac{} & =\mif{} \hspace{0.5cm} \text{in $\Omega$,}\\
\miuc{} & =\mig \hspace{0.5cm} \text{on $\Gamma$}.
\end{align}
\end{subequations}

During the last decade, HDG methods to handle curved domains via extension from polyhedral subdomain have been developed for a variety of problems such as Darcy \cite{CQS,CS}, Stokes \cite{Vargas} and Oseen \cite{Vargas2} equations, convection-diffusion problem \cite{CoSol2} and elliptic interface problems \cite{Patrick}. All these contributions are based on approximating $\Omega$ by a polyhedral subdomain $\Omega_h$ and transferring the boundary condition from $\Gamma$ to the computational boundary $\Gamma_{h}$ by line integration of the extrapolated discrete approximation of the gradient. That is why this methodology is often called {\it transferring technique}. The key feature of the partial differential equation (PDE) that makes possible to use this approach, is to have the gradient of $\boldsymbol{u}$ as one of the unknowns. Therefore, the formulation \eqref{Equation2a} is well suited to this transferring technique since $\nabla \boldsymbol{u}$ is written as $\miA{}\misigmac{}+ \miroc{}$.

In this work, we analyze the resulting HDG scheme for \eqref{Equation2a} posed on a curved domain $\Omega$, combined with the aforementioned transferring technique to approximate the boundary data on the computational domain. Even though this type of unfitted HDG method has been analyzed before \cite{CQS,Vargas,Vargas2}, it has not been studied for elasticity problems, where the main challenges that we address in this manuscript rely in three aspects. The first one is the presence of two of the unknowns, the Cauchy stress tensor and the rotation, in the line integrals used to transfer the boundary data. In our previous work only one of the unknowns is being integrating along the transferring segments. The second aspect is the task of obtaining estimates independent of the value of $\lambda$. Finally, as it is usual in unfitted methods, closeness conditions between $\Gamma_h$ and $\Gamma$ must be assumed in order to to have a wellposed and optimal scheme. In the context of elasticity problems, ideally one would like those conditions to be independent of $\lambda$. As we will see in Section \ref{ExUn}, we were able to get rid of $\lambda$ in all the closeness assumption, except in one of them which requires the distance between $\Gamma_h$ and $\Gamma$ to satisfy $dist(\Gamma_h,\Gamma)h^{1/2}\lesssim \lambda^{-1}$. However, the numerical experiments reported in Section \ref{ChapNE} suggest that this restriction could be relaxed since for the nearly incompressible examples the method still performs optimally.

To fix ideas, let $\ex \in \Gamma_{h}$ and associate to it a point $ \exb \in \Gamma$. The precise specification of $ \exb$ will be introduced in Section \ref{secTP}. We set  $\lx := |\exb-\ex|$ and $\tx$ the unit tangent vector of the segment joining $\ex$ and $\exb$, then integrating \eqref{Equation2_1} between $\ex$ and $\exb$ we deduce  the identity
\begin{align}\label{expansion}
\miuc(\ex)=\mig(\exb)-\int_{0}^{\lx}\hspace{-0.2cm}(\miA\misigmac + \miroc)(\ex + s \tx)\tx ds,
\end{align}
since $\miuc(\exb)=\mig(\exb)$. Defining $\gtilde(\ex):=\miuc(\ex)$, we obtain the following expression for the boundary data $\gtilde$ in $\Gamma_{h}$:
\begin{align}
\gtilde(\ex) = \mig(\exb) - \int_{0}^{\lx}\hspace{-0.2cm}(\miA\misigmac + \miroc)(\ex + s \tx) \tx ds .\label{defg}
\end{align}
Then, we solve the following problem in the computational subdomain $\Omega_{h}$:
\begin{subequations}\label{Equation3a}
\begin{align}
\miA{}\misigmac{}-\nabla \miuc{}+ \miroc{}& =0\hspace{0.5cm}  \text{in $\Omega_{h}$,}\\
\midiv{}\misigmac & =\mif \hspace{0.5cm} \text{in $\Omega_{h}$,}\\
\miuc{} & =\gtilde \hspace{0.5cm} \text{on $\Gamma_{h}:=\partial \Omega_{h}$.}
\end{align}
\end{subequations}

As we mentioned above, the idea of transferring the boundary data from $\Gamma$ to $\Gamma_{h}$ by integrating $\nabla\miuc$ along a segment, was originally introduced and analysed in a one-dimensional diffusion problem \cite{Cockburn9}, where an HDG method was employed. Later, \cite{CS} generalized the method to the two-dimensional case and developed the implementation tools. In the same direction, \cite{CoSol2} numerically showed that the method performs optimaly in convection-diffusion equations. Also, this technique was used in an exterior diffusion problem in a curved domain \cite{Cockburn10}. There, the authors coupled the boundary element method to an HDG scheme and experimentally showed that the order of convergence of the resulting method is optimal. Then \cite{CQS} analysed the method proposed in \cite{CS} using the projections-based error analysis of HDG methods \cite{diffusion}. In fact, \cite{CQS} provided the theoretical framework to analyze this type of techniques of transferring the boundary data. Recently, this data transferring technique have been generalized to other type of boundary conditions \cite{CaSo2021} and transmission conditions over dissimilar and non-matching grids \cite{10.1093/imanum/drab059}.

Let us briefly comment on previous work related to unfitted methods for elasticity problems. In the context of discontinuous Galerkin methods, one of the first unfitted methods for linear and nonlinear elasticity was introduced by \cite{Lew2} based on the immerse DG method proposed in \cite{Lew1}. The approximations functions are piecewise polynomials of degree one and allowed to be discontinuous in those elements intersecting the interface. The resulting method is optimal and does not suffer from boundary locking.

In the context of HDG method, recently an unfitted eXtended HDG(X-HDG) method has been introduced for the elasticity problem \cite{han2021interfaceboundaryunfitted}. In the X-HDG method \cite{xHDG1,xHDG2,xHDG3}, the domain is also immerse in a background mesh and piecewise polynomials functions are employed in the discrete spaces. The local discrete spaces associated to those elements cut by the interface are enriched in order to correctly capture the behavior of the solution across the interface. The authors in \cite{han2021interfaceboundaryunfitted} considered polynomials of degree $k$ and $k-1$ to approximate the displacements and stress, respectively; and polynomials of degree $k$ for the numerical traces. Optimal $L^2$-error estimates were proved.

During the last five years, a close related method to our technique has been introduced: the shifted boundary method (SBM) \cite{Scovazzi1,Scovazzi2} and recently extended to problems in solid mechanics \cite{Scovazzi3}. The main idea of SBM is to properly construct the boundary data in the computational boundary that is ``shifted" from the true boundary. That construction is based on a Taylor expansion of the solution near the boundary and on a Nitsche approach to imposed weakly the boundary data. In our transferring technique, the data in the computational boundary is also properly constructed but using the PDE instead of a Taylor series. Actually, by taking a closer look to \eqref{expansion}, somehow we are expanding $\miuc{}$ around a point $ \exb \in \Gamma$ and the functions that have been integrated are actually differential operators acting on $\miuc{}$. In contrast with a Taylor expansion, we also observe that the expansion in \eqref{expansion} is exact (no residual term), since it comes from the PDE. At the discrete level there will be a residual,  because $\misigmac$ and $\miroc$ will be approximated.

The rest of this manuscript is organized as follows. Section \ref{method} describes the proposed unfitted HDG method, whereas in Section \ref{ChapAnEr} wellposedness of the scheme is shown. The  error estimates are stated in Section \ref{sec:error} and the corresponding proofs are provided in Section \ref{proofs}. Numerical experiments validating the theory are presented in Section \ref{ChapNE}.

\section{The method}\label{method}
\subsection{Computational domain}\label{secNC}
Given $h>0$, we denote by $\Omega_{h}$ a polyhedral domain contained in $\Omega$ with boundary $\Gamma_h=\partial \Omega_h$. We also denote by $\mathcal{T}_{h}$ a triangulation of $\Omega_h$ made of simplices $K$ of diameter $h_{K}$ and  outward unit normal  $\pmb{n}_K$.  When there is no confusion, we just write $\pmb{n}$ instead of $\pmb{n}_K$.  By simplicity we assume that the family of triangulations $\{\mathcal{T}_h\}_{h>0}$ does not have hanging nodes and is uniformly shape regular, i.e., there exists a constant $\gamma$ such that $h_{K}\leq \gamma \varrho_{K}$, for all $K\in \mathcal{T}_h$ and $h>0$. Here, $\varrho_{K}$ is the radius of the biggest ball included in $K$ and the maximum of the diameters $h_{K}$ is at most $h$.

We call $e$ an interior face if there are two elements $K^{+}$ and $K^{-}$ in $\mathcal{T}_{h}$ such that $e = \partial K^{+}\cap \partial K^{-}.$ Similarly, $e$ is a boundary face if there is an element $K \in \mathcal{T}_{h}$ such that $e = \partial K \cap \Gamma_{h}.$ Let $\mathcal{E}_{h}^{0}$ be the set of inferior faces of $\mathcal{T}_{h}, \mathcal{E}_{h}^{\partial}$ the set of faces at the boundary and $\mathcal{E}_{h}:=\mathcal{E}_{h}^{0}\cup \mathcal{E}_{h}^{\partial}$. Given an face $e\in \mathcal{E}_h^\partial$, $\pmb{n}_e$ denotes its unit normal vector pointing outwards $\Omega_h$, also denoted by just  $\pmb{n}$ when there is no confusion.

\subsection{Transferring paths and extrapolation regions}\label{secTP}
As we mentioned in the introduction, given a point $\ex \in \Gamma_{h}$ we need to specify a point $\exb \in \Gamma$ in order to transfer the boundary data from $\boldsymbol{\bar{x}}$ to $\boldsymbol{x}$ according to \eqref{defg}. In principle, $\exb$ could be any point of $\Gamma$ {\it{close enough}} to $\ex.$ The segment joining $\ex$ and $\exb$ will be referred as {\it{transferring path}} associated to $\ex$. We denote by $\lx$ and $\tx$ the length and unit tangent vector, respectively, of the transferring path associated to $\ex$ (see Figure \ref{fig:Tpa_x} for an illustration). From a practical point of view, this transferring path is required to satisfy three conditions: (1) $\exb$ and $\ex$ must be as close as possible, (2) two transferring paths must not intersect each other before terminating at $\Gamma$ and (3) a transferring path must not intersect the interior of the computational domain $\Omega_{h}$. The authors in \cite{CS}, for the two dimensional case, proposed an algorithm to construct a family of transferring paths satisfying the above mentioned condition. The construction in three dimensions can be done using the same ideas. In practice we only need to compute the transferring paths of the quadrature points of all boundary edges (see Figure \ref{fig:Tpa_b}). Another possibility is to consider $\boldsymbol{\bar{x}}$ as the closest point projection of  $\boldsymbol{x}$ onto $\Gamma$, as long it is unique. Actually, the analysis that we present in this work, is independent of how the transferring paths are constructed, if the hypothesis regarding the closedness between $\Gamma_h$ and $\Gamma$ are satisfied, namely the set of Assumptions C presented in Section \ref{ExUn}.

Now, let us introduce the notation associated to the set $\Omega_{h}^{c}:= \Omega \setminus \overline{\Omega_{h}}$. For a face $e \in\mathcal{E}_{h}^{\partial},$ we denote by $K^{e}$ the only element of $\mathcal{T}_{h}$ having $e$ as a face. We define  
$\widetilde{K}_{ext}^{e}:= \lbrace \ex +s \tx: 0 \leq s \leq \boldsymbol{l(x)} , \boldsymbol{x} \in e \rbrace$.
In Figure \ref{fig:Kext} we observe an example of a region $\widetilde{K}_{ext}^{e}$. The subscript {\it{ext}} in $\widetilde{K}_{ext}^{e}$ is introduced to indicate that in those regions the discrete solution will be {\it{extrapolated}} as follows.
 Let $p$ a polynomial defined on $K^{e}.$ The extrapolation of $p$ from $K^{e}$ to $\widetilde{K}_{ext}^{e}$, denoted by $E_{h}(p)$, is defined by $E_{h}(p)(y):= p|_{K^{e}}(y), \forall y \in \widetilde{K}_{ext}^{e}$. To simplify notation, from now on we will just write $p(y)$ instead of $E_{h}(p)(y)$ for $y \in \widetilde{K}_{ext}^{e}$. The same notation will be used for the extrapolation of tensor- and vector-valued polynomial functions.

\begin{center}
\begin{figure}[ht!]
    \begin{subfigure}[t]{0.5\textwidth}
         \centering
        \includegraphics[scale=0.5]{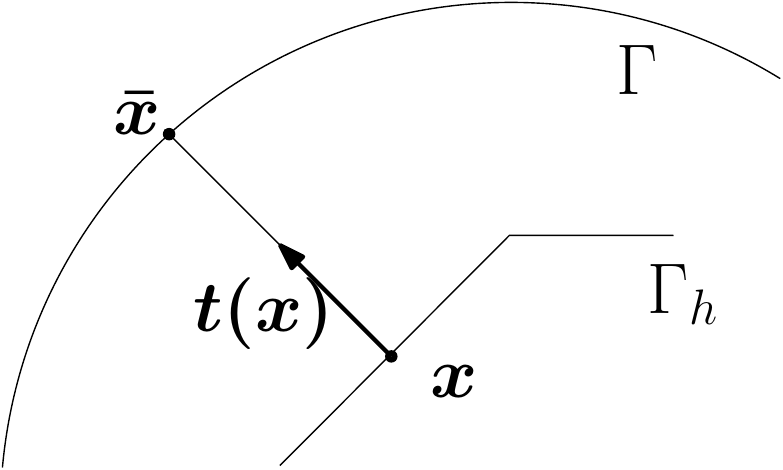}
        \caption{Transferring path associated to $\ex.$}
        \label{fig:Tpa_x}
    \end{subfigure}\hfill
    \begin{subfigure}[t]{0.5\textwidth}
         \centering
        \includegraphics[width=2.4cm,height=3.4cm]{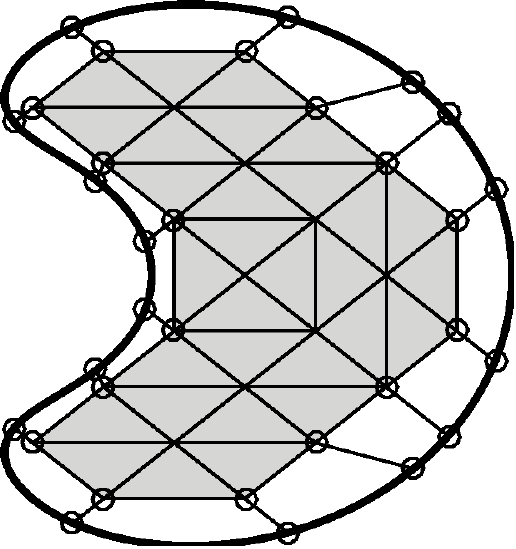}
        \caption{Transferring paths associated to the boundary vertices.}
        \label{fig:Tpa_a}
    \end{subfigure}\\
     \begin{subfigure}[t]{0.5\textwidth}
     \centering
        \includegraphics[width=2.5cm,height=3.5cm]{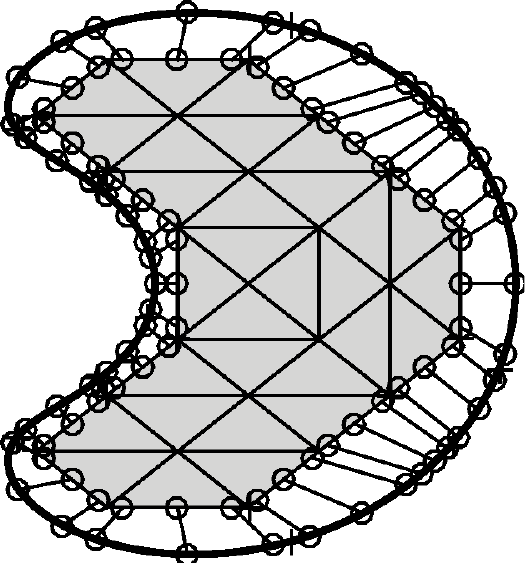}
        \caption{Transferring paths associated to the boundary quadrature points.}
     \label{fig:Tpa_b}
    \end{subfigure}\hfill
\begin{subfigure}[t]{0.5\textwidth}
     \centering
\includegraphics[scale=0.45]{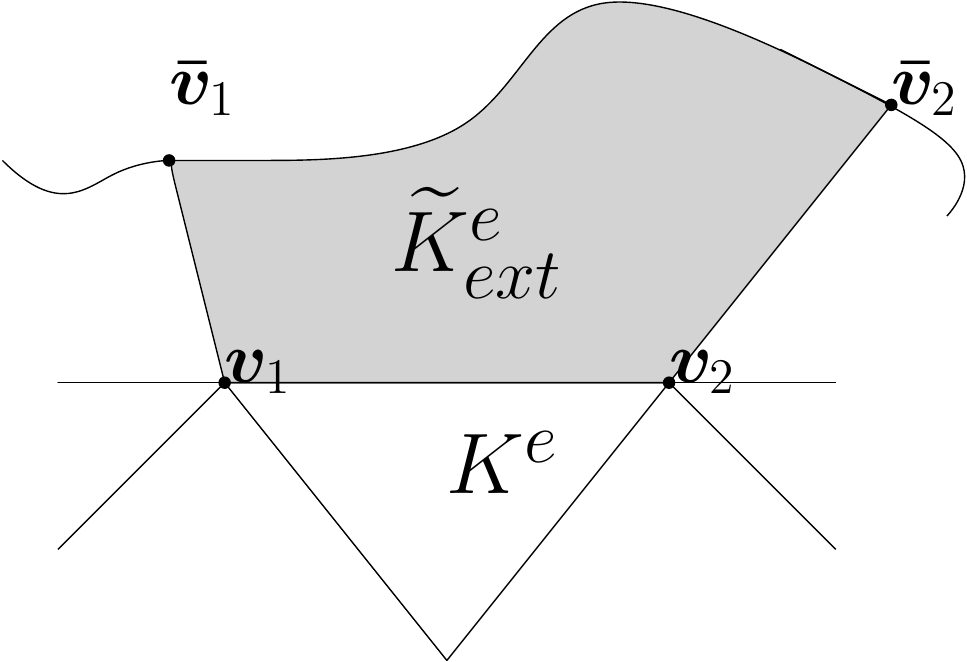}
\caption{Example of $\widetilde{K}_{ext}^{e}$.}\label{fig:Kext}
\end{subfigure}    
    \caption{Examples of the transferring paths and extrapolation regions.}
\end{figure}\label{fig:Tpa}
\end{center}

\subsection{Additional notation}\label{sec:noration}

For tensor-, vector- and scalar-valued functions we use the symbols $\mieta=(\eta_{ij})_{i,j=1}^n$, $\pmb{\eta}=(\eta_{i})_{i=1}^n$ and $\eta$, respectively. The superscript $t$ in a  vector or a tensor refers to its transpose. We define 
$
(\mieta,\mivar)\trian:=\sum\limits_{i,j=1}^{n}(\eta_{ij},\varsigma_{ij})\trian$,
$(\pmb{\eta},\pmb{\varsigma})\trian:=\sum\limits_{i=1}^{n}(\eta_{i},\varsigma_{i})\trian$ and
$(\eta,\varsigma)\trian:=\sum\limits_{K \in \mathcal{T}_{h}}(\eta,\varsigma)_{K}
$,
where $(\eta,\varsigma)_{K}$ denotes the standard $L^2$-inner product on $K$. Similarly, we write 
$
\ipd{\pmb{\eta}}{\pmb{\varsigma}}\dtrian:=\sum\limits_{i=1}^{n}\ipd{\eta_{i}}{\varsigma_{i}}\dtrian $
and $
\ipd{\eta}{\varsigma}\dtrian:=\sum\limits_{K \in \Trian}\ipdk{\eta}{\varsigma}
$,
where $\ipd{\eta}{\varsigma}_{\partial K}$ is the $L^2(\partial K)$-inner product. We also use the standard notation for Sobolev spaces and the associated norms and seminorms. We define $\norm{\eta}{D,w}{}:= \norm{\sqrt{w}\eta}{L^{2}(D)}{}$ and simply write $ \norm{\eta}{D}{}$ when $w=1$. On the set of boundary faces we consider the norm
$\displaystyle\norm{\eta}{h}{}:=\left( \sum_{K \in \Trian } h_{K}\norm{\eta}{\partial K}{2} \right)^{1/2}$. 

On the other hand,  in the two-dimensional case, the operator $\nabla \times$ applied to $\mieta$ and $\pmb{\nu}$ is defined as:
\[ \nabla \times\mieta:=
\begin{pmatrix}
-\partial_{y}\eta_{11} +\partial_{x} \eta_{12}\\
-\partial_{y}\eta_{21} +\partial_{x} \eta_{22}
\end{pmatrix}
\hspace{0.5cm}\textrm{and} \hspace{0.5cm}
\nabla \times\pmb{\nu}:=
\begin{pmatrix}
-\partial_{y}\nu_{1}   & \partial_{x} \nu_{1}\\
-\partial_{y}\nu_{2} & \partial_{x} \nu_{2}
\end{pmatrix},
\]
respectively; whereas in three dimensions,
\[ \nabla \times \mieta:=
\begin{pmatrix}
\nabla\times(\eta_{11},\tau_{12},\eta_{13})\\
\nabla\times(\eta_{21},\tau_{22},\eta_{23})\\
\nabla\times(\eta_{31},\tau_{32},\eta_{33})
\end{pmatrix}.
\]
Finally, $\nabla$ will denote the usual gradient or broken-gradient, depending on the context. Similarly for $\midiv$.


\subsection{The HDG method}\label{ChapHDG}
First of all, we recall the discrete spaces of the HDG method proposed in \cite{Ke} for simplices. Let $K\in \Trian$. We define $\Pke{K}$ as the set of polynomials of degree at most $k$ over $K$, $\Pkv{K}:=[\Pke{K}]^{n}$, $\Pkt{K}:=[\Pke{K}]^{n\times n}$ and $\miAc{K}:=[\miAs{i,j}{K}]^{n\times n}$ such that 
\[ \miAs{i,j}{K}=
	\begin{cases} 
       \Pke{K} & \text{if} \hspace{0.3cm} i < j,  \\
       0 &  \text{if}  \hspace{0.3cm}i = j, \\
       -\Pke{K} &\text{if} \hspace{0.3cm} i > j.
    \end{cases}
\]
We notice that $\miAc{K}$ is contained in $\miAS{K}:=\lbrace \mieta \in \Ltt{K}:\mieta+\mieta^{T}=\underline{\pmb{0}}\rbrace$. In addition, the polynomial space $\miB{K}$ associated to bubble functions is defined as follows. In the two dimensional case $ \miB{K}:= \nabla\times((\nabla\times\miAc{K})\mib{})$, where $\mib:=\prod_{e \subset \partial K} \eta_{e}$ and $\eta_e$ is the barycentric coordinate associated to the edge $e$ of $K$. 
In the three dimensional case, 
$ \miB{K}:= \nabla\times((\nabla\times\miAc{K})\underline{\pmb{b}}_{k})$,
with
\[ \underline{\pmb{b}}_{k}:= \sum_{e \subset \partial K} \Big[\prod_{e' \subset \partial K \setminus \lbrace e \rbrace}\eta_{e'}\Big]\nabla\eta_{e}\otimes\nabla\eta_{e}.
\]

For an element $K$ we define the local space $
\miV{K} :=\Pkt{K}+\miB{K}$. We notice that
$
\miV{K} = \Pkt{K} +  \nabla \times ((\nabla \times \miAc{K}))\mib{})
	    = \Pkt{K}\oplus \nabla \times ((\nabla \times \miAt )\mib{})$,
where $\miAt = \miAc{K} \cap \widetilde{\underline{\textbf{P}}}_{k}(K)$ and $\widetilde{\underline{\textbf{P}}}_{k}(K)$ is the set of polynomials of degree exactly $k$.

\begin{remark}\label{rmk:B} Observe that any function $\miv$ lying in the space $\miBh :=\lbrace \mieta \in \Ltt{\Omega_{h}}: \mieta |_{K} \in \miB{K}, K \in \Trian \rbrace$ is such that $\midiv \miv |_{K} = 0 \,, \forall K \in \Trian$ and $\miv\normal |_{e}=0 \,, \forall e \in \mathcal{E}_{h}.$
\end{remark}

The method seeks an approximation  $(\misigmah,\miuh,\miroh,\miug)$ of the exact solution $(\misigmac,\miuc,\miroc,\miuc|_{\mathcal{E}_{h}})$ in the finite-dimensional space $\miVh\times\miWh\times\miAh\times\miMh \subset \Ltt{\Omega_{h}}\times\Ltv{\Omega_{h}}\times\miAS{\Omega_{h}}\times\Ltv{\Eh{}}$ given by 
\begin{subequations}\label{spaces}
\begin{align}
\miVh&=\lbrace\miv \in \Ltt{\Trian}: \miv|_{K}\in \miV{K},\hspace{0.3cm}\forall K \in \Trian \rbrace, \\
\miWh&=\lbrace\miw \in \Ltv{\Trian}: \miw|_{K}\in\Pkv{K},\hspace{0.3cm} \forall K \in \Trian \rbrace, \\
\miAh&=\lbrace\mieta \in \Ltt{\Trian}: \mieta|_{K}\in \miAc{K},\hspace{0.3cm} \forall K \in \Trian \rbrace, \\
\miMh&=\lbrace\mimu \in \Ltv{\Eh{}}: \mimu|_{e}\in \mathbf{P}_{k}(e),\hspace{0.3cm} \forall e \in \Eh{} \rbrace.
\end{align}
\end{subequations}
\begin{subequations}\label{sq}
The approximation $(\misigmah,\miuh ,\miroh ,\miug )$ is the solution of the following linear system:
\begin{align}
\ipt{\miA\misigmah}{\miv}+\ipt{\miuh}{\midiv \miv}+\ipt{\miroh}{\miv}-\ipb{\miug}{\miv \normal}&=0,\label{seqa}\\
\ipt{\misigmah}{\nabla \miw}-\ipb{\misigmag \normal}{\miw} & = -\ipt{\mif}{\miw},\label{seqb}\\
\ipt{\misigmah}{\mieta} & = 0,\label{seqc}\\
\ipf{\misigmag\normal}{\mimu} & = 0,\label{seqd}\\
\ipg{\miug}{\mimu} &= \ipg{\gtildeh}{\mimu},\label{seqe}
\end{align}
for all $(\miv,\miw,\mieta,\mimu) \in \miVh \times \miWh \times \miAh \times \miMh ,$ where 
\begin{align}
\misigmag \normal&= \misigmah\normal-\tau(\miuh-\miug)\hspace{1cm} \text{on} \hspace{0.5cm} \Dtrian , \label{seqf} \\
\gtildeh(\ex)&=\pmb{g}(\exb)-\int_{0}^{\lx}\hspace{-0.2cm}(\miA \misigmah + \miroh)(\ex +s\tx)\tx ds \label{seqg}
\end{align}
\end{subequations}
and $\tau$ is a positive stabilization parameter defined on $\Dtrian$, that we assume constant on each face. We observe  that \eqref{seqg} is a discrete version of \eqref{defg}, where we recall that $\miA \misigmah $ and $\miroh$ are understood as the local extrapolation as mentioned at the end of Section \ref{secTP}.\\

%
%


\section{Wellposedness}\label{ChapAnEr}
\subsection{Preliminaries}\label{secpre}
As we will see through this section, the analysis of the method requires several technicalities and most of the estimates involve a large number of terms. In order to keep the proofs as clean as possible, we assume the vector $\tx$ of the transferring paths associated to $\ex \in e, e \in \Eh{\partial},$ to be normal to $e$, i.e., $\tx =\normal_{e}$. In the general case where $\tx$ is not necessarily equal to $\normal_{e}$, as it usually happens, terms of the type  $\max\limits_{\ex \in e}\tx\cdot\normal_{e}$ and $\max\limits_{\ex \in e}(\tx\cdot\normal_{e})^{-1}$ would appear in the estimates and the results that we will prove hold also true if $\tx\cdot\normal_{e}$ is close enough to one as shown by \cite{Vargas} in the context of Stokes flows. We emphasize that this assumption is only made to simplify the analysis and we consider that it is not crucial to explain the theory. Moreover, in the numerical experiments we will consider examples where transferring paths are not normal to the boundary edges and will see that results are optimal. Following the discussion in Section \ref{secTP}, for each $e \in \Eh{\partial}$, let us define  
\[ K_{ext}^{e}:= \lbrace \boldsymbol{x}+s \pmb{n}_{e}: 0 \leq s \leq \boldsymbol{l(x)} , \boldsymbol{x} \in e \rbrace. \] 
In addition, we define auxiliary constants that will be used in the analysis. Let $K^{e}$ the element with face $e$. We denote by $\he$ the distance between the vertex, opposite to $e$, and the plane determined by $e$ and set $\He:=\max\limits_{\boldsymbol{x}\in e} l(\boldsymbol{x})$. In order to quantify how close is $\Gamma_h$ from $\Gamma$, related to the meshsize, we define the ratios  
\begin{align}\label{ratios}
\mir := \He / \he \qquad \textrm{and} \quad \miR{}:=\maxE\mir.
\end{align}

We consider the norms: 
\begin{align*}
\normgl{\eta}{} & := \left\lbrace \sum_{e \in \Eh{\partial}} \norm{\eta}{e,l^{-1}}{2}\right\rbrace^{1/2},
\hspace{0.1cm}\normKext{\eta}{}{}  := \left\lbrace \sum_{e \in \mathcal{E}^{\partial}_{h}} (h_{e}^{\bot})^{2}\norm{\eta}{K_{ext}^{e}}{2}\right\rbrace^{1/2},
\end{align*}
where, according to the notation in Section \ref{secTP}, $\norm{\eta}{e,l^{-1}}{}=\norm{l^{-1/2}\eta}{e}{}.$
Finally, we define the constants:
\begin{align}\label{Cext}
\Cext := \frac{1}{\sqrt{\mir}} \sup_{\pmb{\eta} \in \miV{K^{e}}\normal_{e}\setminus\lbrace \boldsymbol{0} \rbrace} 
\frac{\normKex{\pmb{\eta} }{}}{\normKe{\pmb{\eta} }{}}, \hspace{0.5cm}
\Cinv := \he  \sup_{\pmb{\eta}  \in \miV{K^{e}}\normal\setminus\lbrace \boldsymbol{0} \rbrace} 
\frac{\normKe{\partial_{\nore }\pmb{\eta} }{}}{\normKe{\pmb{\eta} }{}},
\end{align}
which are independent of $h$, but depend on the polynomial degree $k$ as shown in Lemma A.2 of \cite{CQS}. 

On the other hand, it is useful to state some estimates from previous work that will be used in the proofs. For any face $e \in \Eh{\partial}$, any point $\boldsymbol{x}$ lying on $e$ and any smooth enough function tensor $\miv$  given in $K^{e}_{ext}$, we define the auxiliary function 
\begin{align}\label{Lambda}
\pmb{\Lambda}^{\miv}(\ex):=\frac{1}{\lx}\int_{0}^{\lx}[\miv(\ex + s\nore)-\miv(\ex)]\nore ds.
\end{align}
If $\miv \in \Pkt{K^{e}}$, Lemma 5.2 of \cite{CQS} applied to each of row of  $\miv$ implies: 
\begin{align}\label{auxlem1} 
\norm{\pmb{\Lambda}^{\miv}}{e,l}{} & \leq \frac{1}{\sqrt{3}}\mir{}^{3/2}\Cext{}\Cinv{}\normKe{\miv}{}.
\end{align}
Moreover, for a symmetric and positive definite tensor $\mathcal{D}$ and $\miv \in \Pkt{K^{e}}$, it holds
\begin{align}\label{auxlemD} 
\norm{\mathcal{D}\pmb{\Lambda}^{\miv}}{e,l}{} = \norm{\pmb{\Lambda}^{\mathcal{D}\miv}}{e,l}{}  & \leq \frac{1}{\sqrt{3}}\mir{}^{3/2}\Cext{}\Cinv{}\normKe{\mathcal{D}\miv}{}.
\end{align}

Let $K \in \Trian$ having a face $e$. For $p \in \Pke{K}$ we recall the discrete trace inequality 
\begin{eqnarray}\label{auxlemTr}
 \norm{p}{L^{2}(e)}{} \leq \Ctr{}h_{e}^{-1/2}\norm{p}{L^{2}(K)}{},
 \end{eqnarray}
where $\Ctr >0$ is independent of $h$.

In addition, given a symmetric and positive definite tensor $\mathcal{D}$ and a region $D\subset \mathbb{R}^n$, we define the the norm $\norm{\misigmac}{D,\mathcal{D}}{}:=(\mathcal{D} \misigmac,\misigmac)_D^{1/2}$. Similarly, for $D\subset \mathbb{R}^{n-1}$, we let $\norm{\misigmac}{D,\mathcal{D}}{}:=\langle\mathcal{D} \misigmac,\misigmac\rangle_D^{1/2}$.

\begin{lemma}\label{auxlemA3} Let $\misigmac \in \Ltt{D}$. It holds
$$ \norm{\misigmac}{D,\miA^2}{}\leq \Big(\frac{1}{2\mu}\Big)^{1/2}\norm{\misigmac}{D,\miA}{} \leq
\frac{1}{2\mu}\norm{\misigmac}{D}{} 
\leq \Big(\frac{1}{2\mu}\Big)^{3/2}\norm{\misigmac}{D,\miA^{-1}}{}.$$
Moreover, if $\misigmac$ is antisymmetric, these inequalities become equalities.
\end{lemma}
\begin{proof} For the first inequality, we consider the definition of $\miA$ in \eqref{Atensor} to deduce 
\begin{align*}
\norm{\misigmac}{D,\miA^2}{2} 
=& \frac{1}{2\mu}\norm{\misigmac}{D,\miA}{2}
-\frac{\lambda}{2\mu(n\lambda+2\mu)^2} \|{\rm tr}(\misigmac)\|_D^2
 \leq \frac{1}{2\mu}\norm{\misigmac}{D,\miA}{2}.
\end{align*}
In addition, by the definition of $\miA^{-1}$ (cf. \eqref{Atensor}), we have
\[
\norm{\misigmac}{D,\miA^{-1}}{2} = \left( \miA^{-1}\misigmac,\misigmac\right)_D
= 2\mu \norm{\misigmac}{D}{2} + \lambda \norm{{\rm tr}(\misigmac)}{D}{2}\geq  2\mu \norm{\misigmac}{D}{2} 
\]
and the third inequality follows. Also, by the first and the Cauchy-Schwarz inequalities,  $$\norm{\misigmac}{D,\miA}{2} =(\miA \misigmac,\misigmac)_D\leq \norm{\misigmac}{D,\miA^2}{} \norm{\misigmac}{D}{} \leq (2\mu)^{-1/2}\norm{\misigmac}{D,\miA}{}\norm{\misigmac}{D}{},$$ which implies the second inequality. Finally,  if $\misigmac$ is antisymmetric, then ${\rm tr}(\misigmac )=0$ and the result follows from previous expressions.
\end{proof}

The following lemma will be useful for obtaining estimates that do not depend on $\lambda$. As we will notice, it bounds the $L^2$ 
-norm of the stress tensor and the constants accompanying the norms on the right hand side are independent of $\lambda$. From now on $C$ will denote a positive constant independent of $h$ and $\lambda$. Moreover, to avoid proliferation of unimportant constants we will write $a\lesssim b$ whenever there exists a constant $C>0$ independent of $h$ and $\lambda$ such that $a\leq C b$.

\begin{lemma}\label{lemma-condition}
Let $(\mizh,\mirh,\mirg)\in \miVh \times \miWh \times\miMh$ such that ${\rm tr}(\mizh)\in L_0^2(\Omega_h)$ ($L^2$-functions with zero mean) and
\begin{equation}\label{condition}
\ipt{\mizh}{\nabla \miw}+\ipb{\tau(\mirh-\mirg)}{\miw-\boldsymbol{P}_{k-1}\miw}  = 0,
\end{equation}
for all $\miw \in \boldsymbol{H}^1_0(\Omega_h)$, where $\boldsymbol{P}
_{k-1}$ is the $L^2$ projection over space $\mathbf{P}_{k-1}(\mathcal{T}_h)$. There exists $M_0>0$, independent of $h$ and the Lam\'e parameters such that
\begin{equation}\label{result-aux}
\|\mizh\|_{\Omega_h}\leq M_0 \left(
h^{1/2}\tau^{1/2}\|\tau^{1/2}(\mirh-\mirg)\|_{\partial \mathcal{T}_h} 
+(2\mu)^{1/2} \norm{\mizh}{\Omega_h,\miA}{}\right)
\end{equation}
\end{lemma}
\begin{proof}
Since ${\rm tr}(\mizh)\in L_0^2(\Omega_h)$, we know that \cite{GR}
\begin{equation}\label{infsup-tr}
\|{\rm tr}(\mizh)\|_{\Omega_h}
\lesssim \sup_{\boldsymbol{w}\in\boldsymbol{H}_0^1(\Omega_h) \setminus \left\{\boldsymbol{0}\right\}}\dfrac{(\nabla\cdot\boldsymbol{w},{\rm tr}(\mizh))_{\Omega_h}}{\|\boldsymbol{w}\|_{\boldsymbol{H}^1(\Omega_h)}}.
\end{equation}
Now, for $\boldsymbol{w}\in\boldsymbol{H}_0^1(\Omega_h)$, we can deduce that
\begin{align*}
(\nabla\cdot\boldsymbol{w},{\rm tr}(\mizh))_{\Omega_h} =&
n (\mizh,\nabla \boldsymbol{w})_{\Omega_h} -n (\mizh,(\nabla \boldsymbol{w})^D)_{\Omega_h}, 
\end{align*}
where $(\nabla \boldsymbol{w})^D:=\nabla \boldsymbol{w}-n^{-1}{\rm tr}(\nabla \boldsymbol{w})\underline{\boldsymbol{I}}$ denotes the deviatoric tensor  associated to $\nabla \boldsymbol{w}$ and we recall that $n$ is the dimension. Then, since $ (\mizh,(\nabla \boldsymbol{w})^D)_{\Omega_h}= ((\mizh)^D,\nabla \boldsymbol{w})_{\Omega_h}$ by \eqref{condition} and the Cauchy-Schwarz inequality,  we have that
\begin{align*}
(\nabla\cdot\boldsymbol{w},{\rm tr}(\mizh))_{\Omega_h} \leq&
n \|\tau(\mirh-\mirg)\|_{\partial \mathcal{T}_h}\|\miw-\boldsymbol{P}_{k-1}\miw\|_{\partial \mathcal{T}_h}
+n \|(\mizh)^D\|_{\Omega_h} \|\nabla \boldsymbol{w}\|_{\Omega_h}\\ 
\lesssim&
n \tau^{1/2}\|\tau^{1/2}(\mirh-\mirg)\|_{\partial \mathcal{T}_h} h^{1/2} |\miw|_{1,\Omega_h}
+n \|(\mizh)^D\|_{\Omega_h} \|\nabla \boldsymbol{w}\|_{\Omega_h},
\end{align*}
where in the last inequality we have used the approximation properties of the $L^2$-projection and the fact that $\miw \in \boldsymbol{H}^1(\Omega_h)$.
On the other hand, we notice that
\begin{equation}\label{dev-aux}
\norm{\mizh}{\Omega_h,\miA}{2} = \dfrac{1}{2\mu} \|(\mizh)^D\|^2_{\Omega_h}+\dfrac{1}{n\lambda+2\mu} \|{\rm tr}(\mizh)\|^2_{\Omega_h}\geq \dfrac{1}{2\mu} \|(\mizh)^D\|^2_{\Omega_h}.
\end{equation}
Therefore,
\begin{align*}
(\nabla\cdot\boldsymbol{w},{\rm tr}(\mizh))_{\Omega_h} \lesssim&\left(
n \tau^{1/2}\|\tau^{1/2}(\mirh-\mirg)\|_{\partial \mathcal{T}_h} h^{1/2}
+(2\mu)^{1/2} n \norm{\mizh}{\Omega_h,\miA}{}\right) |\miw|_{1,\Omega_h},
\end{align*}
and from \eqref{infsup-tr} we have that
\begin{align}\label{tr-aux}
\|{\rm tr}(\mizh)\|_{\Omega_h}
\lesssim 
n \tau^{1/2}\|\tau^{1/2}(\mirh-\mirg)\|_{\partial \mathcal{T}_h} h^{1/2}
+(2\mu)^{1/2} n \norm{\mizh}{\Omega_h,\miA}{}. 
\end{align}
Finally, we recall that $\|\mizh\|_{\Omega_h}^2 = \|\mizh^D\|_{\Omega_h}^2+n^{-1}\|{\rm tr}(\mizh)\|_{\Omega_h}$. Then, \eqref{result-aux} follows from  \eqref{dev-aux} and \eqref{tr-aux}.

\end{proof}


\subsection{Wellposedness}\label{ExUn}

It is convenient for notation purposes, to define
\begin{align}\label{M}
M(h,\tau,\mu):= 2M_0^2\max\{h\tau,2\mu\},    
\end{align}
where we recall that $M_0$ is the constant appearing in previous lemma.

We proceed now to show existence and uniqueness of the HDG Scheme \eqref{sq} under the following conditions that quantify how close $\Gamma$ and $\Gamma_h$ must be in order to ensure  that the scheme is wellposed.

 \noindent\underline{Assumptions C}. For every face $e \in \Eh{\partial}$, we assume 
\setlength{\columnsep}{-1cm}
\begin{multicols}{2}
\begin{enumerate}[({C}.1)]
\setcounter{enumi}{-1}
    	\item $\mir\leq C$,\label{C0}
  
		\item$  \dfrac{10}{12}\mir^{3}(\Cext\Cinv)^2 M(h,\tau,\mu)\left(\dfrac{1}{\mu^2}+C_\rho\right) \leq  \dfrac{1}{8}$,\label{C1}
		 \item$\tau \mir\he M(h,\tau,\mu)\leq  \dfrac{1}{5}$,\label{C3}
		 \item$  \gamma(\Ctr)^{2} \mir\leq   \dfrac{ M(h,\tau,\mu)C_\rho }{4 (C^c)^2} $,  \label{C4}   
		 \item $\dfrac{10}{4}(\Ctr)^{2} \mir M(h,\tau,\mu) \left(\left (1+\dfrac{1}{2\mu}\right)^2+\gamma C_\rho \right) \leq   \dfrac{1}{8}$,  \label{C2} 
		
\item $\displaystyle \dfrac{(n\lambda+2\mu)^2}{n|\Omega_h|} \gamma r_e h_e^2 \leq \dfrac{ M(h,\tau,\mu) }{16}$,\label{C5}
    \end{enumerate}
\end{multicols}
\noindent where $\gamma$ is the shape regularity constant of the family of triangulations, $C_\rho:=2 (2\mu)^{-2} \left(C^0+(1+C^0)C^c\right)^2$, $C^0$ and $C^c$ are positive constants independent of the discretization parameters and $\lambda$ that will be specified in the proof of Lemma \ref{eulem3}. The other constants have been introduced in Section \ref{secpre} and do not depend on $h$ or the Lamé parameters .

Let us briefly comment on these assumptions. First of all, if the $\Gamma_h$ exactly fits $\Gamma$, as it happen for instance when the domain $\Omega$ is a polyhedron, all these assumption trivially hold true since $r_e=0$. On the other hand, if for example  the computational boundary $\Gamma_h$ interpolates $\Gamma$ by a piecewise linear function, the distance $dist(\Gamma_h,\Gamma)$ is of order $h^2$ and hence $r_e$ is of order $h$. Then, all the assumptions are satisfied for $h$ small enough. On the other hand, in the case of \textit{immerse-type} methods, where the domain $\Omega$ is immersed in a background mesh and the computational domain $\Omega_h$ is the union of all the elements in the background triangulation lying completely inside $\Omega$, the distance $dist(\Gamma_h,\Gamma)$ is of order $h$ and, as a consequence, $r_e$ is of order one. In that case, assumptions \ref{C3} and \ref{C5} always hold for a sufficiently small value of $h$, whereas the remaining assumption are satisfied when the ratio $r_e$ is small enough.
Regarding the nearly incompressible case, we observe that Assumptions \ref{C0}-\ref{C2} are independent of $\lambda$, whereas \ref{C5} roughly says that the distance between $\Gamma_h$ and $\Gamma$ should satisfiy $dist(\Gamma_h,\Gamma)h^{1/2}\lesssim \lambda^{-1}$. This condition arises from the fact that at the discrete level, $\langle \widetilde{\boldsymbol{g}}_h,\boldsymbol{n}\rangle_{\Gamma_h}$ is not zero when $\boldsymbol{g}$ vanishes (see the proof of Lemma \ref{eulem1}), as it happens in the continuous case.

Now, let us proceed to show wellposedness of the scheme. First of all, we notice that \eqref{sq} is a square linear system; hence, it is enough to show that if $\mif{} = \Cero$ and $\mig{} = \Cero$, the solution of  \eqref{sq} is the trivial solution. 

The following identity establishes a relation between an energy-type norm and a term $\miT{}{}$ arising from the approximation of the boundary data.

\begin{lemma}\label{eulem1} If $\mif{} = \Cero$ and $\mig{} = \Cero$, then the approximation in \eqref{sq} satisfies
\begin{subequations}\label{energies}
\begin{equation}\label{energy1}
\nDhA{\misigmah}{2}+\|\miuh-\miug\|_{\partial\mathcal{T}_h,\tau}^2= \miT{}{},
\end{equation}
where $\miT{}{}:=\ipg{\gtildeh}{\misigmag\normal}$. Moreover, 
\begin{align}\label{energy2}
\|\misigmah\|_{\Omega_h}^2+\|\miuh-\miug\|_{\partial\mathcal{T}_h,\tau}^2\leq&
M(h,\tau,\mu)\miT{}{}+ 
 \dfrac{(n\lambda+2\mu)^2}{n|\Omega_h|} \gamma \max_{e\in \mathcal{E}_h^I}\left(r_e h_e^2\right) \|\gtildeh\|_{\Gamma_h,l^{-1}}^2.
\end{align}
\end{subequations}
\end{lemma}
\begin{proof}
{\bf Step 1 (energy argument):} We take $\miv=\misigmah ,\miw=\miuh , \mieta=\miroh$ and $\mimu=\left\lbrace \begin{array}{lllll}
 \miug &,&  \text{on} & \partial \mathcal{T}_h \setminus \Gamma_h \\
\misigmag\normal&,& \text{on}& \Gamma_h
\end{array}\right.$, in equations \eqref{seqa}-\eqref{seqe}. Then
\begin{subequations}
\begin{align}
\ipt{\miA\misigmah}{\misigmah}+\ipt{\miuh}{\midiv \misigmah}+\ipt{\miroh}{\misigmah}-\ipb{\miug}{\misigmah \normal}&=0,\label{peq1}\\
\ipt{\misigmah}{\nabla \miuh}-\ipb{\misigmag \normal}{\miuh} & = 0,\label{peq2}\\
\ipt{\misigmah}{\miroh} & = 0,\label{peq3}\\
\ipf{\misigmag\normal}{\miug} & = 0,\label{peq4}\\
\ipg{\miug}{\misigmag\normal} &= \ipg{\gtildeh}{\misigmag\normal}. \label{peq5}
\end{align}
\end{subequations}
Integrating by parts \eqref{peq2}, 
\begin{align}
- \ipt{\midiv \misigmah}{\miuh} + \ipb{\misigmah \normal-\misigmag \normal}{\miuh} = 0. \label{peq6}
\end{align}
Adding  \eqref{peq1} and \eqref{peq6}, and using  \eqref{peq3}, we have
\begin{align*}
\ipt{\miA\misigmah}{\misigmah}+\ipb{\misigmah\normal-\misigmag\normal}{\miuh}-\ipb{\miug}{\misigmah\normal}& =0 .
\end{align*}
Next, note that $\ipb{\miug}{\misigmah\normal}=\ipb{\miug}{\misigmah\normal-\misigmag\normal}+\ipb{\miug}{\misigmag\normal}$ and $\ipb{\miug}{\misigmag\normal}=\ipg{\miug}{\misigmag\normal}$ by  \eqref{peq5}. Then,  by \eqref{seqf} and the above expression, we have
\begin{align*}
\ipt{\miA\misigmah}{\misigmah}+\ipb{\tau(\miuh-\miug)}{(\miuh-\miug)} &= \ipg{\miug}{\misigmag\normal}.
\end{align*} 
Thus, \eqref{energy1} from the definition of $\miT{}{}$.\\

{\bf Step 2 (orthogonal decomposition):}
Now, in order to prove \eqref{energy2} we will make use of Lemma \ref{lemma-condition}. To that end, we first decompose $\misigmah:=\misigmac^0+\alpha \underline{\boldsymbol{I}}$, where ${\rm tr}(\misigmac^0)\in  L^2_0(\Omega_h)$ and
\[
\alpha :=\dfrac{1}{n|\Omega_h|}\int_{\Omega_h} {\rm tr}(\misigmah).
\]
We also have that
\begin{align}\label{decomp}
\|\misigmah\|_{\Omega_h}^2 = \|\misigmac^0\|_{\Omega_h}^2 + \alpha^2n|\Omega_h|.
\end{align}
Let us now verify that $(\misigmac^0,\miuh,\miug)$ satisfies \eqref{condition}. Let $\miw \in \boldsymbol{H}^1_0(\Omega_h)$ and note that 
\begin{align*}
\ipt{\misigmah}{\nabla \miw} &= \ipt{\misigmac^0}{\nabla \miw} +\alpha\ipt{\underline{\boldsymbol{I}}}{\nabla  \miw}.
\end{align*}
But $\ipt{\underline{\boldsymbol{I}}}{\nabla\cdot \miw}=\ipt{1}{\nabla\cdot \miw} =\langle\boldsymbol{w}\cdot \boldsymbol{n},1\rangle_{\Gamma_h}=0$. Therefore
$
\ipt{\misigmah}{\nabla \miw} = \ipt{\misigmac^0}{\nabla \miw}
$.

By integration by parts, introducing the $L^2$-projection $\boldsymbol{P}_{k-1}$ and integrating by parts again, we obtain that
\begin{align*}
\ipt{\misigmac^0}{\nabla \miw} =&\ipt{\misigmah}{\nabla \miw}\\
=&
-\ipt{\nabla \cdot \misigmah}{\miw}+\ipb{\misigmah\normal}{\miw}
=
-\ipt{\nabla \cdot \misigmah}{\boldsymbol{P}_{k-1}\miw}+\ipb{\misigmah\normal}{\miw}\\
=&\ipt{\misigmah}{\nabla \boldsymbol{P}_{k-1}\miw} +\ipb{\misigmah\normal}{\miw-\boldsymbol{P}_{k-1}\miw}
=\ipb{\misigmag \normal}{\boldsymbol{P}_{k-1}\miw} +\ipb{\misigmah\normal}{\miw-\boldsymbol{P}_{k-1}\miw},
\end{align*} 
where in the last step we have used \eqref{seqb} and the fact that $\boldsymbol{f}=\boldsymbol{0}$. Therefore, since $\ipb{\misigmag\normal}{\miw}=0$, we have that
\begin{align*}
\ipt{\misigmac^0}{\nabla \miw} 
=&\ipb{\misigmag \normal-\misigmah\normal}{\miw-\boldsymbol{P}_{k-1}\miw}
\end{align*} 
which, together with \eqref{seqf}, implies that  $(\misigmac^0,\miuh,\miug)$ satisfies \eqref{condition}. Thus, by Lemma \ref{lemma-condition}, we obtain that
\begin{equation}
\|\misigmac^0\|_{\Omega_h}\leq M_0 \left(
h^{1/2}\tau^{1/2}\|\miuh-\miug\|_{\partial\mathcal{T}_h,\tau}
+(2\mu)^{1/2} \norm{\misigmac^0}{\Omega_h,\miA}{}\right).
\end{equation}
In addition, after some algebraic calculations, it is possible to obtain that
\[
\norm{\misigmac^0}{\Omega_h,\miA}{} = \norm{\misigmah}{\Omega_h,\miA}{}- \dfrac{\alpha^2n|\Omega_h|}{n\lambda+2\mu}\leq  \norm{\misigmah}{\Omega_h,\miA}{}.
\]
Hence
\begin{equation}
\|\misigmac^0\|_{\Omega_h}\leq M_0 \left(
h^{1/2}\tau^{1/2}\|\miuh-\miug\|_{\partial\mathcal{T}_h,\tau}
+(2\mu)^{1/2} \norm{\misigmah}{\Omega_h,\miA}{}\right),
\end{equation}
which, together with \eqref{decomp}, implies that
\begin{equation}\label{lab:aux1}
\|\misigmah\|_{\Omega_h}^2\leq 2M_0^2 \left(
h\tau\|\miuh-\miug\|_{\partial\mathcal{T}_h,\tau}^2
+(2\mu) \norm{\misigmah}{\Omega_h,\miA}{2}\right)
+ |\alpha|^2 n|\Omega_h|.
\end{equation}

{\bf Step 3 (characterization of $\alpha$):}
We notice that
$
{\rm tr }(\miA\misigmah) =(n\lambda+2\mu)^{-1}{\rm tr }(\misigmah)
$.
Then,
\[
\alpha = \dfrac{1}{n|\Omega_h|} \int_{\Omega_h} {\rm tr }(\misigmah)
=  \dfrac{(n\lambda+2\mu)}{n|\Omega_h|} \int_{\Omega_h} {\rm tr }(\miA \misigmah).
\]
Taking $\miv=\boldsymbol{I}$ in \eqref{seqa} and making use of \eqref{seqe}, we have that
$
\ipt{{\rm tr }(\miA\misigmah)}{1}=\ipg{\gtildeh}{\boldsymbol{n}}
$. 
Therefore
\begin{align}\label{aux:alpha}
\alpha 
=  \dfrac{(n\lambda+2\mu)}{n|\Omega_h|} \ipg{\gtildeh}{\boldsymbol{n}}.
\end{align}

{\bf Step 4 (bound for $\ipg{\gtildeh}{\boldsymbol{n}}$):}
On the other hand, let $e\in\mathcal{E}_h^\partial$. By the Cauchy-Schwarz inequality and \eqref{ratios} we have that,
\begin{align*}
\langle \gtildeh,\boldsymbol{n}\rangle_e\leq& h_e^{1/2}\|\gtildeh\|_e
\leq h_e^{1/2}l^{1/2}\|\gtildeh\|_{e,l^{-1}} \leq \gamma^{1/2}r_e^{1/2} h_e \|\gtildeh\|_{e,l^{-1}}, 
\end{align*}
where $\gamma$ shape-regularity constant of the family of triangulations.

{\bf Step 5 (conclusion):}
Then, by combining  \eqref{lab:aux1} and \eqref{aux:alpha}, we obtain that
\[
\|\misigmah\|_{\Omega_h}^2\leq 2M_0^2 \left(
h\tau\|\miuh-\miug\|_{\partial\mathcal{T}_h,\tau}^2
+2\mu \norm{\misigmah}{\Omega_h,\miA}{2}\right)+
 \dfrac{(n\lambda+2\mu)^2}{n|\Omega_h|} \gamma \max_{e\in \mathcal{E}_h^I}\left(r_e h_e^2\right) \|\gtildeh\|_{\Gamma_h,l^{-1}}^2, 
\]
and \eqref{energy2} follows from the definition of $M(h,\tau,\mu)$ (cf. \eqref{M}) and \eqref{energy1}.
\color{black}
\end{proof}

In the case of a polyhedral domain $\Omega$, the previous result holds true with $\miT{}{}=0$, since $\gtildeh = \mig = \pmb{0},$ and wellposedness of the method follows by standard arguments. In our case, $\miT{}{}$ is not zero and we proceed now to bound it.
\begin{lemma}\label{eulem6} We have $\miT{}{}=\sum_{i=1}^{6}\miT{}{i}$, where 
\begin{alignat*}{6}
&\miT{}{1} = \ipg{l^{-1/2}\gtildeh}{l^{1/2}(\misigmah-\miA\misigmah)\normal},& 
&\miT{}{2} = \ipg{l^{-1}\gtildeh}{\gtildeh},&
&\miT{}{3} = \ipg{l^{-1/2}\gtildeh}{l^{1/2}\fAsigmawx},\\
&\miT{}{4} = \ipg{l^{-1/2}\gtildeh}{l^{1/2}\frhowx},&
&\miT{}{5} = \ipg{l^{-1/2}\gtildeh}{l^{1/2}\miroh\normal},& 
&\miT{}{6} = \ipg{l^{-1/2}\gtildeh}{l^{1/2}\tau(\miuh-\miug)}.
\end{alignat*}
\end{lemma}
\begin{proof}
By using the auxiliary function defined in \eqref{Lambda}, we rewrite $\gtildeh(\ex)$ (cf. \eqref{seqg}) as 
\begin{align*}
\gtildeh(\ex)
 &= -\lx\left(\fAsigma+\miA\misigmah\nore+\frho+\miroh\nore\right)
\end{align*}
and obtain 
$\miA\misigmah\nore=-\lx^{-1}\gtildeh(\ex)-\fAsigma-\frho-\miroh\nore$.
The result is obtained by replacing the last expression in the definition of $\miT{}{}$ and arranging terms.
\end{proof}

\begin{corollary}\label{eulem7} There holds 
\begin{eqnarray}
|\miT{}{}| &\leq &  -\frac{1}{2}\normgl{\gtildeh}{2}
 + \frac{10}{4}\left(1+\CA\right)^{2}\maxE(\Ctr)^{2}\mir\normDh{\misigmah}{2}
 +\frac{10}{12}\frac{1}{\mu^2}\maxE\mir^{3}(\Cext)^{2}(\Cinv)^{2}\normDh{\misigmah}{2}\nonumber\\
& & + \frac{10}{12}\maxE\mir^{3}(\Cext)^{2}(\Cinv)^{2}\normDh{\miroh}{2}+\maxE\frac{10}{4}(\Ctr)^{2}\mir \gamma\normDh{\miroh}{2}
 + \frac{10}{4}\maxE\mir\he\alfa\normga{\miuh-\miug}{2}.\label{eqT1}
\end{eqnarray}
Moreover, if Assumptions \ref{C0}-\ref{C2} also hold, then 
\begin{eqnarray}
|\miT{}{}| &\leq &  -\frac{1}{2}\normgl{\gtildeh}{2}
 +\dfrac{1}{4M(h,\tau,\mu)}\normDh{\misigmah}{2}
 + \dfrac{1}{4 M(h,\tau,\mu)C_\rho}\normDh{\miroh}{2} + \frac{1}{2M(h,\tau,\mu)}\normga{\miuh-\miug}{2}.\label{eqT2}
\end{eqnarray}
\end{corollary}
\begin{proof} For $\miT{}{1}$, we use the Cauchy-Schwarz and Young's inequalities, the fact that $\lx \leq \He$ and the discrete trace inequality \eqref{auxlemTr} to obtain that
\begin{align*}
\miT{}{1}
&\leq \sumE\normel{\gtildeh}{}(\He)^{1/2}\norme{(\misigmah-\miA\misigmah)\nore}{}
\leq \sumE\normel{\gtildeh}{}\Ctr \mir^{1/2}\left(1+\CA\right)\normKe{\misigmah}{},\\
&\leq \sumE\left(\dfrac{1}{10}\normel{\gtildeh}{2}+\dfrac{10}{4}(\Ctr)^{2}\mir\left(1+\CA\right)^{2}\normKe{\misigmah}{2}\right).
\end{align*}
It is clear that $\miT{}{2}=-\normgl{\gtildeh}{2}$. For $\miT{}{3}$, we use the Cauchy-Schwarz inequality, estimate \eqref{auxlem1}  and Young's inequality, to deduce that
\begin{align*}
\miT{}{3}
\leq& \sumE \normel{\gtildeh}{}\frac{1}{\sqrt{3}}\mir^{3/2}\CA\Cext\Cinv\normKe{\misigmah}{}\leq \sumE \left(\dfrac{1}{10}\normel{\gtildeh}{2}+\dfrac{10}{4}\dfrac{1}{\mu^2}\mir^{3}(\Cext)^{2}(\Cinv)^{2}\normKe{\misigmah}{2}\right).
\end{align*}
For $\miT{}{4}$, we use the same arguments as in the bound of $\miT{}{3}$ and obtain  
\begin{align*}
\miT{}{4} &\leq\sumE\left(\dfrac{1}{10}\normel{\gtildeh}{2}+\dfrac{10}{12}\mir^{3}(\Cext)^{2}(\Cinv)^{2}\normKe{\miroh}{2}\right).
\end{align*}
Analogously to the bound of $\miT{}{1}$, and considering the facts that $\lx \leq \He\leq \mir h_e^\perp$ and $h_e^\perp\leq \gamma h_e$, we have
\begin{align*}
\miT{}{5} &\leq\sumE\left(\dfrac{1}{10}\normel{\gtildeh}{2}+\dfrac{10}{4}(\Ctr)^{2}\mir\gamma\normKe{\miroh}{2}\right).
\end{align*}
Finally, for $\miT{}{6}$ we use the Cauchy-Schwarz inequality, the fact $\lx \leq \He$ and Young's inequality 
\begin{align*}
\miT{}{6}
\leq&\sumE \normel{\gtildeh}{}(\He)^{1/2}\alfa^{1/2}\normea{\miuh-\miug}{}
\leq  \sumE\left(\dfrac{1}{10}\normel{\gtildeh}{2}+\dfrac{10}{4}\mir\he\alfa\normea{\miuh-\miug}{2}\right).
\end{align*}
We obtain \eqref{eqT1} gathering all the above bounds. Moreover, considering \ref{C0}-\ref{C2}, \eqref{eqT1} implies \eqref{eqT2}.
\end{proof}

Now, from Lemmas \ref{eulem1} and previous corollary, we observe that it remains to bound the $L^2$-norm of the approximation of the rotation $\miroh$. 

\begin{lemma}\label{eulem3} Let $\mif{} = \Cero$ and $\mig{} = \Cero$, and assume \ref{C4} holds. For $k \geq 1$, it holds
\begin{align}\label{rho}
\normDh{\miroh}{}\leq&\left(\dfrac{C_\rho}{2}\right)^{1/2}\normDh{\misigmah}{}+ \frac{1}{2}\left(\dfrac{M(h,\tau,\mu) C_\rho}{2} \right)^{1/2}
\normel{\gtildeh}{}.
\end{align}
\end{lemma}
\begin{proof} We follow the ideas in \cite{Ke} and consider the  orthogonal decomposition: 
\begin{eqnarray}\label{rhodesc}
 \miroh=\miroh^{0}+\miroh^{c},\hspace{0.5cm}
\miroh^{c}|_{K}:=\frac{1}{|K|}\int_{K}\miroh\quad \forall K\in\mathcal{T}_h,  \hspace{0.5cm}
\miroh^{0}=\miroh-\miroh^{c}.
\end{eqnarray}
We notice that $\miroh^{0} \in \miAh^{0}:=\lbrace \mieta \in \miAh : \ipk{\mieta}{\miv}=0, \, \forall \miv \in \Pcerok{K}, \, \forall K \in \Trian \rbrace$ and $\miroh^{c} \in \miAh^{c}:= \miAh \cap \Pcerok{\Trian}$. We proceed in two steps to bound the $\miroh^{0}$ and $\miroh^{c}$.

\textbf{Step 1:} By Lemma 2.8 in \cite{Guz} there exists $\miv \in \miBh:= :=\lbrace \mieta \in \Ltt{\Omega_{h}}: \mieta |_{K} \in \miB{K}, K \in \Trian \rbrace \subset \miVh{}$ such that 
\begin{subequations}
\begin{align}
\ipt{ \miroh^{0}}{\underline{\pmb{\gamma}}} &= \ipt{\miv}{\underline{\pmb{\gamma}}}, \hspace{0.3cm} \text{for all $\pmb{\gamma} \in \miAh$ and} \label{eq1lem4}\\
\normDh{\miv}{} &\leq C^0 \normDh{ \miroh^{0}}{}, \label{eq2lem4}
\end{align}
where $C^0>0$ is independent of $h$ and $\lambda$.
\end{subequations}
Then we rewrite equation \eqref{seqa} as
\begin{align}
\ipt{\miA\misigmah}{\miv}+\ipt{\miuh}{\midiv \miv}+\ipt{\miroh^{0}}{\miv}+\ipt{\miroh^{c}}{\miv}-\ipb{\miug}{\miv\normal} = 0.\label{Eql3}
\end{align} 
By Remark \ref{rmk:B}, we have $\ipt{\miuh}{\midiv \miv} = 0$ and $\ipb{\miug}{\miv\normal} = 0$. Now considering $\migama := \miroh^{c}$ in \eqref{eq1lem4}, we have that $\ipt{\miroh^{c}}{\miv} = \ipt{\miroh^{0}}{\miroh^{c}} = 0$, since the decomposition of $\miroh$ is orthogonal in $\underline{\pmb{L}}^{2}$. Also, by taking  $\migama= \miroh^{0}$ in \eqref{eq1lem4} we have that $\ipt{\miroh^{0}}{\miv} = \normDh{\miroh^{0}}{2}$. Thus, replacing the above terms in \eqref{Eql3}, using the Cauchy-Schwarz inequality, Lemma \ref{auxlemA3} and \eqref{eq2lem4}, we obtain 
\[
\normDh{\miroh^{0}}{2}=(\miv,\miroh^{0})\trian = -\ipt{\miA\misigmah}{\miv} 
\leq C^0 (2\mu)^{-1/2}\nDhA{\misigmah}{}\normDh{\miroh^{0}}{}.
\]

Then, 
\begin{eqnarray}
\normDh{\miroh^{0}}{}
\leq C^0 (2\mu)^{-1/2}\nDhA{\misigmah}{}
\leq C^0 (2\mu)^{-1}\normDh{\misigmah}{}.\label{bs1}
\end{eqnarray}

\textbf{Step 2:} Let $\miroh^{c} \in \miAh^{c}:= \miAh \cap \Pcerok{\Trian}$. By Lemma 3.9 in
\cite{Guz}, there exists $\miv \in \Hdiv{\Omega_{h}} \cap \Punok{\Trian}$, such that
\begin{subequations}
\begin{align}
\midiv \miv &= 0, \label{eq1lem5}\\
\ipt{\miv}{\underline{\pmb{\gamma}}} &= \ipt{\miroh^{c} }{\underline{\pmb{\gamma}}} \hspace{0.5cm} \text{for all $\pmb{\gamma} \in \miAh^{c}$,} \label{eq2lem5}\\
\normDh{\miv}{} &\leq C^c \normDh{\miroh^{c} }{},\label{eq3lem5}
\end{align}
where $C^c>0$ where is independent of $h$ and $\lambda$.
\end{subequations}

Then $\ipt{\miuh}{\midiv\miv}=0$ and $\ipb{\miug}{\miv\normal}=\ipg{\gtildeh}{\miv\normal}$, thanks to equation \eqref{seqe} and the fact that $\miv \in \Hdiv{\Omega_{h}}$ (we recall that we are assuming $k\geq 1$). Thus, with the decomposition of $\miroh$, equation \eqref{seqa} yields 
\begin{align}
\ipt{\miA\misigmah}{\miv}+\ipt{\miroh^{0}}{\miv}+\ipt{\miroh^{c}}{\miv}-\ipg{\gtildeh}{\miv\normal}=0.\label{Eq2l3}
\end{align}
Moreover, taking $\migama := \miroh^{c}$ in \eqref{eq2lem5} we have $\ipt{\miroh^{c}}{\miv} = \normDh{\miroh^{c}}{2}$ and  from Equation \eqref{Eq2l3} we obtain 
\begin{align*}
\normDh{\miroh^{c}}{2}
&=  -\ipt{\miA\misigmah}{\miv}-\ipt{\miroh^{0}}{\miv}-
\sumE\ipe{l^{-1/2}\gtildeh}{l^{1/2}\miv\normal}.
\end{align*}

Using the Cauchy-Schwarz inequality, Lemma \ref{auxlemA3}, \eqref{bs1}, \eqref{eq3lem5}, the discrete trace inequality \eqref{auxlemTr}, the facts that $\lx \leq \He$ for all $\boldsymbol{x}\in e$ and $h_e^\perp\leq \gamma h_e$, \eqref{bs1}, Lemma \ref{auxlemA3} and Assumption \ref{C4}, we deduce that
\begin{align*}
\normDh{\miroh^{c}}{2}\leq& \Bigg((2\mu)^{-1} \normDh{\misigmah}{}+\normDh{\miroh^{0}}{}
+\Bigg(\sumE {(\Ctr)}^{2}r_e \gamma 
\normel{\gtildeh}{2}\Bigg)^{1/2} \Bigg) \normDh{\miv}{}\\
\leq & \Bigg((2\mu)^{-1} (1+C^0)\normDh{\misigmah}{}
+\dfrac{1}{4C^c}\left(\dfrac{M(h,\tau,\mu) C_\rho}{2 } \right)^{1/2}
\normel{\gtildeh}{}\Bigg)C^c\normDh{\miroh^{c}}{}.
\end{align*}

Thus, considering the decomposition \eqref{rhodesc} and gathering the estimates in steps 1 and 2, we have
\begin{align*}
\normDh{\miroh}{}\leq&(2\mu)^{-1}\left((1+C^0)C^c+C^0\right)\normDh{\misigmah}{}
+\frac{1}{2}\left(\dfrac{M(h,\tau,\mu) C_\rho}{2} \right)^{1/2}
\normel{\gtildeh}{}.
\end{align*}
The result follows by recalling that $C_\rho=2(2\mu)^{-2}\left(C^0+(1+C^0)C^c\right)^2$.

\end{proof}

We are now in position to prove the main result of this section.

\begin{theorem}\label{eutheo} If the set of Assumptions {\rm C} is satisfied and $k \geq 1$, then the scheme \eqref{sq} has a unique solution.  
\end{theorem}
\begin{proof} We combine  and \eqref{energy2}, \eqref{eqT2}, \eqref{rho} and Assumptions C to obtain that
\begin{align*}
\dfrac{1}{4}\|\misigmah\|_{\Omega_h}^2+\frac{1}{2}\|\miuh-\miug\|_{\partial\mathcal{T}_h,\tau}^2&+
M(h,\tau,\mu)\frac{1}{16}\normgl{\gtildeh}{2}
\leq 
 0.
\end{align*}

Thus, we have $\misigmah=\Cerot$ in $\Omega_{h}$, $\gtildeh=\Cero$ in $\Gamma_{h}$ and $\miug=\miuh$ in $\partial \mathcal{T}_h$. In addition, by Lemma \ref{eulem3} we conclude that $\miroh=\Cerot$. Finally, from \eqref{seqa} we now have 
$$\ipt{\miuh}{\midiv\miv}-\ipb{\miuh}{\miv\normal}=0$$ for all $\miv \in \miVh{}$, which implies, after integration by parts, that $\nabla\miuh$ is constant. By \eqref{seqe} $\miuh=\Cero$ in $\Gamma_{h}$ and the fact that  $\miug=\miuh$, we conclude that $\miuh=\Cero$ in $\Omega_{h}$.
\end{proof}


\section{Error analysis}\label{sec:error}

In this section we provide \textit{a priori} error estimates for our HDG scheme. To that end, we employ the tools of the projection-based analysis of HDG method introduced for the diffusion problem \cite{diffusion}, combined with the methodology in the analyses in \cite{CQS} and \cite{Ke}. We also consider the set of Assumptions C to holds true, however the constants $C^0$ and $C^c$ are not necessarily the same and the values fractions in the right hand side of the inequalities in \ref{C0}-\ref{C5} might be different as well.

\subsection{HDG projection}
At this point, it is necessary to recall the HDG projection. On each element $K$, for $(\misigmac,\miuc)\in \Hv{1}{K}\times\Ht{1}{K}$, we consider the projection $(\pivD\misigmac,\piwD\miuc)\in \Pt{k}{K}\times\Pv{k}{K}$ such that 
\begin{subequations}
\begin{align}
\ipk{\pivD \misigmac}{\miv } &= \ipk{\misigmac}{\miv} &\forall\,\miv \in \Pt{k-1}{K}, \label{Proa}\\
\ipk{\piwD \miuc}{\miw} &= \ipk{\miuc}{\miw}  &\forall \,\miw \in \Pv{k-1}{K}, \label{Prob}\\
\ipe{(\pivD \misigmac)\normal  - \tau(\piwD \miuc\,\nT)\normal}{\mimu} &= \ipe{\misigmac\normal - \tau(\proMD \miuc\, \nT)\normal}{\mimu} &\forall \mimu \in \miMD{e},\label{Proc}
\end{align}
\end{subequations}
for all faces $e$ of the element $K$, where $\proMD{}$ denotes the $L^{2}$ projection onto $\miM{e}$. Theorem 2.1 in \cite{diffusion} allows us to conclude that this projection is well-defined. Moreover, if  $(\misigmac,\miuc)\in \Ht{k+1}{K}\times\Hv{k+1}{K}$, then
\begin{subequations}\label{projerrors}
\begin{align}
\normk{\pivD\misigmac-\misigmac}{}&\lesssim  h_{K}^{k+1}(|\miuc|_{\Hv{k+1}{K}}+|\misigmac|_{\Ht{k+1}{K}}),\label{EProa}\\
\normk{\piwD\miuc-\miuc}{}&\lesssim  h_{K}^{k+1}(|\miuc|_{\Hv{k+1}{K}}+|\midiv\misigmac|_{\Ht{k}{K}}). \label{EProb}
\end{align}

On the other hand, on each element $K$, we denote by $\piA\miroc$ the $\Ltt{K}$-projection of $\miroc \in \Ltt{K}$ into $\miAc{K}$. If $\miroc \in \Ht{k+1}{K}$, then
\begin{align}
\normk{\piA\miroc-\miroc}{}&\lesssim  h_{K}^{k+1}|\miroc|_{\Hv{k+1}{K}}.\label{EProc} 
\end{align}
\end{subequations}

We define the projections of the errors 
$
\esigma :=\pivD\misigmac-\misigmah$,  $\eu := \piwD\miuc-\miuh$, $\erho := \piA\miroc-\miroh$, $\eug := \proMD\miuc-\miug$, $\esigmag\normal := \proMD(\misigmac \normal)-\misigmag\normal$,
and the projection errors,
$ \dsigma :=\misigmac-\pivD\misigmac$, $\du :=\miuc-\piwD\miuc$, $\drho :=\miroc-\piA\miroc$.
Moreover, it is convenient to define the following auxiliary quantity related to the projection errors:
\begin{align}\label{Teta}
\Teta :=\,\bigg(&\normDh{\dsigma}{2}+\normDh{\drho}{2}+\normDhc{\dn{\dsigma}}{2}
+\normDhc{\dn{\drho}}{2}\\
&+\normgpl{\fAdsigmawx}{2}
+\normgpl{\fdrhowx}{2}
+\normgpl{\miA\dsigma\normal}{2}
+\normgpl{\drho\normal}{2}\bigg)^{1/2}.
\nonumber 
\end{align}
\begin{lemma} If $(\misigmac,\miroc)\in [\Ht{k+1}{\Omega}]^2$, then
\begin{align}
 \Teta \lesssim h^{k+1}\left( |\misigmac|_{\Ht{k+1}{\Omega}}+|\miroc|_{\Ht{k+1}{\Omega}}\right). \label{Tetaest}
\end{align}
\end{lemma}
\begin{proof}
First of all, we state Lemma 3.8 of \cite{CQS} applied to any vector-valued function $\miv \in \Ht{k+1}{\Omega}$:
\begin{align}\label{3.8CQS}
\normDhc{\dn{\boldsymbol{\delta}_{\miv}}}{} & \lesssim  \normDh{\boldsymbol{\delta}_{\miv}}{} + h^{k+1}|\miv|_{\Ht{k+1}{\Omega}}.
\end{align}
We also recall Lemma 5.2 of \cite{CQS}: For each $e \in \mathcal{E}_{h}^{\partial}$, 
\begin{eqnarray}\label{Alem1}
\norm{\pmb{\Lambda}^{\boldsymbol{\delta}_{\miv{}}}}{e,l}{}  \leq \frac{1}{\sqrt{3}}r_e\normKext{\dn{\boldsymbol{\delta}_{\miv}}}{},
\end{eqnarray}
which, together with \eqref{3.8CQS} and Assumption \ref{C0}, implies
\begin{eqnarray}\label{Alem1b}
\norm{\pmb{\Lambda}^{\delta_{\miv{}}}}{\Gamma_h,l}{}  \lesssim \normDhc{\dn{\delta_{\miv}}}{} \lesssim  \normDh{\boldsymbol{\delta}_{\miv}}{} + h^{k+1}|\miv|_{\Ht{k+1}{\Omega}}.
\end{eqnarray}

On the other hand, by a scaling argument, trace inequality, the facts that $\lx \leq \He\leq \mir h_e^\perp$ and $h_e^\perp\leq \gamma h_e$, and Assumption \ref{C0},  it is possible to show that
\begin{eqnarray}\label{Alem1c}
\normgpl{\boldsymbol{\delta}_{\miv}\normal}{}\lesssim \normDh{\boldsymbol{\delta}_{\miv}}{} + h \normDh{\nabla \boldsymbol{\delta}_{\miv}}{} .
\end{eqnarray}

The result follows after considering \eqref{3.8CQS}-\eqref{Alem1c} for $\miv=\miroc$ and $\miv=\misigmac$, where for the latter case Lemma \ref{auxlemA3} is also employed.
\end{proof}


\subsection{Main result}\label{secmain}

We now state the error estimates of our methods and postpone their proof to Section \ref{proofs}.

\begin{theorem}\label{EAt3}If $k \geq 1$ and the set of Assumptions {\rm C} holds, then 
\begin{subequations}
\begin{eqnarray}\label{Theo1}
\bignorm+\normDh{\erho}{} \lesssim (1+M(h,\tau,\mu))\Teta,
 \end{eqnarray}
where, 
$$
\bignorm\hspace{-0.1cm}:= \hspace{-0.1cm}\left(\normDh{\esigma}{2}+\|\eu-\eug\|_{\partial\mathcal{T}_h,\tau}^2+\normgl{\gtilde-\gtildeh}{2}\right)^{\frac{1}{2}}.
$$
Moreover, if elliptic regularity holds, then
\begin{eqnarray}\label{Theo1a}
\normDh{\eu}{}
\lesssim   \left(h+    R^{1/2}h^{1/2} \left(1+\tau\right) 
+   h^{1/2} R\right)\Teta
\end{eqnarray}
and
\begin{eqnarray}\label{Theo1b} \norm{\eug}{h}{}\lesssim   h \Teta +  \normDh{\eu}{}.
\end{eqnarray}

\end{subequations}
\end{theorem}

Let us point out that Theorem \eqref{EAt3} generalizes the corresponding estimate in the polyhedral case. In fact, if $\Omega$ is polyhedral and the triangulation is fitted to it, then 
$\gtilde=\gtildeh$ and $R=0$, and we recover the estimates provided in \cite{Ke}. Moreover, in contrast to the estimates in \cite{Ke}, our estimates do not depend on $\lambda$.

\begin{corollary}
\label{EAt3cor}Let us suppose that $k \geq 1$, $\tau$ is of order one and the set of Assumptions {\rm C} holds true. If $(\misigmac,\miuc,\miroc{})\in \Ht{k+1}{\Omega}\times\Hv{k+1}{Omega}\times \Ht{k+1}{\Omega}$, then 
\begin{align*}
\normDh{\misigmac-\misigmah}{}+\normDh{\miroc-\miroh}{}
\lesssim h^{k+1}\left( |\misigmac|_{\Ht{k+1}{\Omega}}+|\miroc|_{\Ht{k+1}{\Omega}}\right).
\end{align*}
Moreover, if elliptic regularity holds, then
\[\normDh{\miuc-\miuh}{}\lesssim  h^{k+1}\left( |\misigmac|_{\Ht{k+1}{\Omega}}+|\miroc|_{\Ht{k+1}{\Omega}}+|\miuc|_{\Hv{k+1}{\Omega}}\right) \]
and
\[
\norm{\eug}{h}{}\lesssim  \left(h+
 R^{1/2}h^{1/2}\right)h^{k+1}\left( |\misigmac|_{\Ht{k+1}{\Omega}}+|\miroc|_{\Ht{k+1}{\Omega}}+|\miuc|_{\Hv{k+1}{\Omega}}\right).
\]

\begin{proof} Since $\misigmac-\misigmah = \esigma + \dsigma$,  $\miroc-\miroh = \esigma + \dsigma$ and $\miuc-\miuh = \eu + \du$, the result is a direct consequence of previous theorem, triangle inequality, the estimate in \eqref{Tetaest} and projection error estimates in \eqref{projerrors}. 
\end{proof}
\end{corollary}


\section{Proofs of the error estimates}\label{proofs}

This section is divided in several steps that will lead to the results stated in Theorem \ref{EAt3}. We will follow the main procedures behind the proofs of wellposedness in Section \ref{ExUn}.  We will first employ an energy argument to control the $L^2$-norm of the errors $\esigma$ and $\erho$. Then, we will use a duality argument that allows us to control the the $L^2$-norm of $\eu$ under regularity assumptions.

\subsection{Energy argument}\label{secEarg}

It is not difficult to realize that the projections of the errors satisfy 
\begin{subequations}\label{EAeq}
\begin{align}
        (\miA\esigma,\miv)\trian\hspace{-0.05cm}+\hspace{-0.05cm}(\eu,\midiv \miv)\trian\hspace{-0.05cm}+\hspace{-0.05cm}(\erho,\miv)\trian\hspace{-0.05cm}-\hspace{-0.05cm}\ipd{\eug}{\miv\normal}\dtrian{} &= \hspace{-0.05cm}-\hspace{-0.05cm}(\miA\dsigma,\miv)\trian\hspace{-0.05cm}-\hspace{-0.05cm}(\drho,\miv)\dtrian,\label{era}\\
        (\esigma,\nabla \miw)\trian-\ipd{\esigmag\normal}{\miw}\dtrian{} &= 0,\label{erb}\\
        (\esigma,\mieta)\trian{} &= -(\dsigma,\mieta)\trian,\label{erc}\\
        \ipf{\esigmag\normal}{\mimu}&= 0,\label{erd}\\
        \ipd{\eug}{\mimu}\gamah{} &= \ipg{\gtilde-\gtildeh}{\mimu}.\label{ere}    
\end{align}
for all $(\miv,\miw,\mieta,\mimu) \in \miVh \times \miWh \times \miAh \times \miMh$. Moreover, combining \eqref{defg} and \eqref{seqg}, we obtain that
\begin{align}\label{gt_ght}
(\gtilde-\gtildeh(\ex)) =  - \int_{0}^{\lx}\hspace{-0.2cm}(\miA\esigma+ +\erho)(\ex + s \tx) \tx ds 
-  \int_{0}^{\lx}\hspace{-0.2cm}(\miA\dsigma +\drho)(\ex + s \tx) \tx ds 
.
\end{align}
\end{subequations}

Similarly to Lemma \ref{eulem1}, the following energy-type identities hold.

\begin{lemma}\label{EAp1} The projection of the error satisfy
\begin{subequations}
\begin{equation}\label{EAp1_a}
\nDhA{\esigma}{2}+\|\eu-\eug \|_{\partial\mathcal{T}_h,\tau}^2=\ipt{\dsigma}{\erho}-\ipt{\miA\dsigma}{\esigma}-\ipt{\drho}{\esigma}+\miT{}{},
\end{equation}
where $\mathbb{T}=\ipd{\gtilde-\gtildeh}{\esigmag\normal}\gamah $.
Moreover,
\begin{equation}\label{EAp1_b}
\|\esigma\|_{\Omega_h}^2+\|\eu-\eug\|_{\partial\mathcal{T}_h,\tau}^2\leq M(h,\tau,\mu)\left(\nDhA{\esigma}{2}+\|\eu-\eug \|_{\partial\mathcal{T}_h,\tau}^2\right)+
 \dfrac{(n\lambda+2\mu)^2}{n|\Omega_h|} \gamma \max_{e\in \mathcal{E}_h^I}\left(r_e h_e^2\right) \|\gtilde-\gtildeh\|_{\Gamma_h,l^{-1}}^2.
\end{equation}
\end{subequations}
\end{lemma}
\begin{proof} It follows by the same arguments and steps as in the proof of Lemma \ref{eulem1}. In this case, $(\esigma,\eu,\erho,\eug)$ plays the role of $(\misigmah,\miuh,\miroh,\miug)$ and $\gtilde-\gtildeh$ plays the role of $\gtildeh$.
\end{proof}

Following the structure in Section \ref{ExUn}, we  rewrite the term $ \mathbb{T}$ to facilitate the bound in the estimate of $\esigma$. First of all, we rewrite $\gtilde-\gtildeh$ as follows  
\[ \gtilde(\ex)-\gtildeh(\ex)=
-\int_{0}^{\lx}\hspace{-0.2cm}\miA(\misigmac-\misigmah)(\ex +s\nore)\nore ds-\int_{0}^{\lx}\hspace{-0.2cm}(\miroc-\miroh)(\ex +s\nore)\nore ds .\]
Now, since $\misigmac-\misigmah= \dsigma + \esigma$, by the definition in \eqref{Lambda}, we can write
\begin{align*}
 -\int_{0}^{\lx}\hspace{-0.4cm}\miA(\misigmac-\misigmah)(\ex +s\nore)\nore ds
= \lx\Big(\fAdsigma+\hspace{-0.1cm}\miA\dsigma(\ex)\nore
 +\fAesigma+\hspace{-0.1cm}\miA\esigma(\ex)\nore\Big).
\end{align*}
Similarly,
$\displaystyle \int_{0}^{\lx}\hspace{-0.2cm}(\miroc-\miroh)(\ex +s\nore)\nore ds= \lx\left(\fdrho+\drho\nore +\ferho+\erho\nore \right).$
Thus, replacing the above terms in expression $\gtilde(\ex)-\gtildeh(\ex)$, we have 
\begin{align*} \gtilde(\ex)-\gtildeh(\ex)=&
 -\lx[\fAdsigma+\miA\dsigma(\ex)\nore
 +\fAesigma+\miA\esigma(\ex)\nore ]\\
 &-\lx[\fdrho +\drho\nore +\ferho+\erho\nore]\nonumber.
\end{align*}

Let $\miT{\rho}{}:=\fdrho +\drho\nore +\ferho+\erho\nore$. We obtain then
\begin{align}\label{dec}
\miA \esigma(\ex)\nore =-\frac{1}{\lx}(\gtilde(\ex)-\gtildeh(\ex))-\miT{\rho}{}
-\fAdsigma-\miA\dsigma(\ex)\nore -\fAesigma.
\end{align}
Using  \eqref{seqf} and \eqref{Proc}, we have $\esigmag \normal=\esigma \normal-\tau(\eu-\eug)$ for all $e \in \Eh{}$ and, similarly to the arguments in Section \ref{ExUn}, we decompose $ \miT{}{}= \sum_{i=1}^{6}\miT{}{i}$, where
\begin{alignat*}{6}
 \miT{}{1}&=-\ipg{l^{-1/2}(\gtilde-\gtildeh)}{l^{1/2}(\esigma-\miA \esigma)\normal},&
 \miT{}{2}&=-\ipg{l^{-1}(\gtilde-\gtildeh)}{\gtilde-\gtildeh},\\
  \miT{}{3}&=-\ipg{l^{-1/2}(\gtilde-\gtildeh)}{l^{1/2}\fAesigmawx},&
   \miT{}{4}&=-\ipg{l^{-1/2}(\gtilde-\gtildeh)}{l^{1/2}\miT{\rho}{}},\\
 \miT{}{5}&=-\ipg{l^{-1/2}(\gtilde-\gtildeh)}{l^{1/2}(\fAdsigmawx+\miA\dsigma\normal)},&
 \miT{}{6}&=-\ipg{l^{-1/2}(\gtilde-\gtildeh)}{l^{1/2}\tau(\eu-\eug)}.
\end{alignat*}

\begin{lemma}\label{EAcorT} It holds
\begin{align*}
 |\miT{}{}| \leq & -\frac{1}{2}\normgl{(\gtilde-\gtildeh)\nT}{2}
 + \frac{10}{4}\left(1+\CA\right)^{2}\maxE(\Ctr)^{2}\mir\normDh{\esigma}{2}
  +\frac{10}{12}\frac{1}{\mu^2}\maxE\mir^{3}(\Cext)^{2}(\Cinv)^{2}\normDh{\esigma}{2}\nonumber\\
  &+ \frac{10}{12}\maxE\mir^{3}(\Cext)^{2}(\Cinv)^{2}\normDh{\miroh}{2}+\maxE\frac{10}{4}(\Ctr)^{2} \gamma\mir\normDh{\erho}{2}\nonumber\\
  &+ \frac{10}{4}\maxE\mir\he\alfa\normga{\eu-\eug}{2}+\frac{15}{2} \Teta^2 \left((2\mu)^{-1}+(2\mu)\right)  
  \nonumber.
\end{align*}
Moreover, if Assumptions \ref{C0}-\ref{C2} also hold, then 
\begin{align}\label{EQcorT2}
|\miT{}{}| \leq&   -\frac{1}{2}\normgl{\gtilde-\gtildeh}{2}
 +\dfrac{1}{4M(h,\tau,\mu)}\normDh{\esigma}{2}\nonumber
 + \dfrac{1}{4 M(h,\tau,\mu)C_\rho}\normDh{\erho}{2} \\
 &+ \frac{1}{2M(h,\tau,\mu)}\normga{\eu-\eug}{2}
+\frac{15}{2} \Teta^2 \left((2\mu)^{-1}+(2\mu)\right).
\end{align}
\end{lemma}

\begin{proof} It is clear that $\miT{}{2}=-\normgl{(\gtilde-\gtildeh)
}{2}$.

For $\miT{}{1}$, $\miT{}{3}$  and $\miT{}{6}$, by mimicking the corresponding steps in the proof of Lemma \ref{eulem7}, we deduce that
\begin{align*}
\miT{}{1}
\leq \sumE\left(\dfrac{1}{10}\normel{\gtilde-\gtildeh}{2}+\dfrac{10}{4}(\Ctr)^{2}\mir\left(1+\CA\right)^{2}\normKe{\esigma}{2}\right),
\end{align*}
\begin{align*}
\miT{}{3}
\leq \sumE \left(\dfrac{1}{10}\normel{\gtilde-\gtildeh}{2}+\dfrac{10}{4}\dfrac{1}{\mu^2}\mir^{3}(\Cext)^{2}(\Cinv)^{2}\normKe{\esigma}{2}\right)
\end{align*}
and
\begin{align*}
\miT{}{6}
\leq&  \sumE\left(\dfrac{1}{10}\normel{\gtilde-\gtildeh}{2}+\dfrac{10}{4}\mir\he\alfa\normea{\eu-\eug}{2}\right).
\end{align*}


For $\miT{}{4}$, we proceed similarly but considering in addition \eqref{auxlem1},  discrete trace inequality \eqref{auxlemTr} and  the facts that $\lx \leq \He\leq \mir h_e^\perp$ and $h_e^\perp\leq \gamma h_e$:
\begin{align*}
\miT{}{4}
\leq \,& \sumE\normel{\gtilde-\gtildeh}{}(2\mu)^{1/2}\left( \normell{\fdrhowx}{}+\normell{\drho\nore}{}
+\normell{\ferhowx}{}+\normell{\erho\nore}{}\right)\\
\leq& \quad\, \normgl{\gtilde-\gtildeh}{2}\,\,\Teta (2\mu)^{1/2}
 +  \normgl{\gtilde-\gtildeh}{2} \left(\frac{1}{3}\maxE  \mir^{3}(\Cext)^{2}(\Cinv)^{2}\right)^{1/2}\normDh{\erho}{} (2\mu)^{1/2}\\
 &+  \normgl{\gtilde-\gtildeh}{2}\left(\maxE  (C_{tr}^e)^2r_e \gamma \right)^{1/2}\normDh{\erho}{} (2\mu)^{1/2}\\
\leq& 
\frac{1}{10} \normgl{\gtilde-\gtildeh}{2}
+\frac{15}{2}\Teta^2(2\mu)
\,+\frac{5}{2}(2\mu)\maxE \mir^{3}(\Cext)^{2}(\Cinv)^{2}\normDh{\erho}{2}  +
\frac{15}{2}(2\mu)\maxE (C_{tr}^e)^2r_e \gamma \normDh{\erho}{2} .
\end{align*}

Finally, for $\miT{}{5}$ we use the Cauchy-Schwarz inequality, Lemma \ref{auxlemA3}, the definition in \eqref{Teta} and Young's inequality, to obtain
\begin{align*}
 \miT{}{5}=&-\ipg{l^{-1/2}(\gtilde-\gtildeh)}{l^{1/2}(\fAdsigmawx+\miA\dsigma\normal)}
  \leq\sumE\normel{\gtilde-\gtildeh}{}\left( \normell{\fAdsigmawx}{}+\normell{\miA\dsigma\nore}{}\right)\\
\leq& \normgl{\gtilde-\gtildeh}{}\,\,(2\mu)^{-1/2}\Teta
\leq \frac{1}{10}  \normgl{\gtilde-\gtildeh}{2}+(2\mu)^{-1}\frac{5}{2}\Teta^2.
\end{align*}

 We obtain the first inequality gathering all the above bounds. Moreover, considering \ref{C0}-\ref{C2}, the first inequality implies \eqref{EQcorT2}.
\end{proof}

From Lemma \ref{EAp1} and previous corollary, we observe that it remains to bound the $L^2$-norm of the approximation of the rotation $\erho$. To that end, proceeding exactly as in Lemma \ref{eulem3}, but taking into account the presence of the projection errors $\dsigma$ and $\drho$, it is possible to deduce the following result.
\begin{lemma}\label{EAlemrho} Suppose Assumption \ref{C4} holds true. If $k \geq 1$, then
\begin{align*}
\normDh{\erho}{}\leq&\left(\dfrac{C_\rho}{2}\right)^{1/2}\normDh{\esigma}{}+ \left(\dfrac{M(h,\tau,\mu) C_\rho}{2} \right)^{1/2}
\normel{\gtilde-\gtildeh}{}
+\left(\dfrac{C_\rho}{2}\right)^{1/2}\left( \normDh{\dsigma}{}+\normDh{\drho}{}\right).
\end{align*}

\end{lemma}
%
%

\subsection{Conclusion of the proof  of estimate (\ref{Theo1}) in Theorem \ref{EAt3}}

\begin{proof} We use combine \eqref{EAp1_a} and \eqref{EAp1_b}, together with the Cauchy-Schwarz inequality to obtain that
\begin{align*} 
\normDh{\esigma}{2}+\|\eu-\eug \|_{\partial\mathcal{T}_h,\tau}^2
\leq & M(h,\tau,\mu)
\normDh{\dsigma}{}\normDh{\erho}{}
+M(h,\tau,\mu)\nDhA{\dsigma}{}\nDhA{\esigma}{}\\
&+M(h,\tau,\mu)\normDh{\drho}{}\normDh{\esigma}{}+M(h,\tau,\mu)|\miT{}{}|
+
 \dfrac{(n\lambda+2\mu)^2}{n|\Omega_h|} \gamma \max_{e\in \mathcal{E}_h^I}\left(r_e h_e^2\right) \|\gtilde-\gtildeh\|_{\Gamma_h,l^{-1}}^2.
\end{align*}
Then, using  Young's inequality and Lemma \ref{auxlemA3}, we have that
\begin{align*}
\normDh{\esigma}{2}+\|\eu-\eug \|_{\partial\mathcal{T}_h,\tau}^2
\leq&
 M(h,\tau,\mu)^2\left(\dfrac{1}{8C_\rho} +\dfrac{2}{(2\mu)^2}\right)\normDh{\dsigma}{2}
+\frac{1}{16 C_\rho}\normDh{\erho}{2}
+\dfrac{1}{8}\normDh{\esigma}{2}\\
&+\dfrac{1}{2}M(h,\tau,\mu)^2\normDh{\drho}{2}+M(h,\tau,\mu)|\miT{}{}|
+
 \dfrac{(n\lambda+2\mu)^2}{n|\Omega_h|} \gamma \max_{e\in \mathcal{E}_h^I}\left(r_e h_e^2\right) \|\gtilde-\gtildeh\|_{\Gamma_h,l^{-1}}^2.
\end{align*}

By \eqref{EQcorT2},
\begin{align*}
\normDh{\esigma}{2}+&\dfrac{1}{2}\|\eu-\eug \|_{\partial\mathcal{T}_h,\tau}^2
+\frac{M(h,\tau,\mu)}{2}\normgl{\gtilde-\gtildeh}{2}\\
\leq&
M(h,\tau,\mu)^2\left(\dfrac{1}{8C_\rho} +\dfrac{2}{(2\mu)^2}\right)\normDh{\dsigma}{2}
+\left(\dfrac{1}{16C_\rho}+\dfrac{1}{4C_\rho}\right)\normDh{\erho}{2}
+\dfrac{3}{8}\normDh{\esigma}{2}\\
&+\dfrac{1}{2}M(h,\tau,\mu)^2
+C\Teta^2
+
 \dfrac{(n\lambda+2\mu)^2}{n|\Omega_h|} \gamma \max_{e\in \mathcal{E}_h^I}\left(r_e h_e^2\right) \|\gtilde-\gtildeh\|_{\Gamma_h,l^{-1}}^2,
\end{align*}
Then, by Lemma \ref{EAlemrho} considering again the definition of $\Teta$  (cf. \eqref{Teta}) to absorb the projection error terms $\dsigma$ and  $\drho$, we deduce that
\begin{align*}
\normDh{\esigma}{2}+&\dfrac{1}{2}\|\eu-\eug \|_{\partial\mathcal{T}_h,\tau}^2
+\frac{M(h,\tau,\mu)}{2}\normgl{\gtilde-\gtildeh}{2}\\
\leq&
\dfrac{3}{4}\normDh{\esigma}{2}
+\left(\dfrac{3}{8}M(h,\tau,\mu)+
 \dfrac{(n\lambda+2\mu)^2}{n|\Omega_h|} \gamma \max_{e\in \mathcal{E}_h^I}\left(r_e h_e^2\right)\right) \normgl{\gtilde-\gtildeh}{2}
+C\Teta^2,
\end{align*}
and the results follows by Assumption \ref{C5}.
\end{proof}



\subsection{Duality argument}

In this section we use a duality argument to obtain an estimate for $\eu$ and we introduce the auxiliary problem:
\begin{subequations}\label{DAeq}
\begin{align}
\midiv \mipsi &= \miteta      \hspace{0.5cm} \text{in} \hspace{0.5cm}\Omega,\label{DAeqa}\\
\miA\mipsi-\nabla\miphi +\mixi & = 0\hspace{0.5cm}\text{in} \hspace{0.5cm}\Omega,\label{DAeqb}\\
\miphi &= \miteta\hspace{0.5cm}\text{on}\hspace{0.2cm}\partial\Omega.\label{DAeqc}
\end{align}
Here $\mixi=\frac{1}{2}(\nabla\miphi-\nabla^{t}\miphi)$. We assume the solution $(\mixi,\miphi)$ has the elliptic regularity property:
\begin{align}
\norm{\mipsi}{H^{s}(\Omega)}{}+\norm{\miphi}{H^{1+s}(\Omega)}{}\leq C_{reg}\norm{\pmb{\theta}}{\Omega}{}\label{DAeqd}
\end{align}
\end{subequations}
for some $s\geqslant 0$ and $C_{reg}>0$ independent of the Lam\'e coefficients. This property holds, for example, with $s=1$ in the case of planar elasticity with scalar coefficients in a convex domain; see \cite{Dual}.

\begin{lemma}\label{Darglem1} Suppose the Assumption \ref{C0} is satisfied and \eqref{DAeqd} holds with $s=1$, then 
\begin{subequations}
\begin{align}
\normGh{(\proI-\proMD)\miphi}{-1}{}&\lesssim h\normO{\miteta}{}, \label{Darga}\\
\normGl{(\proI-\proMD)\partial_{\normal}\miphi}{}{}&\lesssim \miR h\normO{\miteta}{},\label{Dargb} \\
\normGl{\miphi+l\partial_{\normal}\miphi}{-3}{}&\lesssim \normO{\miteta}{}, \label{Dargc}\\
\normGl{\miphi}{-2}{}&\lesssim \normO{\miteta}{}.\label{Dargd}
\end{align}
\end{subequations}
\end{lemma}
\begin{proof} It follows from Lemma 5.5 in \cite{CQS} applied to each component of $\miphi$.
\end{proof}
\begin{proposition}\label{DAp1} The following identity holds
\begin{align*}
(\eu,\miteta)\trian=&(\miA\esigma,\dpsi)\trian+(\erho,\dpsi)\trian-(\miA\dsigma+\drho,\pivD\mipsi)\trian
+(\esigma,\dxi)\trian-(\dsigma,\piA\mixi)\trian+\miT{}{u,\sigma},
\end{align*}
where $\dxi=\mixi-\piA\mixi$, $\dpsi=\mipsi-\pivD\psi$, and $\miT{}{u,\sigma}:=\ipg{\eug}{\mipsi\normal}-\ipg{\esigmag\normal}{\miphi}$,
\end{proposition}
\begin{proof}By equation \eqref{DAeq}, we can write 
\begin{align*}
(\eu,\miteta)\trian{}=&(\eu,\midiv\mipsi)\trian+(\esigma,\miA\mipsi-\nabla\miphi+\mixi)\trian  \\
=&(\eu,\midiv\pivD\mipsi)\trian+(\eu,\midiv\dpsi)\trian+(\miA\esigma,\pivD\mipsi)\trian+(\miA\esigma,\dpsi)\trian\\
 &-(\esigma,\nabla\piwD\miphi)\trian-(\esigma,\nabla\dphi)\trian+(\esigma,\mixi)\trian.
\end{align*}
Next, note by \eqref{Proa} of the projection and the fact that $\eu \in \miWh{}$, we have
\begin{align*}
\ipt{\eu}{\midiv\dpsi} = \ipb{\eu}{\dpsi\normal}-\ipt{\nabla\eu}{\dpsi} = \ipb{\eu}{\dpsi\normal}.
\end{align*}
Similarly, with the fact $\esigma \in \miVh{}$ and \eqref{Prob}, we obtain 
\begin{align*}
\ipt{\esigma}{\nabla\dphi} = \ipb{\esigma\normal}{\dphi}-\ipt{\midiv\esigma}{\dphi} = \ipb{\esigma\normal}{\dphi}.
\end{align*} 
Inserting these two results onto the first equation, we get
\begin{align}\label{eqeu}
\ipt{\eu}{\miteta} = &\ipt{\eu}{\midiv\pivD\mipsi}+\ipt{\miA\esigma}{\pivD\mipsi}-\ipt{\esigma}{\nabla\piwD\miphi}\nonumber\\
 &+\ipt{\esigma}{\mixi}+\ipt{\miA\esigma}{\mipsi}+\ipb{\eu}{\dpsi\normal}-\ipb{\esigma\normal}{\dphi}.
\end{align}
Taking $\miv:=\pivD\mipsi$, and $\miw:=\piwD\miphi$, in the error equations \eqref{era} and \eqref{erb}, respectively, we have
\begin{align*}
\ipt{\miA\esigma}{\pivD\mipsi}\hspace{-0.1cm}+\hspace{-0.05cm}\ipt{\eu}{\midiv\pivD\mipsi}\hspace{-0.1cm}=\hspace{-0.05cm} -\ipt{\erho}{\pivD\mipsi}\hspace{-0.1cm}+\hspace{-0.05cm}\ipb{\eug}{\pivD\mipsi}
\hspace{-0.1cm}-\hspace{-0.05cm}\ipt{\miA\dsigma+\drho}{\pivD\mipsi}
\end{align*}
and
$
\ipt{\esigma}{\nabla\piwD\miphi} = \ipb{\esigmag\normal}{\piwD\miphi}
$.
Replacing these last two expression in to \eqref{eqeu}, we obtain 
\begin{align*}
\ipt{\eu}{\miteta}= &\ipt{\esigma}{\mixi} + \ipt{\miA\esigma}{\dpsi}-\ipt{\erho}{\pivD\mipsi}
-\ipt{\miA\dsigma+\drho}{\pivD\mipsi}\\
&-\ipb{\eu}{\dpsi\normal}-\ipb{\esigma\normal}{\dphi}+\ipb{\eug}{\pivD\mipsi\normal}-\ipb{\esigmag\normal}{\piwD\miphi}.
\end{align*}
Next, note that $(\erho,\mipsi)\trian=0$ since $\erho \in \miAS{\Omega_{h}}$ and $\mipsi$ is symmetric. Also, note that by the regularity assumption, $(\mipsi,\miphi) \in \Ht{1}{\Omega}\times \Hv{1}{\Omega}$, so $\mipsi\normal,\phi$ are single-valued on each face $e \in \Eh{}$. This implies that 
\begin{align*}
\ipb{\eug}{\mipsi\normal}&= \ipg{\eug}{\mipsi\normal} = \ipg{\gtilde -\gtildeh}{\mipsi\normal}, &&\text{by \eqref{ere},}\\
\ipb{\esigmag\normal}{\miphi}&= \ipb{\esigmag\normal}{\proMD\miphi} = \ipg{\esigmag\normal}{\proMD\miphi}, 
&&\text{by \eqref{erc} and \eqref{DAeqc}.} 
\end{align*}
Inserting these three terms onto the previous equation, we can write
\begin{align*}
\ipt{\erho}{\pivD\mipsi}&= \ipt{\erho}{\pivD\mipsi-\mipsi} = -\ipt{\erho}{\dpsi}, & \\
\ipb{\eug}{\pivD\mipsi\normal}&= \ipb{\eug}{\pivD\mipsi\normal-\mipsi\normal+\mipsi\normal} = -
\ipb{\eug}{\dpsi\normal} + \ipg{\eug}{\mipsi\normal},\\
\ipb{\esigmag\normal}{\piwD\miphi}&= -\ipb{\esigmag\normal}{\proMD\miphi-\piwD\miphi} + \ipg{\esigmag\normal}{\piwD\miphi}
= - \ipb{\esigmag\normal}{\dphi}+\ipg{\esigmag\normal}{\piwD\miphi}.
\end{align*}
Therefore, we have
\begin{align}\label{eu:aux1}
(\eu,\miteta)\trian=(\esigma,\mixi)\trian+(\miA\esigma,\dpsi)\trian+(\erho,\dpsi)\trian-(\miA\dsigma+\drho,\pivD\mipsi)\trian+\widetilde{\miT{}{}}+\miT{}{u,\sigma},
\end{align}
where, 
$ \widetilde{\miT{}{}}:=\ipd{\eu-\eug}{\dpsi\normal}\dtrian-\ipd{\esigma\normal-\esigmag\normal}{\dphi}\dtrian$,
$\miT{}{u,\sigma}:=\ipg{\eug}{\mipsi\normal}-\ipg{\esigmag\normal}{\miphi}$
and $\dphi :=\miphi-\piwD\miphi$. 

Now, by \eqref{erc}, we have 
\begin{align}\label{eu:aux2}
\ipt{\esigma}{\mixi} =\ipt{\esigma}{\dxi}+\ipt{\esigma}{\piA\mixi}=\ipt{\esigma}{\dxi}-\ipt{\dsigma}{\piA\mixi}.
\end{align}
 Moreover, 
\begin{align}\label{eu:aux3}
 \widetilde{\miT{}{}}&=\ipb{\eu-\eug}{\dpsi\normal}-\ipb{\esigma\normal-\esigmag\normal}{\dphi}
 = \ipb{\eu-\eug}{\dpsi\normal}-\ipb{\esigma\normal-\esigmag\normal}
{\proMD\miphi-\piwD\miphi}\nonumber\\
 &= \ipb{\eu-\eug}{\dpsi\normal}-\ipb{\tau(\eu-\eug)}{\proMD\miphi-\piwD\miphi}
 = \ipb{\eu-\eug}{\dpsi\normal-\tau(\proMD\miphi-\piwD\miphi)}
 = 0,
\end{align}
by property \eqref{Proc}. The result follows by gathering \eqref{eu:aux1}, \eqref{eu:aux2} and \eqref{eu:aux3}.
\end{proof}

\begin{lemma}\label{DAl1} We have
$\displaystyle \miT{}{u,h}=\sum_{i=1}^{11}\miT{i}{u,h}$,
where,
\begin{alignat*}{10}
\miT{1}{u,h}&=\ipg{(\gtilde-\gtildeh)/l}{\miphi+l\partial_{n}\miphi},& 
\miT{2}{u,h}&=-\ipg{\gtilde-\gtildeh}{(\proI -\proMD)\partial_{n}\miphi},& 
\miT{3}{u,h}&=\ipg{\gtilde-\gtildeh}{(\mipsi-\nabla\miphi)\normal},\\
\miT{4}{u,h}&=\ipg{\fdrhon}{\miphi},&
\miT{5}{u,h}&=\ipg{\drho\normal}{\miphi},& 
\miT{6}{u,h}&=\ipg{\ferhon}{\miphi}, \\
\miT{7}{u,h}&=\ipg{\erho\normal}{\miphi},& 
\miT{8}{u,h}&=\ipg{\fAesigman}{\miphi},&
\miT{9}{u,h}&=\ipg{\miA\dsigma\normal}{\miphi},\\
\miT{10}{u,h}&=\ipg{\fAdsigman}{\miphi},&
\miT{11}{u,h}&=\ipg{\alfa(\eu-\eug)}{\miphi}.
\end{alignat*}
\end{lemma}
\begin{proof} By \eqref{seqf} and \eqref{dec},
\[\miT{}{u,\sigma}=\ipg{\eug}{\mipsi\normal}-\lbrace\ipg{-(\gtilde-\gtildeh)/l-\miT{\rho}{}-\fAesigman-\fAdsigman-\miA\dsigma\normal-\alfa(\eu-\eug)}{\miphi}\rbrace,\]

Using the fact that $\eug=P_{M}(\gtilde-\gtildeh)$ 
and adding and subtracting the term $\ipg{\gtilde-\gtildeh}{\partial_{n}\miphi}$, we obtain 
\begin{align*}
\miT{}{u,\sigma}= &-\ipg{\gtilde-\gtildeh}{(\proI-\proMD\partial_{n})\miphi}
+\ipg{\gtilde-\gtildeh}{(\mipsi-\nabla\miphi)\normal}+\ipg{(\gtilde-\gtildeh)/l}{\miphi + l\partial_{n}\miphi}\\
&+ \ipg{\fdrhon+\drho\normal+\ferhon+\erho\normal}{\miphi}+\ipg{\fAesigman+\miA\dsigma\normal+\fAdsigman}{\miphi}
+\ipg{\alfa(\eu-\eug)}{\miphi}
\end{align*}
and the result follows.
\end{proof}
\begin{lemma}\label{DAl2} If the set of Assumptions {\rm C} is satisfied and \eqref{DAeqd} holds with $s=1$, then  
\begin{align*}  
|\miT{}{u,h}|\lesssim&  (Rh)^{1/2} \left(1+\tau\right) 
\bignorm \,\,\normO{\miteta}{}
 +  R h^{1/2}\normDh{\erho}{}\,\,\normO{\miteta}{}
+ \Teta\normO{\miteta}{}.
\end{align*}
\end{lemma}

\begin{proof}
By Lemma 3.6, we can write $\miT{}{u,h}=\sum_{i=1}^{11}\miT{i}{u,h}$. Applying the Cauchy-Schwarz inequality, we get 
\begin{alignat*}{6}
|\miT{1}{u,h}|&\leq  \normgpl{\gtilde-\gtildeh}{}\normgl{\miphi/l+\psn\miphi}{},&
|\miT{2}{u,h}|&\leq  \normgl{\gtilde-\gtildeh}{}\normgpl{(Id-\proMD)\psn\miphi}{},
\end{alignat*}
\begin{alignat*}{6}
|\miT{3}{u,h}|&\leq  \normgl{\gtilde-\gtildeh}{}\normgpl{(\mipsi-\nabla\miphi)\normal}{},&
|\miT{4}{u,h}|&\leq \normgll{\fdrhowx}{}\normgml{\miphi}{},&
|\miT{5}{u,h}|&\leq  \normgll{\drho\normal}{}\normgml{\miphi}{},\\
|\miT{6}{u,h}|&\leq  \normgll{\ferhowx}{}\normgml{\miphi}{},&
|\miT{7}{u,h}|&\leq  \normgll{\erho\normal}{}\normgml{\miphi}{},&
|\miT{8}{u,h}|&\leq  \normgll{\fAesigmawx}{}\normgml{\miphi}{},\\
|\miT{9}{u,h}|&\leq  \normgll{\miA\dsigma\normal}{}\normgml{\miphi}{},&
|\miT{10}{u,h}|&\leq  \normgll{\fAdsigmawx}{}\normgml{\miphi}{},&
|\miT{11}{u,h}|&\leq  \normgal{\eu-\eug}{}\normgml{\miphi}{}.
\end{alignat*}
By \eqref{Dargb},\eqref{Dargc} and the fact that, for all $\ex$ in a face $e$,
\begin{align}
|\lx |\leq \He = \mir\he \leq \mir h_{e} \leq \mir h \leq \miR h , \label{EQlx}
\end{align}
we have 
$
|\miT{1}{u,h}|\lesssim(\miR h)^{2}\normgl{\gtilde-\gtildeh}{}\normO{\miteta}{}$ and 
$|\miT{2}{u,h}|\lesssim\miR h\normgl{\gtilde-\gtildeh}{}\normO{\miteta}{}$.

Now, since 
\begin{align*}
\normGl{(\mipsi-\nabla\miphi)\normal}{}{}\leq& (Rh)^{1/2}\normgh{(\mipsi-\nabla\miphi)\normal}{}
\lesssim (Rh)^{1/2}
\left( \norm{\mipsi}{\Ht{1}{\Omega}}{}+\norm{\miphi}{\Hv{2}{\Omega}}{}\right)\lesssim (Rh)^{1/2}{}\normO{\miteta}{},
\end{align*}
we get $|\miT{3}{u,h}|\lesssim(Rh)^{1/2}\normgl{\gtilde-\gtildeh}{}\normO{\miteta}{}$.
On the other hand, we use the estimates  \eqref{auxlem1},  \eqref{Alem1}, \eqref{Dargc}, \eqref{EQlx} and Assumption (S.1) to obtain  
\begin{align*}
|\miT{4}{u,h}|&\lesssim (Rh)^{1/2}\normDhc{\dn{\drho}}{}\normO{\miteta}{},
&|\miT{6}{u,h}|&\lesssim R h^{1/2}\normDh{\erho}{}\normO{\miteta}{},\\
|\miT{10}{u,h}|&\lesssim(Rh)^{1/2}\normDhc{\dn{\miA\dsigma}}{}\normO{\miteta}{},
&|\miT{8}{u,h}|&\lesssim R h^{1/2}\normDh{\esigma}{}\normO{\miteta}{}.
\end{align*}
Now, using \eqref{Dargc}, \eqref{EQlx} and \eqref{auxlemA3},  we see that 
$
|\miT{5}{u,h}|\lesssim R h^{1/2}\normGh{\drho\normal}{}{}\normO{\miteta}{}$ and
$|\miT{9}{u,h}|\lesssim Rh^{1/2}\normGh{\dsigma}{}{}\normO{\miteta}{}$.
Considering \eqref{Dargc} and \eqref{EQlx}, we obtain 
$$
|\miT{11}{u,h}|\lesssim Rh\tau\normga{\eu-\eug}{}\normO{\miteta}{}.
$$

Finally, by \eqref{Dargc}, \eqref{EQlx} and discrete trace inequality (Lemma \ref{auxlemTr}) we obtain
$
|\miT{7}{u,h}|\lesssim R h^{1/2}\normDh{\erho}{}\normO{\miteta}{}.
$
Then, by the definition of $|||(\esigma,\eu-\eug,\gtilde-\gtildeh)|||$, the fact 
$
\normGh{\dsigma\normal}{}{}\lesssim  h^{1/2}\normDh{\dsigma}{}$,$\normGh{\du}{}{}\lesssim  h^{1/2}\normDh{\du}{}$, the result in Theorem \ref{EAt3} and recalling the definition of $\Teta$ (cf. \eqref{Teta}), we obtain

\begin{align*}  
|\miT{}{u,h}|\lesssim  &  \left((Rh)^{2}\hspace{-0.05cm}+\hspace{-0.05cm}Rh\hspace{-0.05cm}+\hspace{-0.05cm}(Rh)^{1/2}+\hspace{-0.05cm}Rh\tau\hspace{-0.05cm}+\hspace{-0.05cm}Rh^{1/2}\right) 
\bignorm \normO{\miteta}{}\\
&+  R h^{1/2}\normDh{\erho}{}\normO{\miteta}{}
+ \Big( (Rh)^{1/2}+ Rh\Big)\Teta\normO{\miteta}{}.
\end{align*}
The results follows noticing that $R$ and $h$ are bounded above.

\end{proof}

\subsection{Conclusion of the proof of estimate (\ref{Theo1b}) in Theorem \ref{EAt3}}

\begin{proof}Taking $\miteta=\eu$ in Proposition \ref{DAp1}, we can write 
\begin{align*}
\normDh{\eu}{2}=  &\ipt{\miA\esigma}{\dpsi}+\ipt{\erho}{\dpsi}+\ipt{\esigma}{\dxi}-\ipt{\dsigma}{\piA\mixi}-\ipt{\miA\dsigma}{\pivD\mipsi}
+\ipt{\drho}{\pivD\mipsi}+\miT{}{u,h}\\
= &\ipt{\miA\esigma}{\dpsi}+\ipt{\erho}{\dpsi}+\ipt{\esigma}{\dxi}+\ipt{\dsigma}{\dxi}+\ipt{\miA\dsigma}{\dpsi}\\
&+\ipt{\drho}{\dpsi}
-\ipt{\dsigma}{\mixi}-\ipt{\miA\dsigma}{\mipsi}-\ipt{\drho}{\mipsi}+\miT{}{u,h}.
\end{align*}
Using equation \eqref{DAeqb}, and the fact that $\drho$ is antisymmetric and $\mipsi$ is symmetric, we have $\ipt{\drho}{\mipsi}=0.$ Next, note that 
$$\ipt{\miA\dsigma}{\mipsi}+\ipt{\dsigma}{\mixi}=\ipt{\dsigma}{\miA\mipsi+\mixi}=\ipt{\dsigma}{\nabla\miphi}.$$
Then, by the property \eqref{Proa} with $\miv := \pcero\nabla\miphi$ (since $k\geq 1$), we have  $\ipt{\dsigma}{\pcero\nabla\miphi}=0$. Here, $\pcero$ is the $L^{2}$ projection onto $\pcero(K)$ on each $K \in \Trian$, then 
\begin{eqnarray*}
\normDh{\eu}{2} &=& \ipt{\miA\esigma}{\dpsi}+\ipt{\erho}{\dpsi}+\ipt{\esigma}{\dxi}+\ipt{\dsigma}{\dxi}
+\ipt{\miA\dsigma}{\dpsi}+\ipt{\drho}{\dpsi}-\ipt{\dsigma}{\nabla\miphi}+\miT{}{u,h}\\
&=&\ipt{\miA\esigma}{\dpsi}+\ipt{\erho}{\dpsi}+\ipt{\esigma}{\dxi}+\ipt{\dsigma}{\dxi}\\
&&+\ipt{\miA\dsigma}{\dpsi}+\ipt{\drho}{\dpsi}-\ipt{\dsigma}{\nabla\miphi-\pcero\nabla\miphi}+\miT{}{u,h}.
\end{eqnarray*}
Applying the Cauchy-Schwarz inequality, we obtain
\begin{align*}
\normDh{\eu}{2}\lesssim& \left( \normDh{\esigma}{}+\normDh{\erho}{}+\normDh{\dsigma}{}+\normDh{\drho}{}\right)
\left(\normO{\dpsi}{}+\normO{\dxi}{}+\normO{\nabla\miphi-P_{0}\nabla\miphi}{} \right)+|\miT{}{u,h}|. 
\end{align*}
we note that, by \eqref{EProa}
$ \normO{\dxi}{}\lesssim h|\mixi|_{H^{1}(\Omega)}$ and $\normO{\dpsi}{}\lesssim h|\mipsi|_{H^{1}(\Omega)}$.
Then, considering \eqref{DAeqd} and using Lemma \ref{DAl2} with $\miteta=\eu$, we have
\begin{align*}
\normDh{\eu}{2}\lesssim & h\left( \normDh{\esigma}{}+\normDh{\erho}{}+\normDh{\dsigma}{}+\normDh{\drho}{}\right)(|\mipsi|_{H^{1}(\Omega)}+|\mixi|_{H^{1}(\Omega)})+|\miT{}{u,h}|\\
\lesssim  & h\left( \normDh{\esigma}{}\hspace{-0.1cm}+\hspace{-0.05cm}\normDh{\erho}{}\hspace{-0.05cm}+\hspace{-0.05cm}\normDh{\dsigma}{}\hspace{-0.05cm}+\hspace{-0.05cm}\normDh{\drho}{}\right)\normDh{\eu}{}
+  (Rh)^{1/2} \left(1+\tau\right) 
\bignorm \normDh{\eu}{}\\
& +   h^{1/2} R\normDh{\erho}{}\normDh{\eu}{}
+\Teta\normDh{\eu}{}.
\end{align*}

Finally by Theorem \ref{EAt3} and recalling the definition of $\Teta$ (cf. \eqref{Teta}), we obtain the result in (\ref{Theo1b}). Moreover, if $k\geq 1$, the estimate of $\norm{\eug}{h}{}$ follows from standard arguments in HDG (see \cite{diffusion}, for instance).
\end{proof}

\section{Numerical experiments}\label{ChapNE}
In this section we present numerical experiments for HDG method \eqref{sq} in the two-dimensional case. For all the computations we consider the spaces specified in \eqref{spaces} with $k \in \lbrace 1,2,3 \rbrace $ and  the exact solution $\miuc=(u_1,u_2)^t$, with $u_1= \sin(\pi x)\cos(\pi y)$ and $u_2=\cos(\pi x)\sin(\pi y)$. We fix $E=1$ and take $\nu \in \lbrace 0.3,0.4999 \rbrace$ in order to see the effect of the nearly incompressible case.  The stabilization parameter $\tau$ is set to be one. According to Corollary \ref{EAt3cor}, the theoretical order of convergence for the $L^{2}$-norm of the errors in $\miuc$, $\misigmac$ and $\miroc$ is $k+1$, whereas for the numerical trace is $k+2$ if $R\approx h$ and $k+3/2$ if $R\approx 1$.

\paragraph{Example 1.}
We consider the domain as $\Omega := \lbrace (x,y)\in \mathbb{R}^{2}: x^{2}+y^{2}\leq 1 \rbrace$ and the computational domain is constructed by linearly interpolating the boundary of $\Omega$. In this case, $R$ is of order $h$, then the set of Assumptions C is satisfied for $h$ small enough even in the nearly incompressible case. Figures \ref{fig:Ex1} and \ref{fig:Ex1b} display the behavior of the errors when $\nu =0.3$  and $\nu=0.4999$, respectively. We observe that the $L^{2}$-errors of $\miuc$, $\misigmac$ and $\miroc$  behave as Corollary \ref{EAt3cor} predicts, that is, an order of convergence of $k+1$. The approximation $\miug$ converges to the trace of the solution with order $k+1$, which is half a power higher than the one predicted. in addition, we notice that the magnitude of the errors is larger when $\nu=0.4999$, however the rates of convergence are the same as in the case when $\nu=0.3$, indicating that the method is optimal even in the nearly incompressible case.

\begin{figure}[ht!]
  \centering
\includegraphics[width=0.3\textwidth]{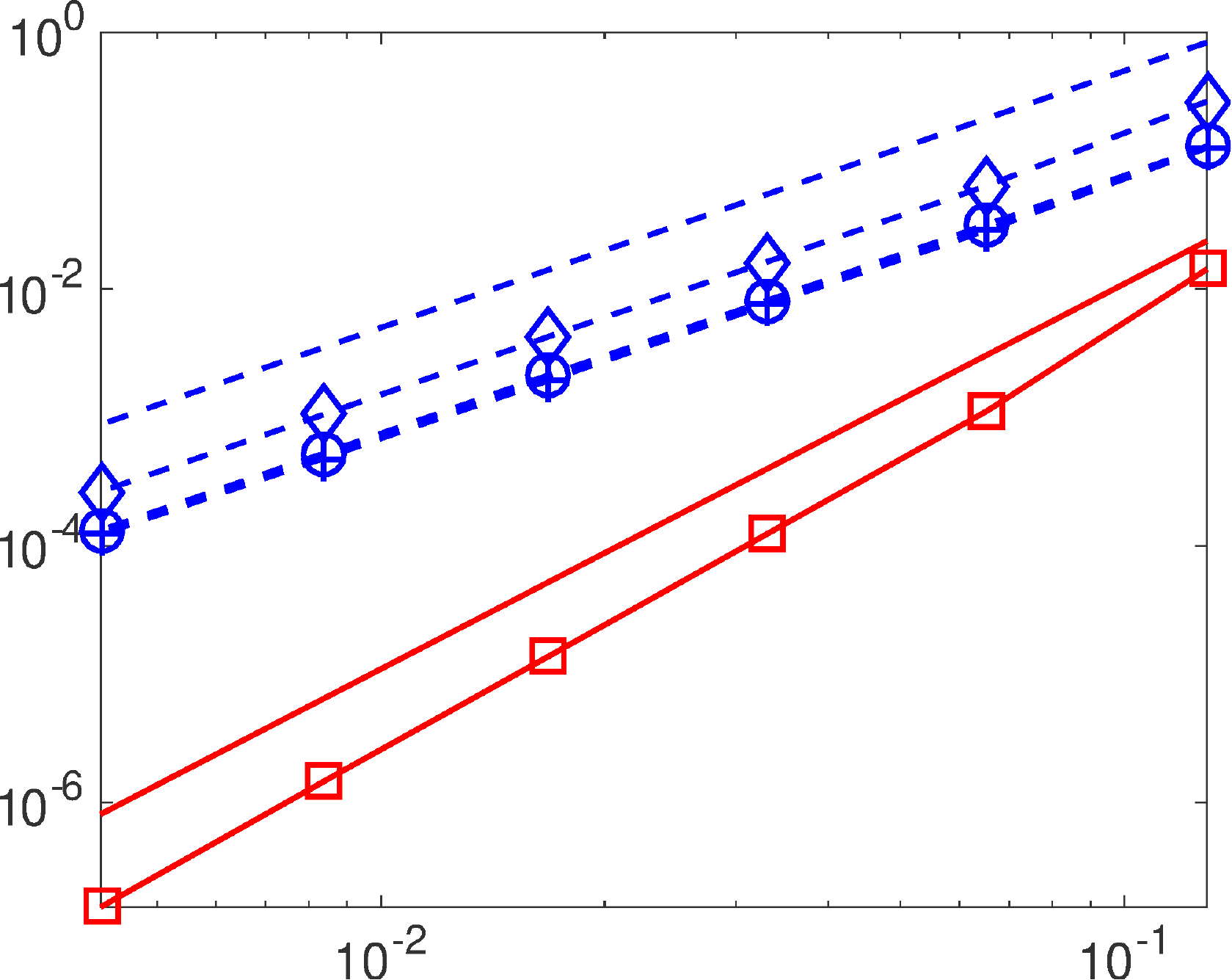}
\includegraphics[width=0.3\textwidth]{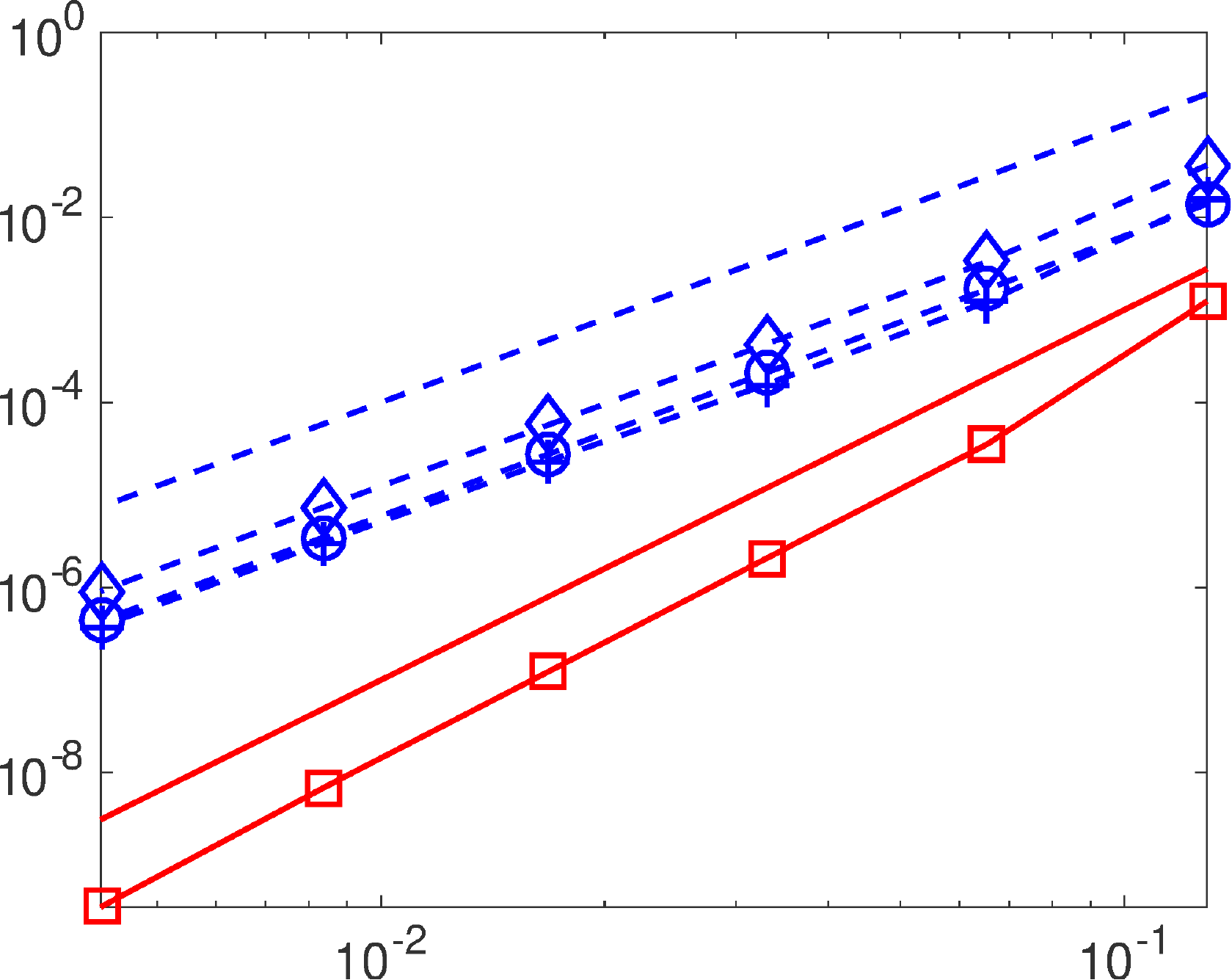}\\
\includegraphics[width=0.3\textwidth]{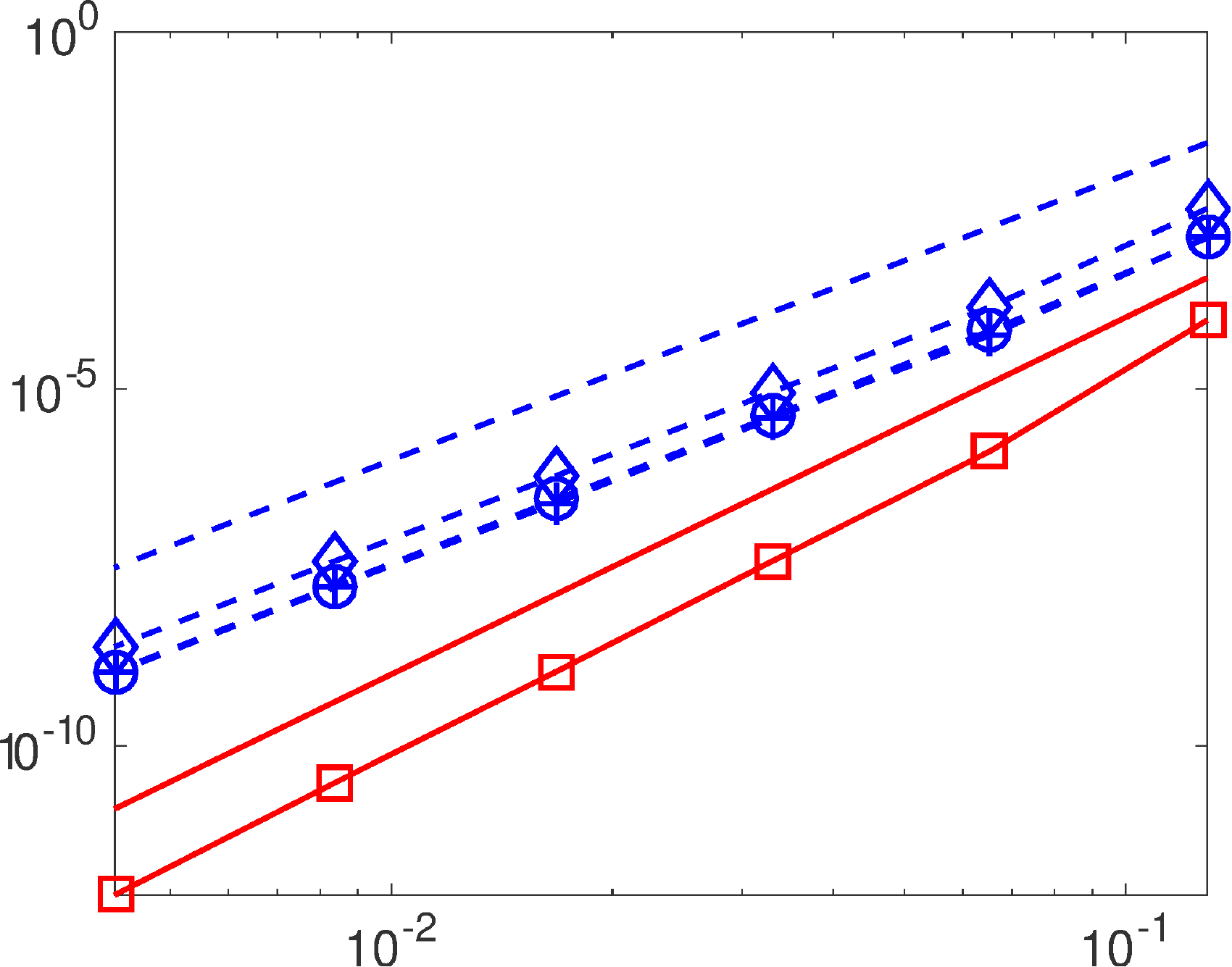}
  \caption{Errors in Example 1 with $\nu = 0.3$, $k=1$ (top-left), $k=2$ (top-right) and $k=3$ (bottom). The ordinates indicates $\normDh{\misigmac-\misigmah}{}$ (dashed-blue line $\diamondsuit$), $\normDh{\miroc-\miroh}{}$ (dashed-blue line $+$), $\normDh{\miuc-\miuh}{}$ (dashed-blue line $\circ$) and $\norm{\proMD\miuc-\miug}{h}{}$ (solid-red line $\square$). The abscissas correspond to $N_{elem}^{-1/2}$. The dashed-blue and solid-red lines indicate the slope $k+1$ and $k+2$, resp.}\label{fig:Ex1}
\end{figure}

\begin{figure}[ht!]
  \centering
\includegraphics[width=0.3\textwidth]{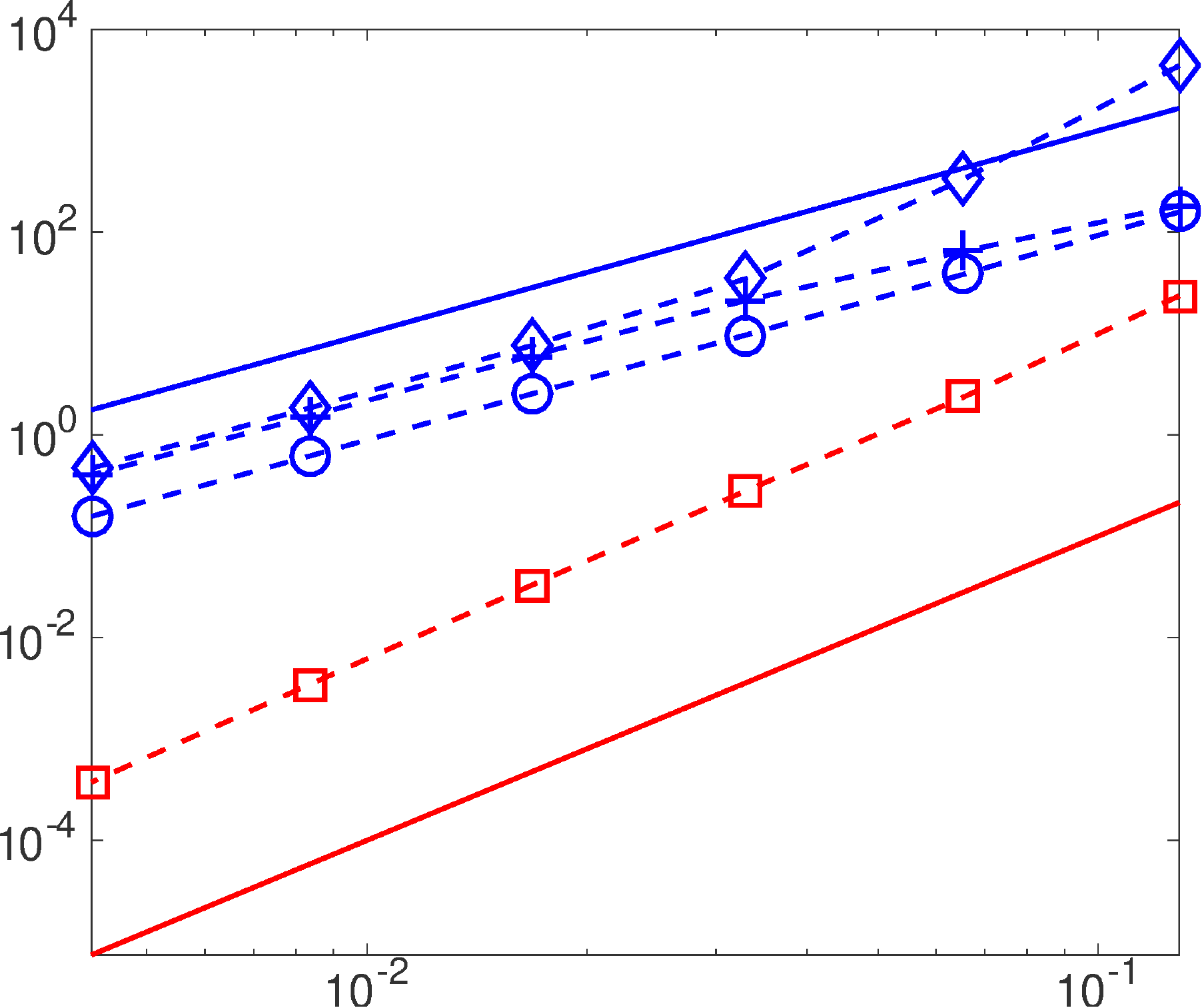}
\includegraphics[width=0.3\textwidth]{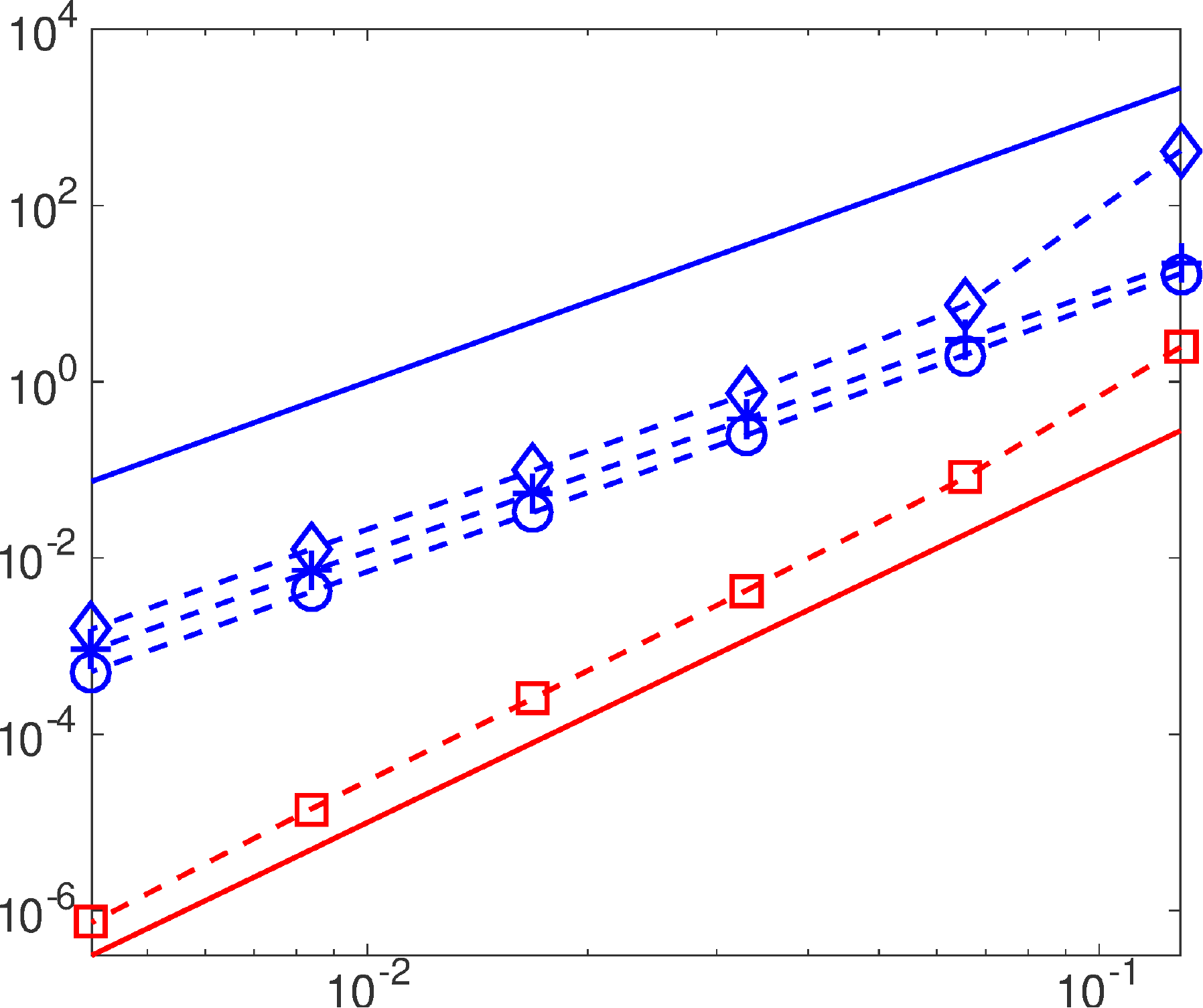}\\
\includegraphics[width=0.3\textwidth]{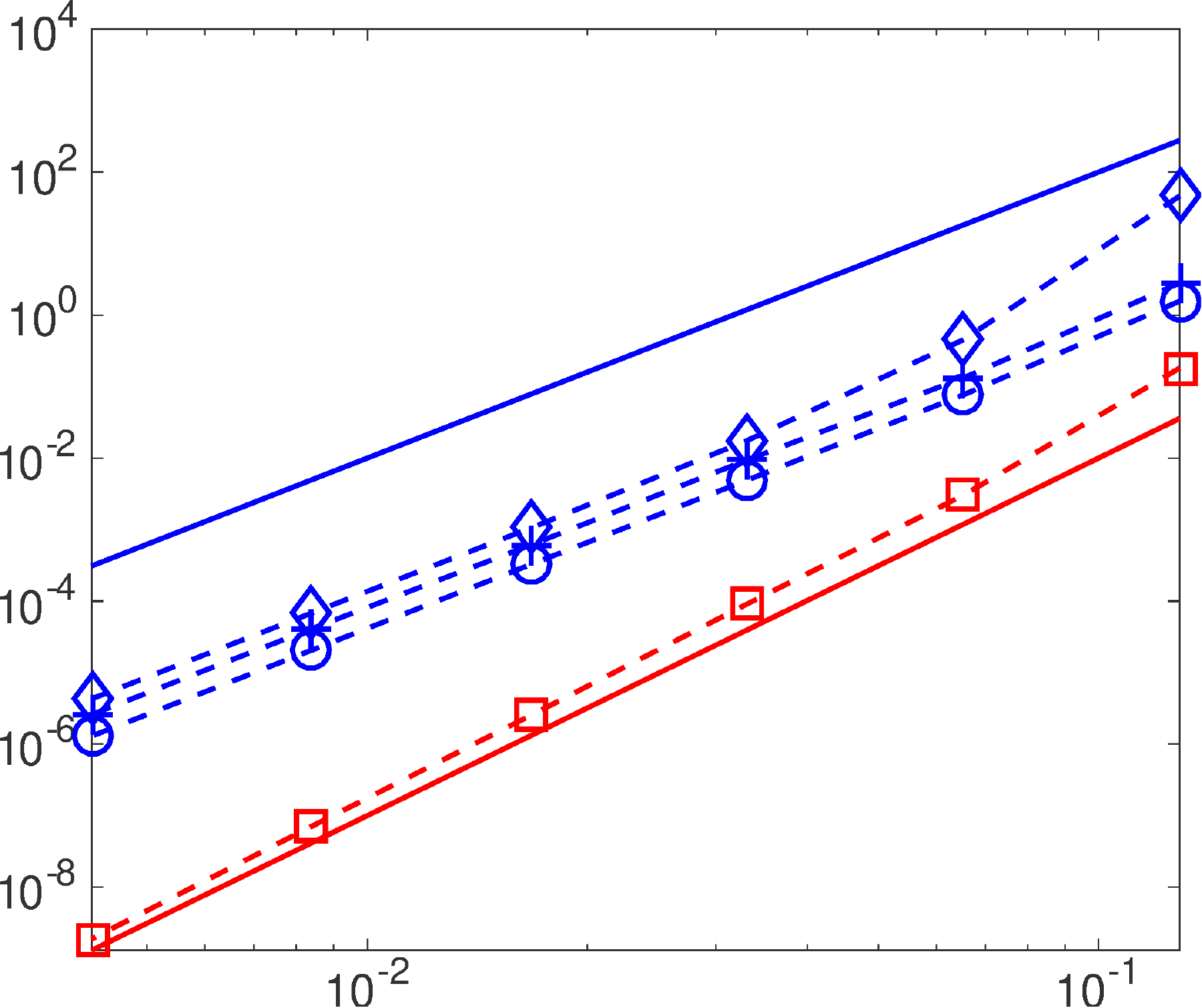}
  \caption{Errors in Example 1 with $\nu = 0.4999$. The legend is the same as in Figure \ref{fig:Ex1}}\label{fig:Ex1b}
\end{figure}

In the following set of examples we  construct the computational domain and transferring path according to the procedure described in Section 2 in \cite{CS}. Roughly speaking, $\Omega$ is immersed in a background mesh and the computational domain $\Omega_h$ is the union of all the elements in the background triangulation completely inside $\Omega$. In addition, the transferring paths are constructed using the algorithm in Section 2.4.1  in \cite{CS} that ensures that  $\ex$ and $\exb$ are as close as possible,  two transferring paths do not intersect each other before terminating at $\Gamma$ and they do not intersect the interior of the computational domain $\Omega_{h}$. For a kidney-shaped domain, Figures \ref{fig:Tpa_a} and \ref{fig:Tpa_b} show the  computational domain (gray) and transferring paths constructed by the procedure just mentioned.  In this case, $dist(\Gamma_h,\Gamma)$ is of order $h$ and hence $R$ is of order one. Then Assumptions C hold for $R$ small enough; however, we cannot  control how small $R$ is.

\paragraph{Example 2.}
We consider the same domain as in Example 1. For $\nu =  0.3 $ and $0.4999 $, we depict in Figures  \ref{fig:Ex2a} and \ref{fig:Ex2b}, respectively,  the behavior of the errors. Even though it look more erratic for some meshes, it seems that asymptotically is decaying with optimal rate.

\begin{figure}[ht!]
  \centering
\includegraphics[width=0.3\textwidth]{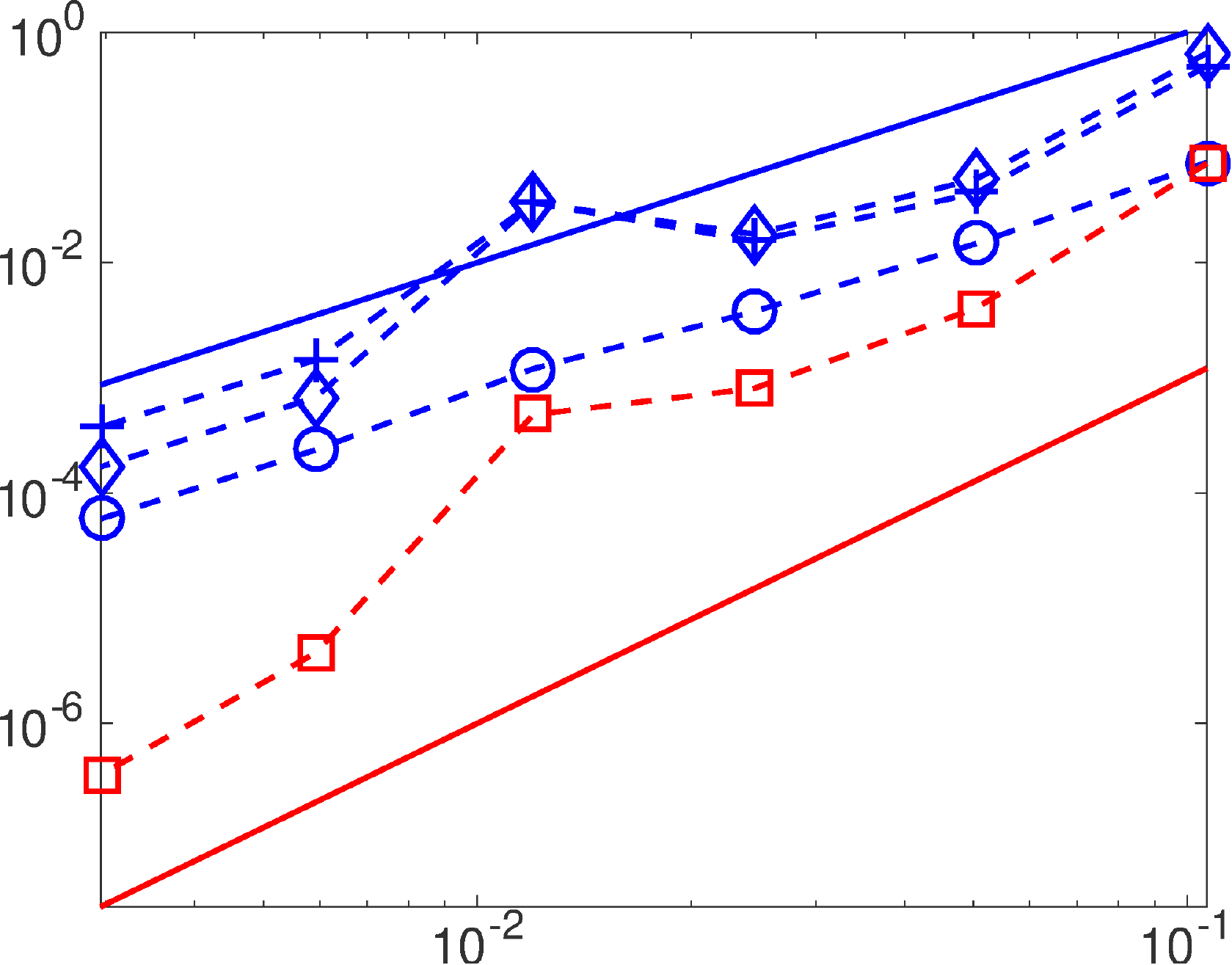}
\includegraphics[width=0.3\textwidth]{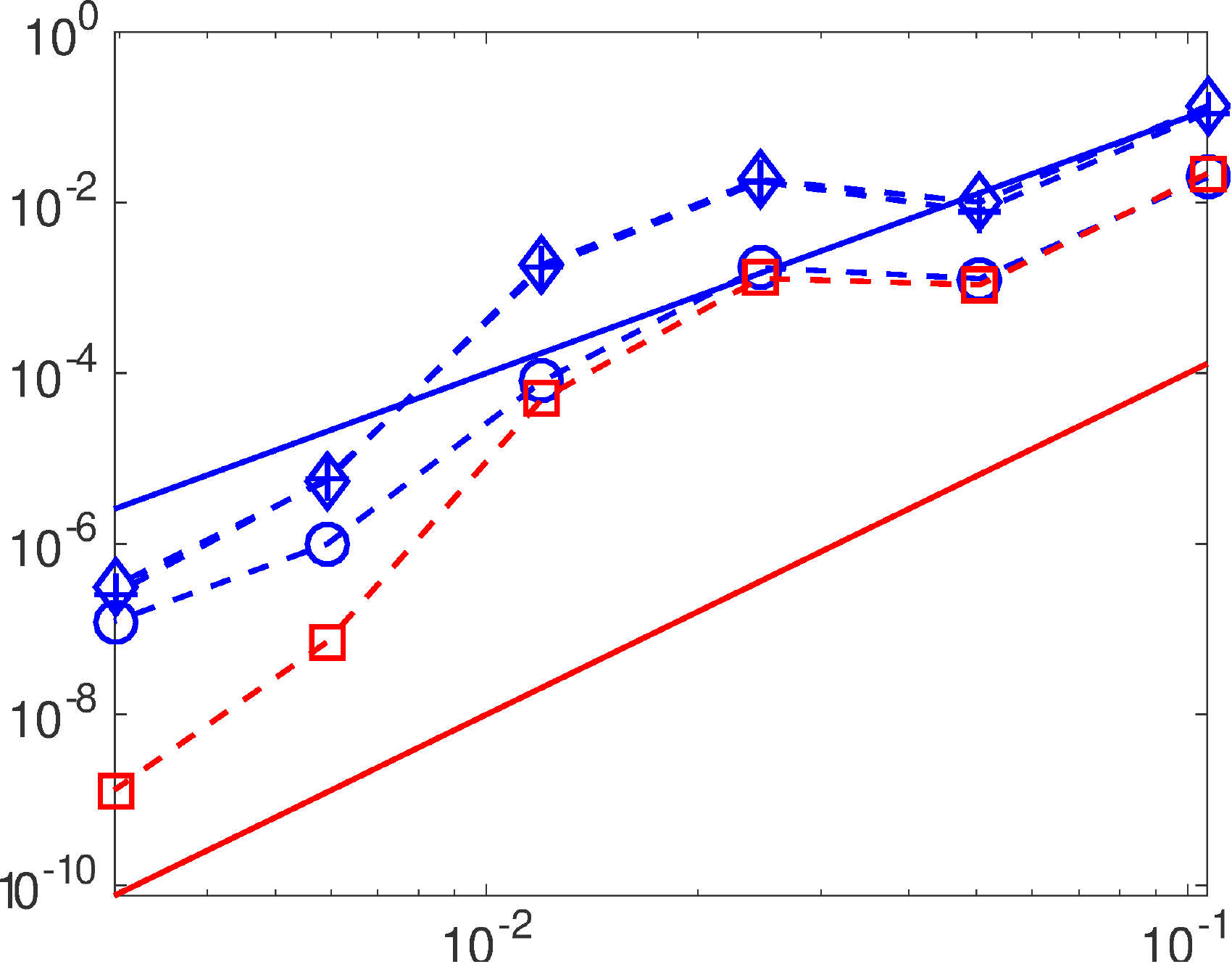}\\
\includegraphics[width=0.3\textwidth]{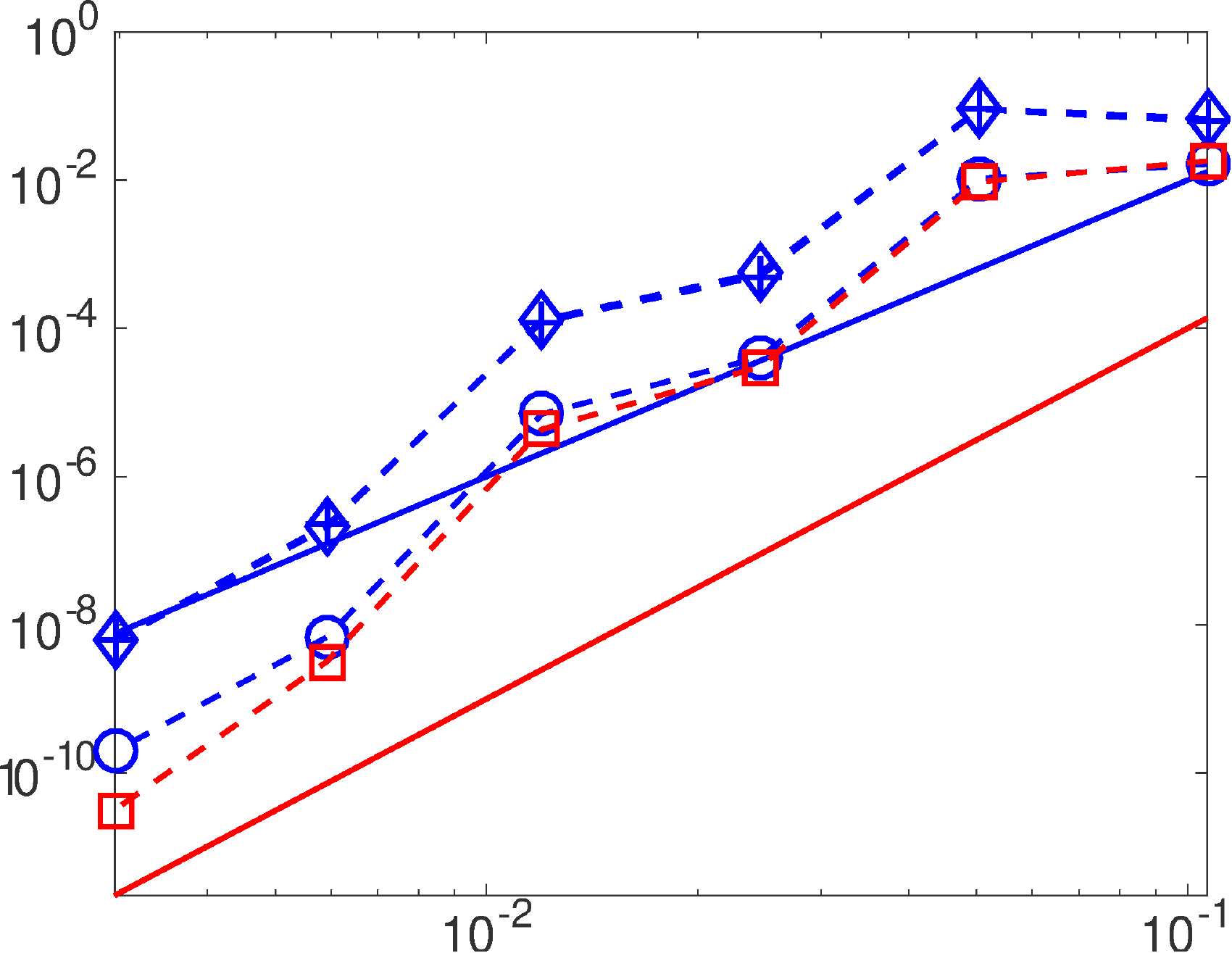}
  \caption{Errors in Example 2 with $\nu = 0.3$. The legend is the same as in Figure \ref{fig:Ex1}.}\label{fig:Ex2a}
\end{figure}

\begin{figure}[ht!]
  \centering
\includegraphics[width=0.3\textwidth]{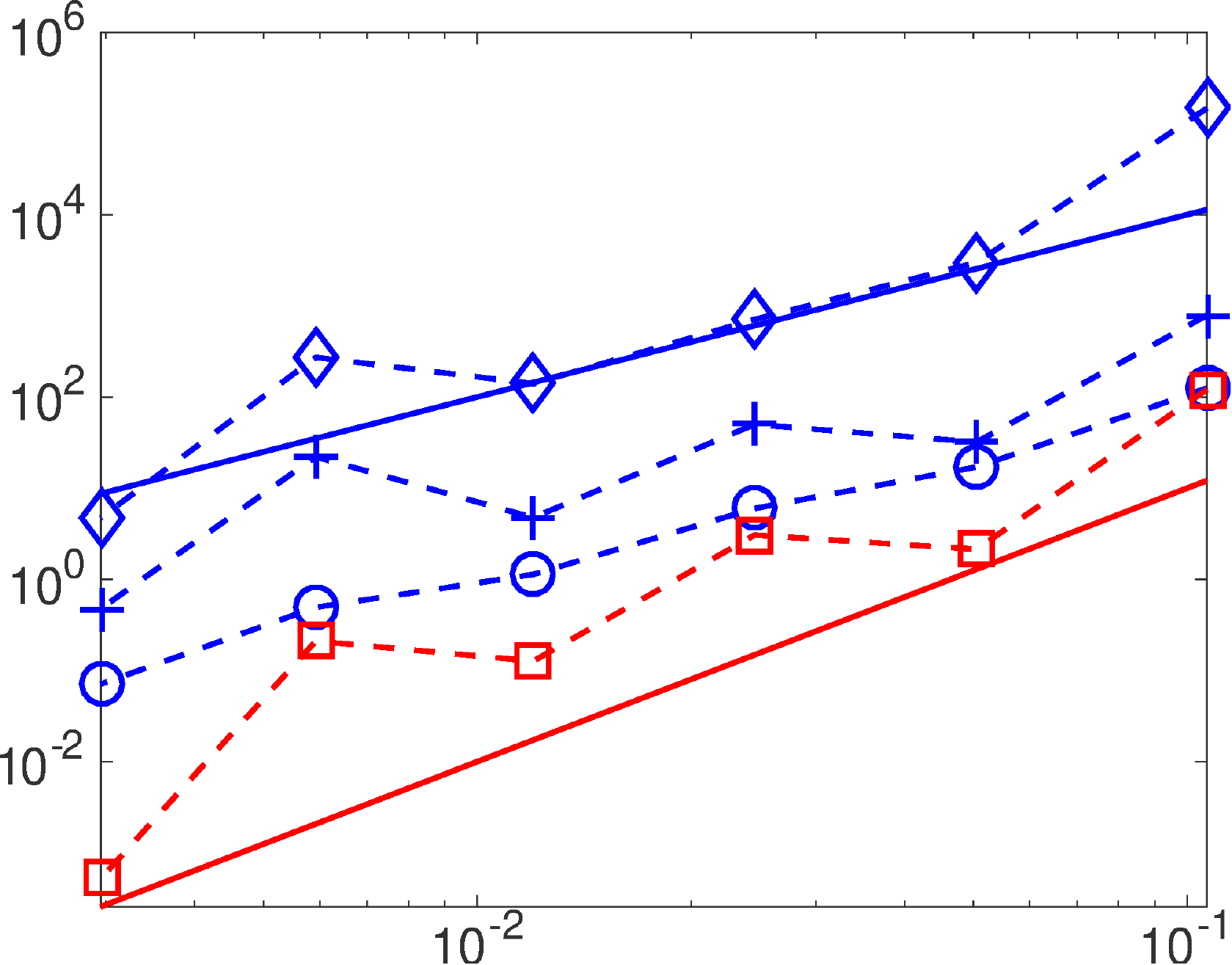}
\includegraphics[width=0.3\textwidth]{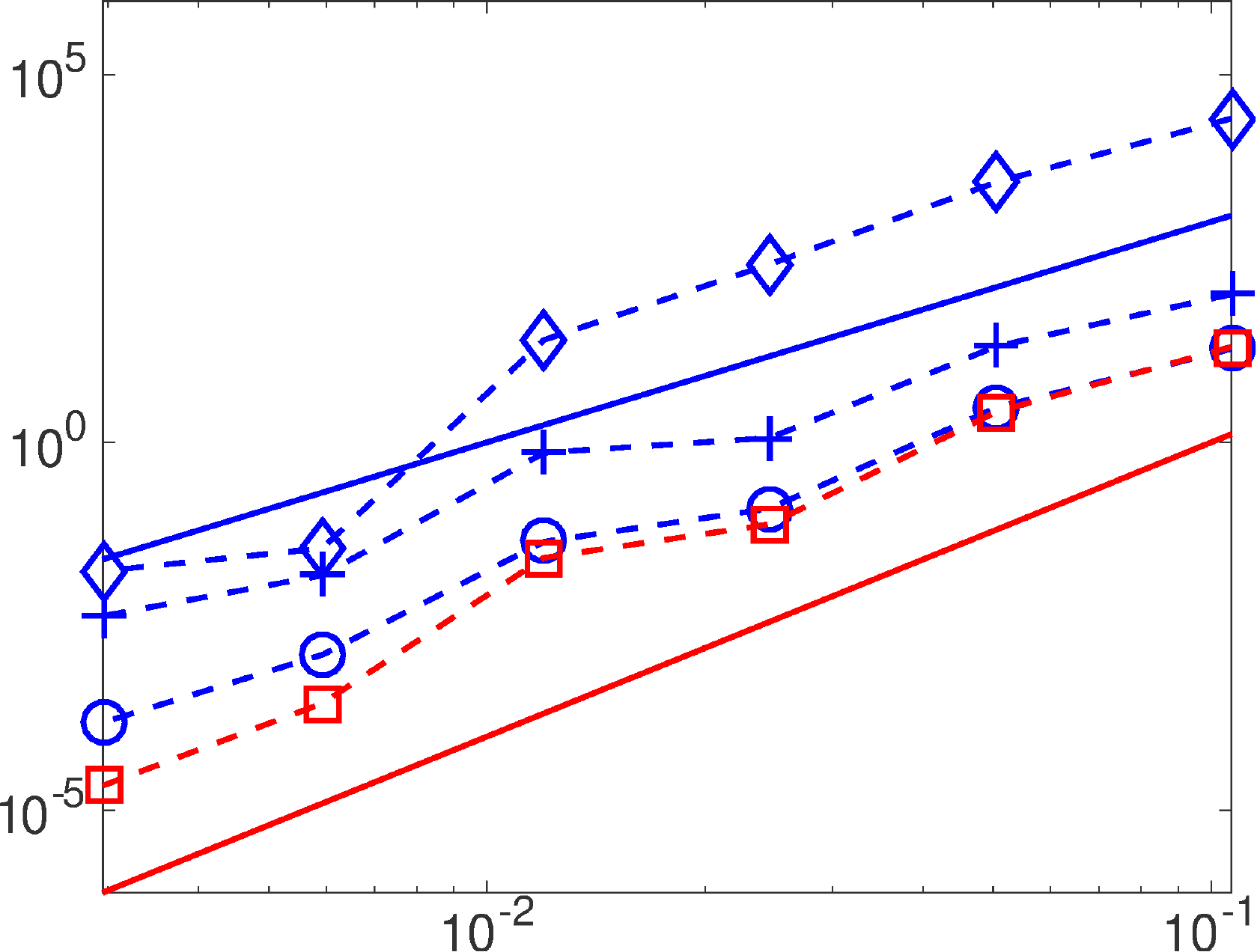}\\
\includegraphics[width=0.3\textwidth]{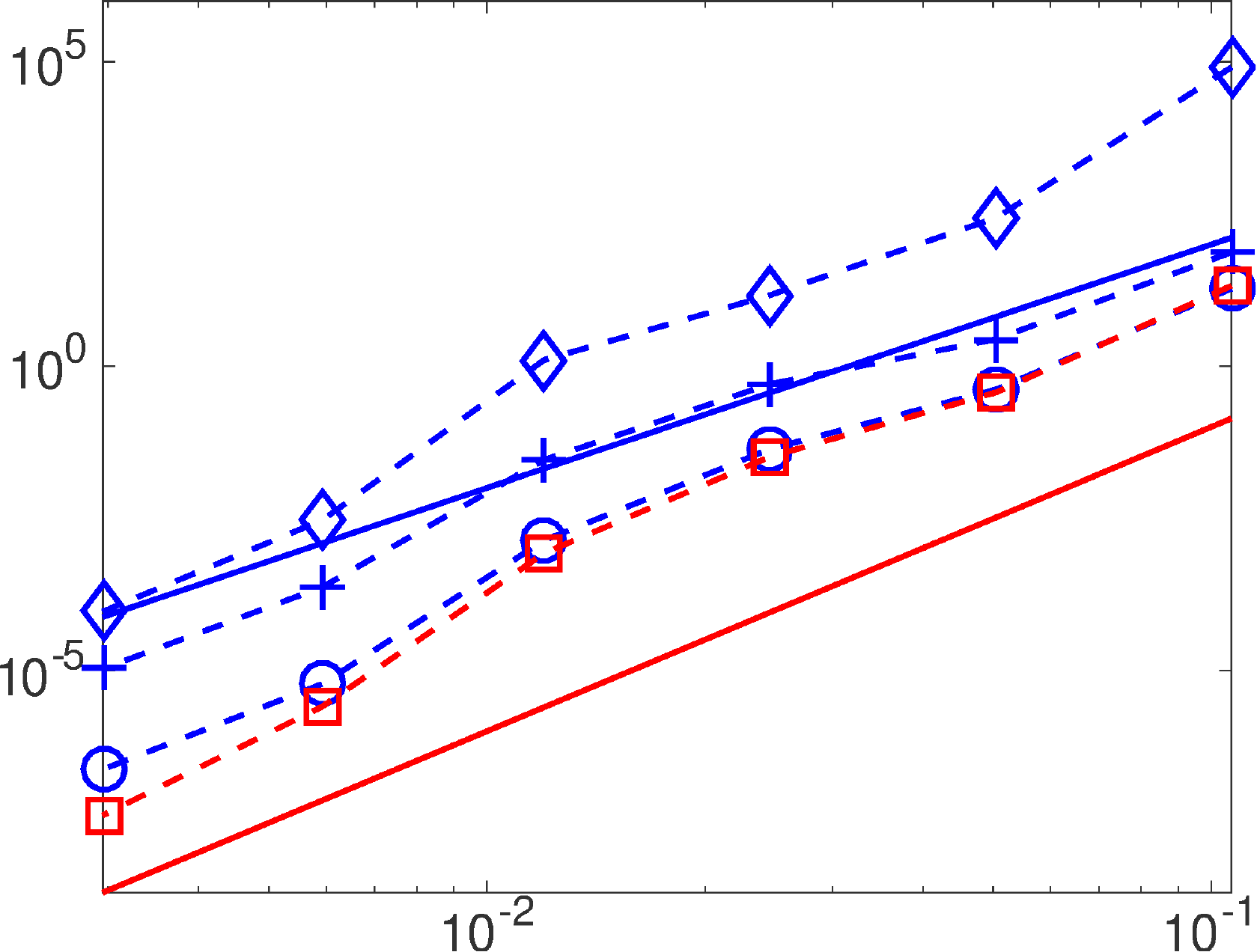}
  \caption{Errors in Example 2 with $\nu = 0.4999$. The legend is the same as in Figure \ref{fig:Ex1}.}\label{fig:Ex2b}
\end{figure}

\paragraph{Example 3.}
We consider level set 
$2\left((x+(1/2))^{2}+y^{2})-x-(1/2)\right)^{2}-\left((x+(1/2)^{2}+y^{2})\right)+0.1=0$ that defines a non-convex domain. In Figures \ref{fig:Ex3a} and \ref{fig:Ex3b} we display the how the errors decays when the meshsize decreases. Similar convulsions as in Example 2 hold.

\begin{figure}[ht!]
  \centering
\includegraphics[width=0.3\textwidth]{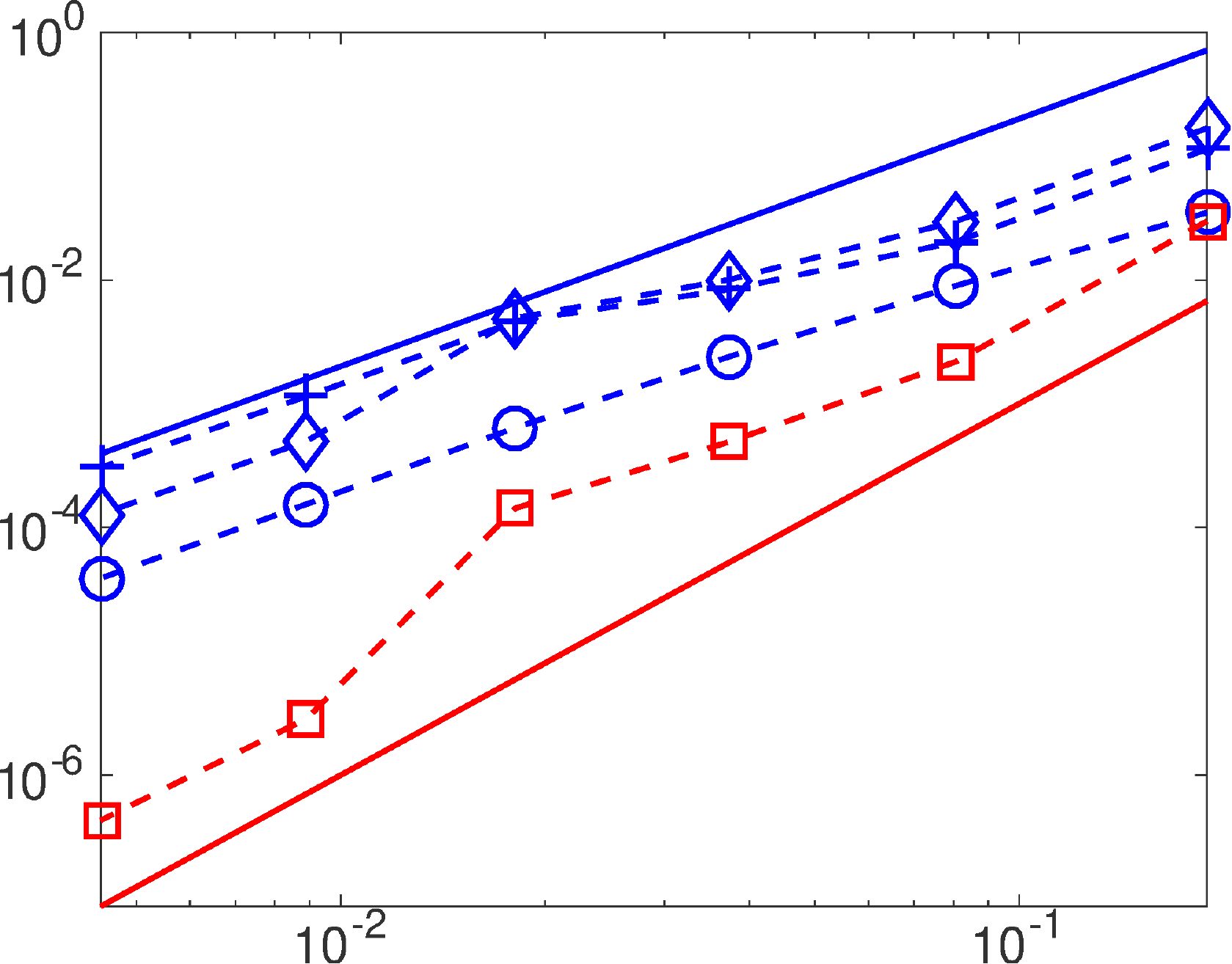}
\includegraphics[width=0.3\textwidth]{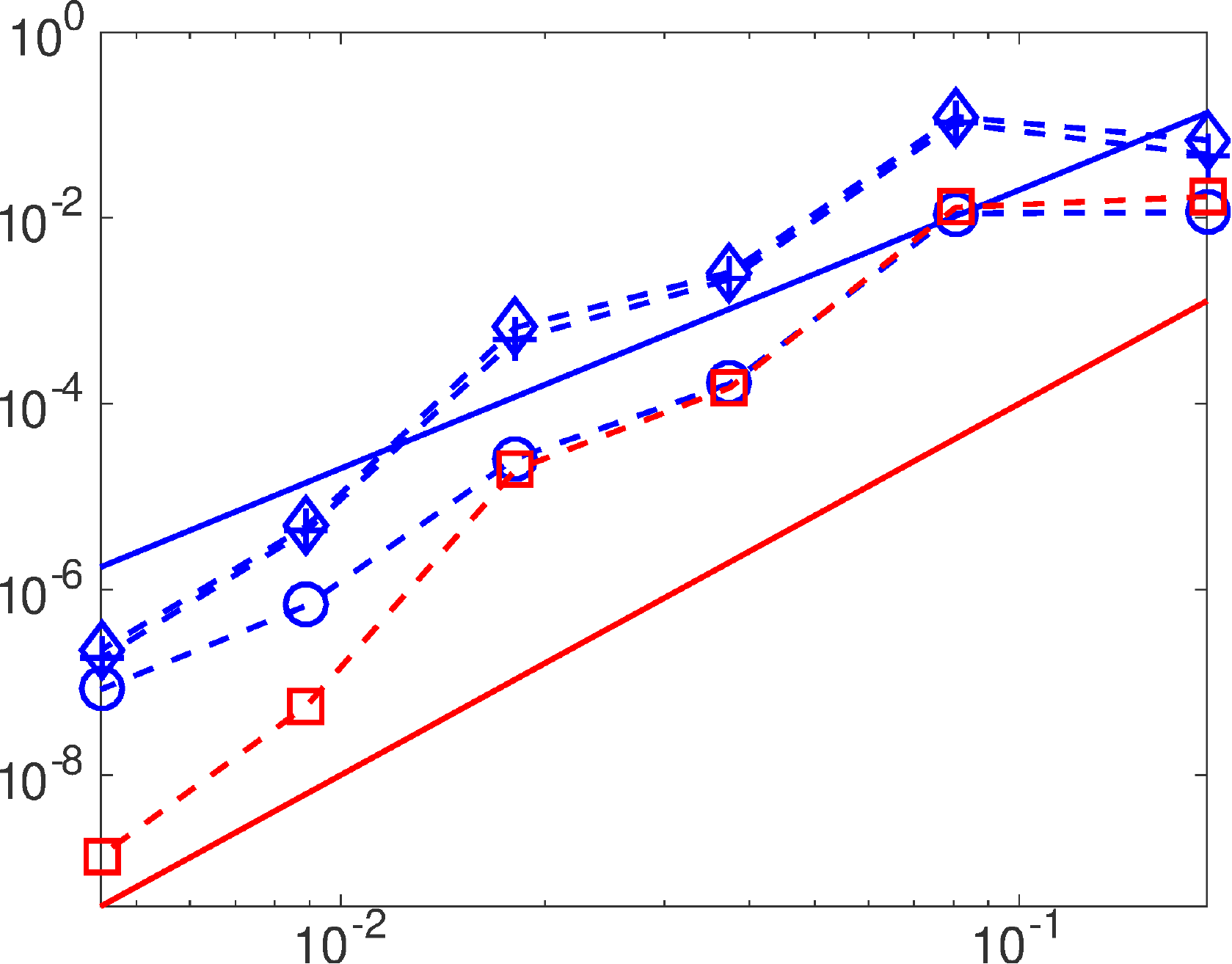}\\
\includegraphics[width=0.3\textwidth]{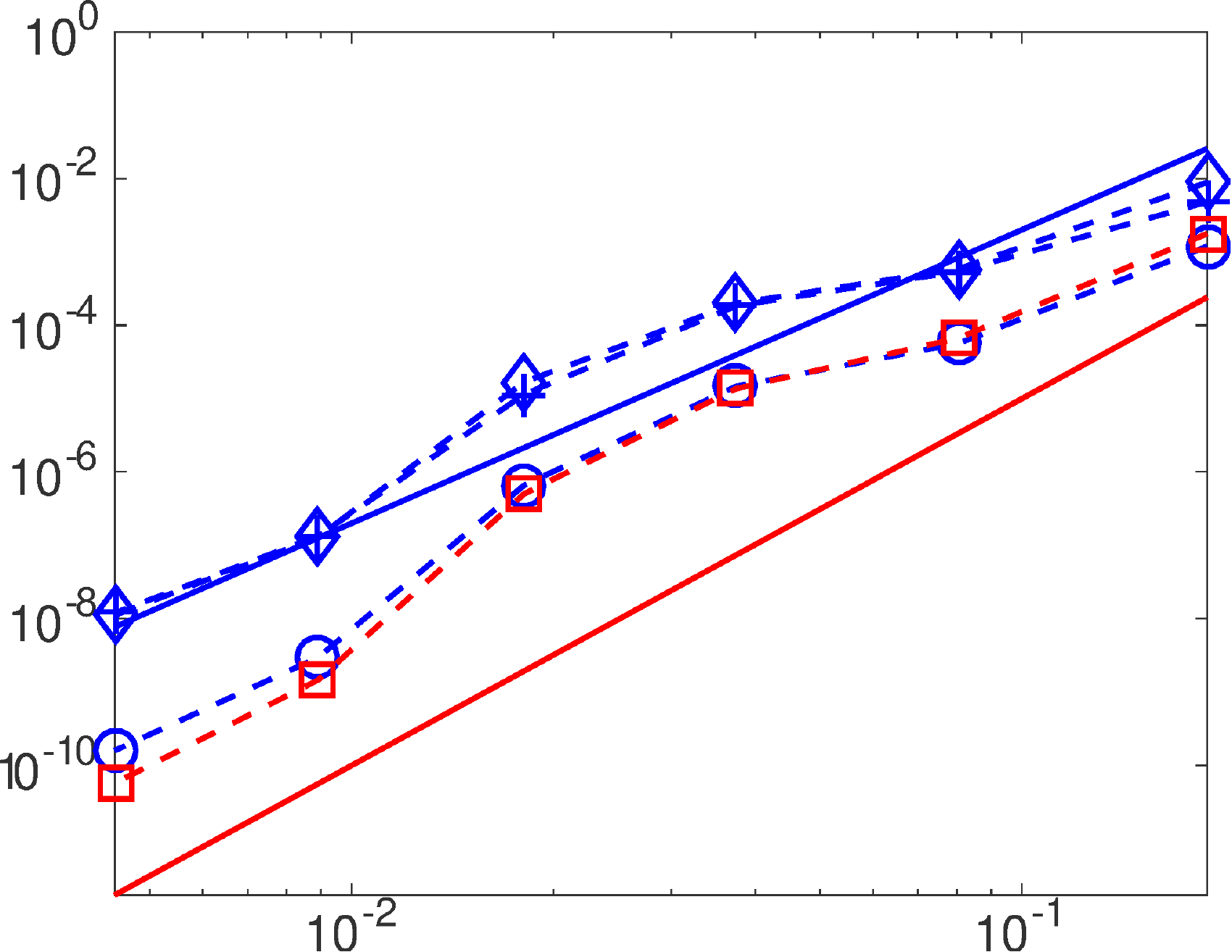}
  \caption{Errors in Example 3 with $\nu = 0.3$. The legend is the same as in Figure \ref{fig:Ex1}.}\label{fig:Ex3a}
\end{figure}

\begin{figure}[ht!]
  \centering
\includegraphics[width=0.3\textwidth]{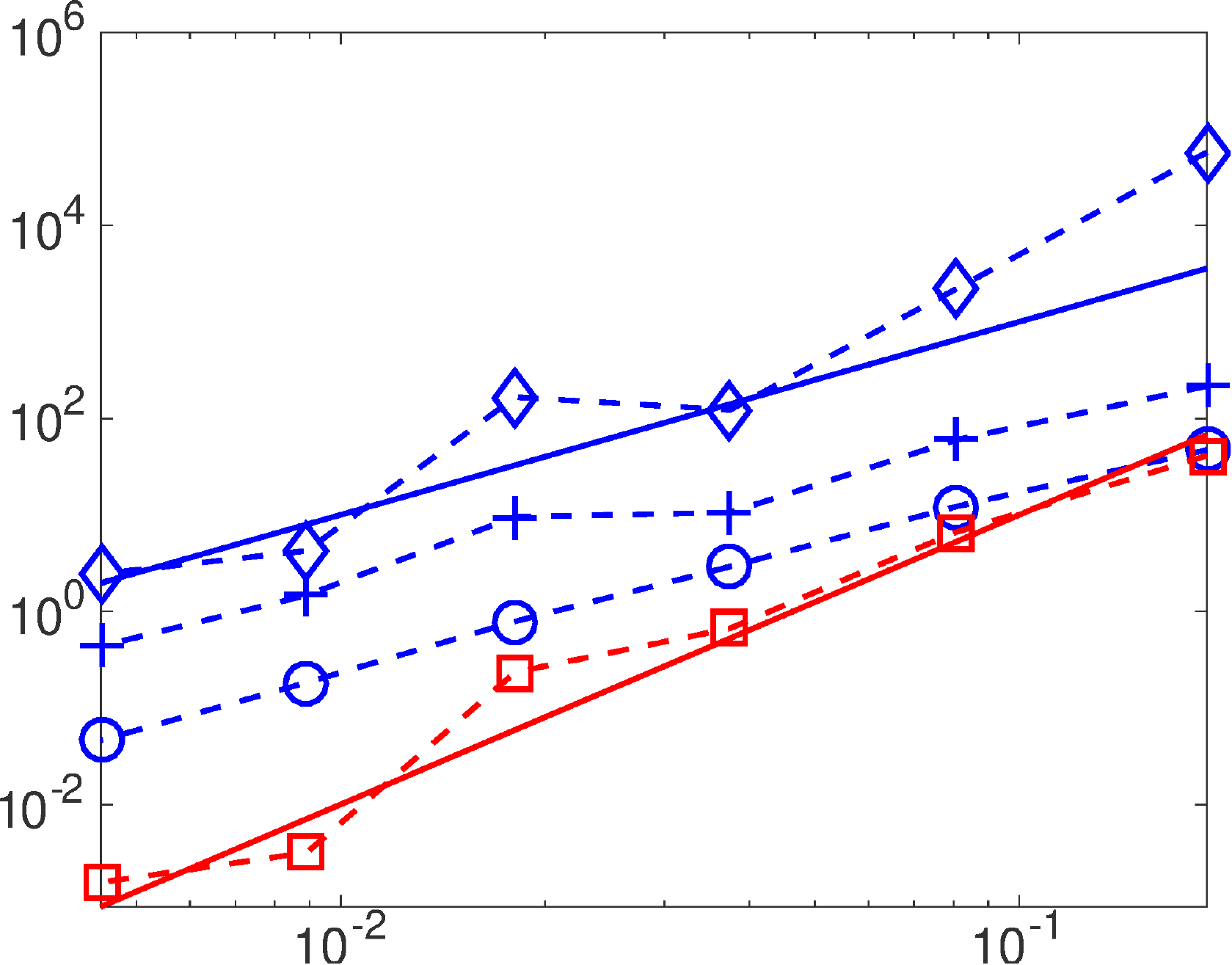}
\includegraphics[width=0.3\textwidth]{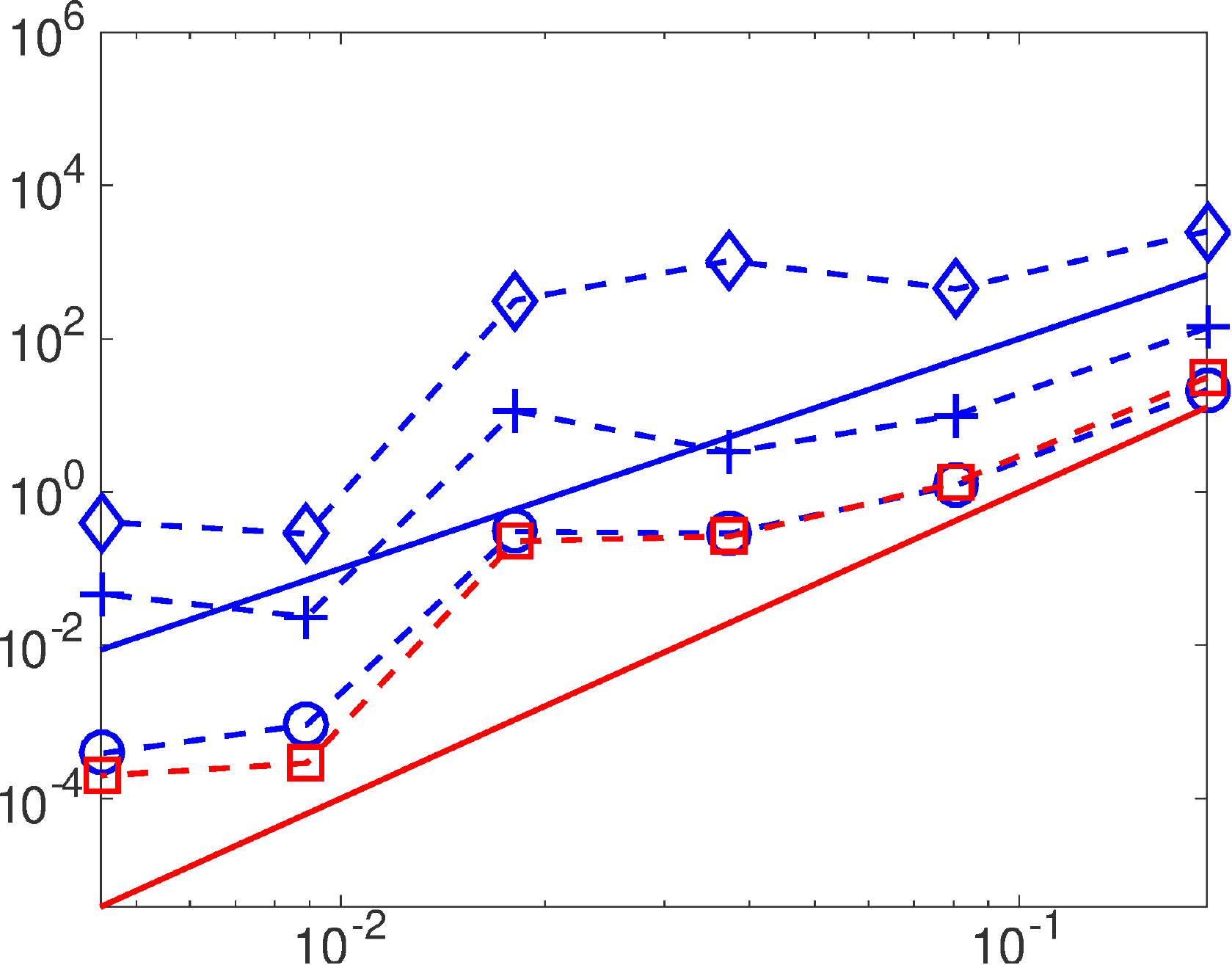}\\
\includegraphics[width=0.3\textwidth]{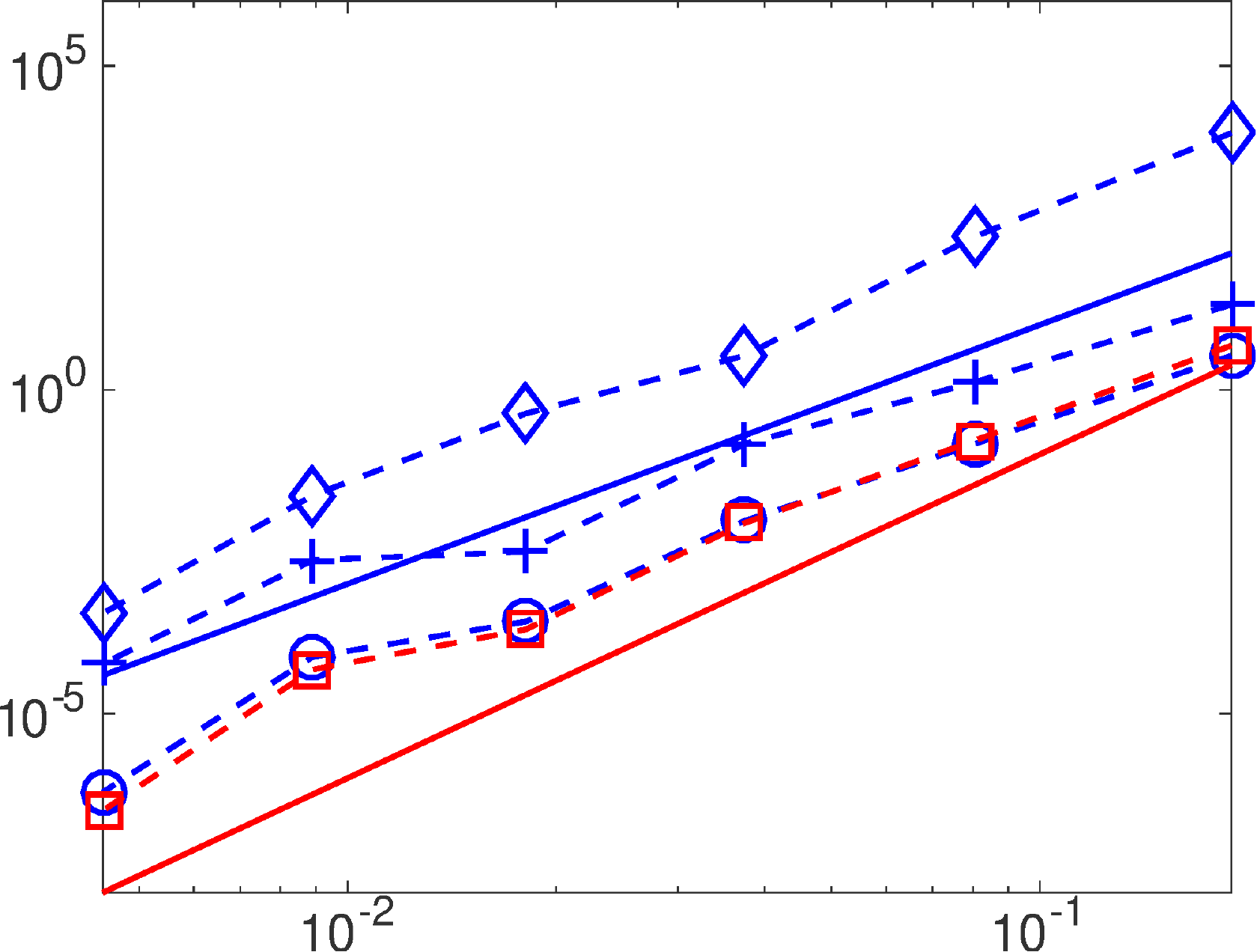}
  \caption{Errors in Example 3 with $\nu = 0.4999$. The legend is the same as in Figure \ref{fig:Ex1}.}\label{fig:Ex3b}
\end{figure}

\section*{Acknowledgements}
This work was supported by ANID--Chile through Fondecyt 1200569 and by Centro de Modelamiento Matemático (CMM), ACE210010 and FB210005, BASAL funds for center of excellence from ANID-Chile

\bibliographystyle{acm}
\bibliography{references}

\begin{thebibliography}{10}

\bibitem{Scovazzi3}
{\sc Atallah, N.~M., Canuto, C., and Scovazzi, G.}
\newblock The shifted boundary method for solid mechanics.
\newblock {\em International Journal for Numerical Methods in Engineering 122},
  20 (2021), 5935--5970.

\bibitem{Dual}
{\sc Bacuta, C., and Bramble, J.~H.}
\newblock Regularity estimates for solutions of the equations of linear
  elasticity in convex plane polygonal domains.
\newblock {\em Zeitschrift f{\"u}r angewandte Mathematik und Physik ZAMP 54}, 5
  (Sep 2003), 874--878.

\bibitem{CaSo2021}
{\sc Camargo, L., and Solano, M.}
\newblock A high order unfitted hdg method for the helmholtz equation with
  first order absorbing boundary condition.
\newblock {\em Preprint 2021-027, Centro de Investigaci\'on en Ingenier\'ia
  Matem\'atica (CI$^{2}$MA), Universidad de Concepci\'on, Chile\/} (2021).

\bibitem{xHDG2}
{\sc Ceren~Gürkan, M.~K., and Fernández-Méndez, S.}
\newblock {e}{X}tended hybridizable discontinuous {G}alerkin with heaviside
  enrichment for heat bimaterial problems.
\newblock {\em Journal of Scientific Computing 72\/} (2017), 542–567.

\bibitem{xHDG1}
{\sc Ceren~Gürkan, Esther Sala-Lardies, M.~K., and Fernández-Méndez, S.}
\newblock {eX}tended hybridizable discontinous {G}alerkin {(X-HDG)} for void
  problems.
\newblock {\em Journal of Scientific Computing 66\/} (2016), 1313–1333.

\bibitem{CoFu2017}
{\sc Cockburn, B., and Fu, G.}
\newblock {Devising superconvergent HDG methods with symmetric approximate
  stresses for linear elasticity by {M}-decompositions}.
\newblock {\em IMA Journal of Numerical Analysis 38}, 2 (06 2017), 566--604.

\bibitem{diffusion}
{\sc Cockburn, B., Gopalakrishnan, J., and Sayas, F.}
\newblock A projection-based error analysis of {HDG} methods.
\newblock {\em Math. Comput. 79}, 271 (2010), 1351--1367.

\bibitem{Cockburn9}
{\sc Cockburn, B., Gupta, D., and Reitich, F.}
\newblock Boundary-conforming discontinuous {G}alerkin methods via extensions
  from subdomains.
\newblock {\em J. Sci. Comput. 42}, 1 (2010), 144--184.

\bibitem{Soon}
{\sc Cockburn, B., Guzm{\'{a}}n, J., Soon, S., and Stolarski, H.~K.}
\newblock An analysis of the embedded discontinuous {G}alerkin method for
  second-order elliptic problems.
\newblock {\em {SIAM} J. Numerical Analysis 47}, 4 (2009), 2686--2707.

\bibitem{CQS}
{\sc Cockburn, B., Qiu, W., and Solano, M.}
\newblock A priori error analysis for {HDG} methods using extensions from
  subdomains to achieve boundary conformity.
\newblock {\em Math. Comput. 83}, 286 (2014), 665--699.

\bibitem{Cockburn10}
{\sc Cockburn, B., Sayas, F., and Solano, M.}
\newblock Coupling at a distance {HDG} and {BEM}.
\newblock {\em {SIAM} J. Scientific Computing 34}, 1 (2012), A28–A47.

\bibitem{Ke}
{\sc Cockburn, B., and Shi, K.}
\newblock {Superconvergent HDG methods for linear elasticity with weakly
  symmetric stresses}.
\newblock {\em IMA Journal of Numerical Analysis 33}, 3 (10 2012), 747--770.

\bibitem{CS}
{\sc Cockburn, B., and Solano, M.}
\newblock Solving {D}irichlet boundary-value problems on curved domains by
  extensions from subdomains.
\newblock {\em {SIAM} J. Scientific Computing 34}, 1 (2012), A28–A47.

\bibitem{CoSol2}
{\sc Cockburn, B., and Solano, M.}
\newblock Solving convection-diffusion problems on curved domains by extensions
  from subdomains.
\newblock {\em J. Sci. Comput. 59}, 2 (2014), 512--543.

\bibitem{DuSA2020}
{\sc Du, S., and Sayas, F.-J.}
\newblock New analytical tools for {HDG} in elasticity, with applications to
  elastodynamics.
\newblock {\em Mathematics of Computations 89\/} (2020), 1745--1782.

\bibitem{FuCoSt2015}
{\sc Fu, G., Cockburn, B., and Stolarski, H.}
\newblock Analysis of an hdg method for linear elasticity.
\newblock {\em International Journal for Numerical Methods in Engineering 102},
  3-4 (2015), 551--575.

\bibitem{Fu}
{\sc Fu, G., Lehrenfeld, C., Linke, A., and Streckenbach, T.}
\newblock Locking-free and gradient-robust {H}(div)-conforming {HDG} methods
  for linear elasticity.
\newblock {\em Journal of Scientific Computing 86}, 39 (2021), 427--454.

\bibitem{GR}
{\sc Girault, V., and Raviart, P.-A.}
\newblock {\em Finite element methods for {N}avier-{S}tokes equations. Theory
  and algorithms}, vol.~5 of {\em Springer Series in Computational
  Mathematics}.
\newblock Springer-Verlag, Berlin, 1986.

\bibitem{Guz}
{\sc Guzm{\'a}n, J.}
\newblock A unified analysis of several mixed methods for elasticity with weak
  stress symmetry.
\newblock {\em Journal of Scientific Computing 44}, 2 (Aug 2010), 156--169.

\bibitem{xHDG3}
{\sc Han, Y., Chen, H., Wang, X.-P., and Xie, X.}
\newblock Extended {HDG} methods for second order elliptic interface problems.
\newblock {\em ArXiv e-prints\/} (2019).
\newblock https://arxiv.org/abs/1910.09769.

\bibitem{han2021interfaceboundaryunfitted}
{\sc Han, Y., Wang, X.-P., and Xie, X.}
\newblock An interface/boundary-unfitted e{X}tended {HDG} method for linear
  elasticity problems.
\newblock {\em ArXiv e-prints\/} (2021).
\newblock https://arxiv.org/abs/2004.06275.

\bibitem{Lew1}
{\sc Lew, A.~J., and Buscaglia, G.~C.}
\newblock A discontinuous-{G}alerkin-based immersed boundary method.
\newblock {\em International Journal for Numerical Methods in Engineering 76},
  4 (2008), 427--454.

\bibitem{Scovazzi1}
{\sc Main, A., and Scovazzi, G.}
\newblock The shifted boundary method for embedded domain computations. {P}art
  {I}: {P}oisson and {S}tokes problems.
\newblock {\em Journal of Computational Physics 372\/} (2018), 972 -- 995.

\bibitem{Scovazzi2}
{\sc Main, A., and Scovazzi, G.}
\newblock The shifted boundary method for embedded domain computations. {P}art
  {II}: Linear advection–diffusion and incompressible {N}avier–{S}tokes
  equations.
\newblock {\em Journal of Computational Physics 372\/} (2018), 996 -- 1026.

\bibitem{QiShSh2018}
{\sc Qiu, W., Shen, J., and Shi, K.}
\newblock {HDG} method for linear elasticity with strong symmetric stresses.
\newblock {\em Mathematics of Computations 87\/} (2018), 69--93.

\bibitem{Patrick}
{\sc Qiu, W., Solano, M., and Vega, P.}
\newblock A high order {HDG} method for curved-interface problems via
  approximations from straight triangulations.
\newblock {\em Journal of Scientific Computing 69}, 3 (2016), 1384--1407.

\bibitem{Lew2}
{\sc Rangarajan, R., Lew, A., and Buscaglia, G.~C.}
\newblock A discontinuous-{G}alerkin-based immersed boundary method with
  non-homogeneous boundary conditions and its application to elasticity.
\newblock {\em Computer Methods in Applied Mechanics and Engineering 198}, 17
  (2009), 1513--1534.

\bibitem{10.1093/imanum/drab059}
{\sc Solano, M., Terrana, S., Nguyen, N.-C., and Peraire, J.}
\newblock {An HDG method for dissimilar meshes}.
\newblock {\em IMA Journal of Numerical Analysis\/} (08 2021).
\newblock drab059.

\bibitem{Vargas}
{\sc Solano, M., and Vargas, F.}
\newblock A high order {HDG} method for {S}tokes flow in curved domains.
\newblock {\em Journal of Scientific Computing 79\/} (2019), 11505–1533.

\bibitem{Vargas2}
{\sc Solano, M., and {Vargas M.}, F.}
\newblock An unfitted {HDG} method for {O}seen equations.
\newblock {\em Journal of Computational and Applied Mathematics 399\/} (2022),
  113721.

\end{thebibliography}

\end{document}